\documentclass[reqno,11pt]{amsart}
\usepackage{color,eucal,enumerate,mathrsfs}
\usepackage[normalem]{ulem}
\usepackage{amsmath,amssymb,epsfig,amsthm} 
\usepackage[latin1]{inputenc}
\usepackage{psfrag}
\usepackage{epstopdf}

\usepackage[shortlabels]{enumitem}

\usepackage[abs]{overpic}

\numberwithin{equation}{section}

\newtheorem{thm}{Theorem}[section]
\newtheorem{lem}[thm]{Lemma}
\newtheorem{prop}[thm]{Proposition}

\newtheorem{cor}[thm]{Corollary}
\newtheorem{defn}[thm]{Definition}
\numberwithin{equation}{section}

\newtheorem*{convention}{Convention}

\newtheorem*{namedtheorem}{\theoremname}
\newcommand{\theoremname}{testing}

\theoremstyle{remark}
\newtheorem{rem}[thm]{Remark}

\DeclareMathOperator{\diam}{diam}
\DeclareMathOperator{\dist}{dist}

\newcommand{\N}{\mathbb{N}}
\renewcommand{\S}{\mathbb{S}}
\newcommand{\D}{\mathbb{D}}
\newcommand{\R}{\mathbb{R}}
\newcommand{\Z}{\mathbb{Z}}

\def\az{\alpha}

\def\dist{{\mathop\mathrm{\,dist\,}}}

\def\ez{\epsilon}

\def\ls{\lesssim}
\def\gs{\gtrsim}

\def\bint{{\ifinner\rlap{\bf\kern.35em--}
\int\else\rlap{\bf\kern.45em--}\int\fi}\ignorespaces}

\def\bbint{{\ifinner\rlap{\bf\kern.35em--}
\hspace{0.078cm}\int\else\rlap{\bf\kern.45em--}\int\fi}\ignorespaces}

\def\diam{{\mathop\mathrm{\,diam\,}}}

\def\bint{{\ifinner\rlap{\bf\kern.35em--}
\int\else\rlap{\bf\kern.45em--}\int\fi}\ignorespaces}

\newcommand{\C}{\mathbb{C}}
\newcommand{\bS}{\mathbb{S}}
\newcommand{\bH}{\mathbb{H}}

\newcommand{\capa}{\mathrm{Cap}}

\newcommand{\id}{\mathrm{id}}

\newcommand{\dw}{\;\mathrm{d}w}
\newcommand{\ds}{\;\mathrm{d}s}

\newcommand{\sR}{\mathscr{R}}
\newcommand{\sI}{\mathscr{I}}
\newcommand{\sP}{\mathscr P}
\newcommand{\sQ}{\mathscr Q}
\newcommand{\sC}{\mathscr C}

\newcommand{\refi}{\mathrm{ref}}
\newcommand{\horz}{\mathrm{HE}}
\newcommand{\geom}{\mathrm{geom}}
\newcommand{\comb}{\mathrm{comb}}

\newcommand{\Diam}{\mathrm{Diam}}
\newcommand{\subharm}{\mathscr{H}}

\newcommand{\Jordan}{\mathscr{J}}

\usepackage{graphicx}
\usepackage{import}
\usepackage{xifthen}
\usepackage{pdfpages}
\usepackage{transparent}

\begin{document}

\title[Ahlfors reflection theorem]{Ahlfors reflection theorem for $p$-morphisms}

\author{Pekka Koskela}
\author{Pekka Pankka}
\author{Yi Ru-Ya Zhang}

\address{Department of Mathematics and Statistics, P.O. Box 35 (MaD) \\
         FI-40014 University of Jyv\"as\-kyl\"a, Finland}
\email{pekka.j.koskela@jyu.fi} 

\address{Department of Mathematics and Statistics, P.O. Box 64 (Pietari Kalmin katu 5), FI-00014 University of Helsinki, Finland}
\email{pekka.pankka@helsinki.fi}

\address{ETH Zurich, Department of Mathematics, Ramistrasse 101, 8092 Zurich, Switzerland}
\email{yizhang3@ethz.ch}

\thanks{All authors were partially supported by the Academy of Finland. 
P.P. was partially supported by Academy of Finland grant \#297258 and the Simons semester \emph{Geometry and analysis in function and mapping theory on Euclidean and metric measure spaces} at IMPAN, Warsaw.
Y. R.-Y. Zhang was partially supported by European Research Council under the Grant Agreement No.
721675 ``Regularity and Stability in Partial Differential Equations (RSPDE)".
}
\date{\today}


\begin{abstract}
We prove an Ahlfors reflection theorem for $p$-reflections over Jordan curves bounding subhyperbolic domains in $\widehat \C$. 
\end{abstract}


\maketitle

\setcounter{tocdepth}{1}
\tableofcontents

\section{Introduction}

Given a Jordan curve $\Gamma$ on the Riemann sphere $\widehat \C$, i.e.~the complex plane plus a point $\infty$, there exists a homeomorphic reflection $f\colon \widehat \C \to \widehat \C$ with respect to $\Gamma$: \emph{$f$ maps the two Jordan domains $\Omega$ and $\tilde{\Omega}$ associated to $\Gamma$ homeomorphically onto each other and keeps $\Gamma$ pointwise fixed}. By a classical theorem of Ahlfors \cite{A1963}, \cite[Lemma 3, Page 48]{A2006}, this homeomorphism can be chosen to be bilipschitz precisely when $\Gamma$ satisfies a so-called three point condition: \emph{There exists a constant $C\ge 1$ having the property that, given any pair of points $z_1,\,z_2\in \Gamma$ and a point $z_3$ on the arc of smaller diameter between $z_1$ and $z_2$,  
\[
|z_{1}-z_{3}|+|z_{2}-z_{3}|\leq C|z_{1}-z_{2}|.
\]
}
This motivates the following natural problem: 

\begin{quote}
\emph{Determine the classes} $\mathcal F$ \emph{of homeomorphisms and the corresponding classes} $\mathbf \Gamma$ \emph{of Jordan curves so that a given  Jordan curve} $\Gamma$ \emph{admits a reflection} $f\in \mathcal F$ \emph{precisely when} $\Gamma \in \mathbf {\Gamma}$.
\end{quote}

Another result by Ahlfors, also from \cite{A1963}, states that a Jordan curve $\Gamma$ admits a quasiconformal reflection precisely when $\Gamma$ is a quasicircle. Here quasiconformality of our homeomorphism $f$ requires that $f\in W^{1,1}(\widehat \C,\, \widehat \C)$ satisfies almost everywhere the distortion inequality
\begin{equation}
\label{peruskvasi}
|Df(z)|^2\le K |J_f(z)|
\end{equation}
for some constant $K\ge 1$, where $|Df(z)|$ is the operator norm of the differential matrix $Df(z)$ and $J_f(z)$ is the determinant of $Df(z)$. A quasicircle is, by definition, the image of the circle $S^1$ under a quasiconformal mapping from $\widehat \C$ onto $\widehat \C.$  As shown by Ahlfors, a Jordan curve $\Gamma$ is a quasicircle exactly when $\Gamma$ satisfies a three point condition.
Hence both the class of bilipschitz homeomorphisms and the class of quasiconformal homeomorphisms give a class $\mathcal F$ corresponding precisely to the class $\mathbf \Gamma$ of Jordan curves that satisfy a three point condition. 
This seems to be all that is known up to now regarding the above problem.


Towards our results regarding this classification problem, let us recall an equivalent characterization for the three point condition (see \cite{GO1980}): \emph{a Jordan curve $\Gamma$ satifies the Ahlfors three point condition if and only if there is a constant 
$C>0$ so that for any pair of points $z_1,z_2\in G$ there is a curve 
$\gamma\subset G$ 
joining $z_1$ and $z_2$ so that
\begin{equation}\label{GO}
\int_{\gamma}\frac{\ds(z)}{\dist(z,\,\partial G)} 
\le C \log\left(1+\frac{|z_1-z_2|}{\min\{\dist(z_1,\,\partial G),\,\dist(z_2,\,\partial G)\}}\right);
\end{equation}
where $G$ is any of the two complementary domains $\Omega$ and $\tilde\Omega$ of the Jordan curve $\Gamma$}. 

Recall that the quasihyperbolic distance $d_{qh}(z_1,\,z_2)$ of a pair of points $z_1,z_2$ in $G$ is defined by
\[
d_{qh}(z_1,\,z_2)
=\inf_{\gamma} \int_{\az}\frac{\ds(z)}{\dist(z,\,\partial G)},
\]
where the infimum is taken over all (rectifiable) curves $\az\subset G$ joining $z_1$ and $z_2$. 
Hence \eqref{GO} can be reformulated as
\begin{equation}\label{gehring osgood}
d_{qh}(z_1,\,z_2)\le C \log\left(1+\frac{|z_1-z_2|}{\min\{\dist(z_1,\,\partial G),\,\dist(z_2,\,\partial G)\}}\right). 
\end{equation}
In conclusion, a Jordan curve $\Gamma$ admits a bilipschitz or a quasiconformal  reflection exactly when $\Omega$ (or equivalently $\tilde \Omega $) satisfies \eqref{gehring osgood}. 

The quasihyperbolic distance was introduced by Gehring and Osgood in \cite{GO1980} and it has turned out to be a very useful concept; see e.g.~\cite{BHK2001,GHM1989,GP1976}. Notice that the upper bound in \eqref{gehring osgood} is, modulo the constant $C,$  the optimal one even when $G$ is a disk.  Thus, \eqref{gehring osgood} can be viewed to be a reverse inequality.

In \cite{GM1985}, Gehring and Martio introduced a related distance $d_\alpha$ by replacing the density $z\mapsto d(z,\partial G)^{-1}$
with $z\mapsto d(z,\partial G)^{\alpha-1}$ for a fixed $0<\alpha<1$, that is, 
\[
d_{\az}(z_1,\,z_2)=\inf_{\gamma} \int_{\az}\dist(z,\,\partial \Omega)^{ \az-1 }\, ds(z)
\]
for $z_1,z_2\in G$.
We call $d_{\alpha}$ the \emph{$\alpha$-subhyperbolic distance}. In spirit of the Gehring--Osgood condition  \eqref{gehring osgood}, we say that a domain $G$ is \emph{$\alpha$-subhyperbolic} if there is a constant $C\ge 1$ for which
\begin{equation}
\label{eq:p-subhyperbolic}
d_{\az}(z_1,\,z_2)\ \le  C  |z_1-z_2|^{\az},
\end{equation}
for all $z_1,\,z_2\in G$.


\bigskip
In this article, we characterize the class of Jordan curves bounding $(2-p)$-subhyperbolic domains as the class of Jordan curves admitting a $p$-reflection for $1<p<2$ -- by the aforemententioned Gehring--Osgood theorem, the characterization in the limiting case $p=2$ is given by Ahlfors' quasiconformal reflection theorem for quasicircles. For the statements, we first introduce the class of $p$-reflections.
\bigskip


Given $1<p<\infty$ and domains $G,G'\subset \widehat \C$, we say that a homeomorphism $f:G\to G'$ is a \emph{$p$-morphism} if $f\in
W^{1,1}(G,G')$ and there is a constant $K\ge 1$ so that the distortion inequality
\begin{equation}\label{pkvasi}
|Df(z)|^p\le K|J_f(z)|
\end{equation}
holds for almost every $z\in G.$ Thus $f$ is a 2-morphism if and only if $f$ is
quasiconformal. The class of mappings satisfying \eqref{pkvasi} was introduced independently by Gehring \cite{G1971} and Maz'ya \cite{M1969}. 

For a Jordan curve $\Gamma\subset \widehat \C$ with complementary Jordan domains $\Omega$ and $\widetilde\Omega$ in $\widehat \C$, we say that a homeomorphism $f:\widehat \C\to \widehat \C$ is a \emph{$p$-reflection from $\Omega$ to $\widetilde\Omega$} if $f$ restricted to $\Gamma$ is the identity mapping and $f:\Omega\to \widetilde\Omega$ is a $p$-morphism. We then say that \emph{$\Gamma$ admits a $p$-reflection from $\Omega$ to $\widetilde\Omega$}.

Our main result is the following existence theorem for $p$-morphisms.

\begin{thm} \label{thm:main1} 
Let $1<p< 2$, $\Gamma\subset \widehat \C$ be a Jordan curve and, let 
$\Omega$ and $\widetilde\Omega$ be the complementary Jordan domains in $\widehat \C$ having $\Gamma$ as their common boundary. Suppose that $\widetilde \Omega$ is $(2-p)$-subhyperbolic. Then there exists a $p$-reflection $\widetilde \Omega \to \Omega$, quantitatively, which is locally bilipschitz.
\end{thm}

The existence of a $p$-reflection is a characterization of $(2-p)$-subhyperbolic domains and we obtain the following solution to the reflection problem for $1<p<2$.

\begin{thm}
\label{main}
Let $1<p< 2$, $\Gamma\subset \widehat \C$ be a Jordan curve and, let 
$\Omega$ and $\widetilde\Omega$ be the complementary Jordan domains in $\widehat \C$ having $\Gamma$ as their common boundary.
Then the following conditions are equivalent:
\begin{enumerate}
\item  $\Gamma$ admits a $p$-reflection from $\widetilde \Omega$ to $\Omega$.
\item $\widetilde \Omega$ is $\alpha$-subhyperbolic with $\alpha=2-p.$
 \item $\Gamma$ admits a $q$-reflection from $\Omega$ to $\widetilde\Omega$ with $q=p/(p-1)$.\label{item:main3}
 \end{enumerate}
\end{thm} 

As a corollary we also obtain an analog of the result of Martio and Sarvas \cite{MS1979} on quasidisks.


\begin{cor}\label{p quasidisk}
Let $\Omega\subset \widehat\C$ be a Jordan domain whose complementary domain $\tilde \Omega = \widehat\C \setminus \overline{\Omega}$ satisfies the $(2-p)$-subhyperbolic condition for a fixed  $1<p<2$. Then there is a $p$-morphism $f\colon\hat{\mathbb C}\to \hat{\mathbb C}$ such that $f(\mathbb D)=\Omega.$ 
\end{cor}

Regarding the statement of Theorem \ref{main}, recall that the inverse of a bilipschitz homeomorphism is  bilipschitz and that the inverse of a quasiconformal homeomorphism is quasiconformal.  This kind of symmetry does not hold for $p$-morphisms: \emph{the inverse of a $p$-morphism is a $q$-morphism with $q=p/(p-1)$}. This explains why we need to give our geometric criteria on one of the complementary domains instead of giving a condition on the Jordan curve in question. The key explanation behind this duality in the reflection is the duality of capacities in the plane: \emph{given a Jordan domain $\Omega\subset \hat{\mathbb C}$  whose boundary is partitioned into four arcs $\gamma_1$, $\gamma_2$, $\gamma_3$ and $\gamma_4$ counterclockwise,  we have 
\begin{equation}\label{duality}
\left({\rm Cap}_p(\gamma_1,\,\gamma_3;\,\Omega)\right)^{\frac 1 p}\left({\rm Cap}_q(\gamma_2,\,\gamma_4;\,\Omega)\right)^{\frac 1 q}= 1
\end{equation}
for $1<p<\infty$ and $q=\frac p {p-1}$}; see \cite{Z2020} for more details.

\subsection*{Comparison to quasiconformal reflections}
In Ahlfors' reflection theorems, 
a quasiconformal reflection promotes to a bilipschitz reflection.
Obviously, it cannot be the case that a $p$-reflection could promote, say, to a quasiconformal reflection. However, it follows from Lemma \ref{self improve} and Theorem \ref{main} that the existence of a $p$-reflection is an open ended condition, that is, \emph{the existence of a $p$-reflection for some $1<p<2$ implies the existence of an $s$-reflection for any $1<s<p+\epsilon<2,$ where
$\epsilon>0$ depends only on $p$ and the distortion coefficient $K$ in \eqref{pkvasi}}. For $p>2,$ the respective range is $2<p-\epsilon<s<\infty.$ In particular, bilipschitz reflections are $p$-reflections for all $p$ and a quasiconformal reflection promotes to an $s$-reflection for any $s$.

The proof of the Theorem \ref{thm:main1} reveals that the $p$-reflection in Theorem \ref{thm:main1} can be chosen to have additional infinitesimal properies similar to quasiconformal mappings. Recall that, by Hadamard's inequality, a $K$-quasiconformal mapping $g \colon G\to G'$ between domains $G$ and $G'$ in $\C$ satisfies the double inequality
\[
|J_g(z)| \le |Dg(z)|^2 \le K |J_g(z)|
\]
for almost every $z\in G$. 

We show that we can choose the reflection $f \colon \widetilde \Omega \to \Omega$ to satisty 
the additional constraint
\begin{equation}\label{tuplak}\frac 1 K |J_f(z)|\le |Df(z)|^p \le K |J_f(z)|
\end{equation}
for almost every $z\in \widetilde \Omega$; see Remarks~\ref{rem tukia} and ~\ref{rem final} for details.
Note that, since $p<2,$ the Jacobian determinant of $f$ is not a priori controlled by $|Df|^p.$
An analogous estimate can then be also required for $f^{-1},$ but with the exponent $q=p/(p-1).$ 


A specific interest to this double inequality comes from the observation that, in the case of quasiconformal reflections, the Hausdorff dimension of the Jordan curve $\Gamma$ is controlled by the distortion $K$ of the reflection; see e.g.~\cite[Chapter 13.3]{AIM2009} for the sharp estimate.
In the setting of $p$-reflections, the Hausdorff dimension of $\Gamma$ is necessarily strictly less than $2$. Therefore it would be natural to ask for estimates in terms of $p$ and the distortion coefficient $K$ in \eqref{tuplak}.

\subsection*{Connection to generalized Cauchly--Riemann equations}
Let us describe a technique for producing homeomorphisms that satisfy \eqref{tuplak} with $K=1$ and its dual estimate with exponent $q=p/(p-1)$ in regions separated by a smooth interface.

Consider the generalized Cauchy-Riemann equation
\begin{equation}\label{general CR}
\left\{ \begin{array}{l}
u_x=|Dv|^{p-2}v_y\\
u_y=-|Dv|^{p-2}v_x
\end{array} \right.
\end{equation}
which has been applied in the theory of partial differential equation and functional analysis; see \cite{BI1987, IM1989, L1947, L1948, lewis1987}, together with the monographs \cite{AIM2009, L1962} and the reference therein. This equation has a $C^2$-solution precisely when $v$ is $p$-harmonic. 

Let now $G\subset \C$ be a domain and fix $1<p<2,$ a $p$-harmonic function $v\colon G\to \R$ a solution $u \colon G\to \R$ to the system \eqref{general CR} and define 
$$w=u+iv \colon G\to \C. $$ Suppose also $|\nabla v(z_0)|=1$ at a point $z_0\in G$. By results due to J. Lewis \cite{lewis1987}, the mapping $w$ is real analytic and homeomorphic in a neighborhood of $z_0.$ 
Additionally, by \eqref{general CR} we have that $\nabla u,\nabla v$ are perpendicular and that $|\nabla u|=|\nabla v|^{p-1}$ and $|J_w|=|\nabla v|^p. $
Thus
$$|Dw|^q = |J_w| \text{ when } |Dw|\le 1,$$
and 
$$|Dw|^p=|J_w| \text{ when } |Dw|> 1,$$
where $q=\frac p {p-1}$. In particular, in a neighborhood of $z_0$, the smooth set $\{|Dw|=1\}$ separates these two regions.

The above discussion motivates one to consider the  class of homeomorphisms $f\in W^{1,\,1}(G,G')$ between domains $G$ and $G'$ in $\C$ for which there exists a constant $K\ge 1$ satisfying
\begin{equation}\label{new}
\frac 1 K\le \min \left\{\frac{|Df|^{p} }{|J_f|},  \frac{|Df|^{q}} {|J_f|}\right\}\le K, 
\end{equation}
where $1<p<\infty$ and $q=\frac p {p-1}$. 
The reflection that we construct for Theorem~\ref{main} falls into this new category (see Remarks~\ref{rem tukia} and ~\ref{rem final}). We would like to know if one could establish Hausdorff measure estimates in terms of the distortion constant $K$ for the interface in \eqref{new}. 

\subsection*{Idea of the proof}
Let us briefly explain the main difficulty in proving Theorem \ref{thm:main1}.
The counterpart of Theorem \ref{thm:main1} for $p=2$ is essentially the Ahlfors
reflection theorem. The desired reflection in this case is obtained as follows. 
Applying a rotation, we may assume that 
$\infty \in \Gamma$. 
Pick a conformal map $\varphi: \bH^+\to \Omega$ and a conformal
map $\psi: \bH^+\to  \widetilde\Omega$, where $\bH^+$ is the upper half 
plane. By the Caratheodory--Osgood theorem both these maps extend homeomorphically to the boundary. Next, condition \eqref{GO} allows (with work) one to deduce 
that $h=\psi^{-1}\circ \varphi:\R\to \R$ is quasisymmetric. 
Then the Beurling-Ahlfors \cite{BA1956}
extension gives us a quasiconformal mapping $\widehat h: \mathbb H^+\to 
\mathbb H^+$ with the boundary value $h.$ The desired quasiconformal map $f$ 
is given by the composition $z\mapsto \psi(\widehat h(\overline \varphi^{-1}(z)))$. For $p\neq 2$, this approach is doomed: \emph{quasiconformal or conformal maps are not in general 
$p$-morphisms and an inverse of a $p$-morphism may fail to be a 
$p$-morphism}. Hence we cannot reduce the question to a half plane and need to
construct the reflection by hand. 

The idea behind our construction is to first give a reflection in a 
neighborhood of $\Gamma$ using the hyperbolic rays in the complementary domains $\widetilde \Omega$ and $\Omega$. Since this local reflection near $\Gamma$ is merely topological, it maps Whitney squares and Whitney-type sets to sets that can be far from being squarelike. 
We next decompose them to better topological rectangles
and decompose our Whitney-type sets in a combinatorially matching manner. 
After this, we modify the boundaries of the topological rectangles so as to 
become Lipschitz and construct homeomorphisms between the respective boundaries
and eventually fill in using ideas from Tukia--V\"ais\"al\"a \cite{TV1982}. As one could 
expect, our approach also gives a new proof of the quasiconformal reflection
result but the details will not be recorded here.

\subsection*{Application: Sobolev extension via composition}
One of the direct applications of Theorem \ref{thm:main1} is that it induces an extension operator for Sobolev extension domains.



Let $u\in C^1(G')$ and let $f:G\to G'$ be a $p$-morphism. If $f$ is differentiable at a point $z\in G$, then
$$
|\nabla (u\circ f)(z)|^p\le K|\nabla u(f(z))|^p
|Df(z)|^p\le K|\nabla u(f(z)|^ p|J_f(z)|.
$$
This suggests that composition under $f$ should preserve $L^p$-energy and that $f$ should generate  a morphism for the homogeneous
Sobolev spaces $L^{1,p}$ via composition. 
This is indeed the case; see e.g.\  \cite{VG1975,GG1994,VU1998}. 
Hence Theorem \ref{main} together with results from \cite{KRZ2015, S2010}  
shows that the extension
operator $W^{1,p}(\Omega) \to W^{1,p}(\C)$ for a Jordan Sobolev extension domain $\Omega\subset \C$
can be required to be generated by a composition operator. 
This explains the duality results in \cite{KRZ2015}.
The following corollary was previously only known for the case $p=2,$ 
see \cite{GV1980, GV1981, VGL1979}.

\begin{cor}\label{dual extension}
Let $1<p<\infty$ and let $\Gamma\subset \widehat \C$ be a Jordan curve
with complementary domains $\Omega$ and $\widetilde\Omega$ in $\widehat \C$. 
If $\Omega$ is a $W^{1,p}$-extension 
domain, then $\Omega$ admits the corresponding extension via an operator 
generated by a composition whose inverse generates the analogous extension from 
$W^{1,q}(\widetilde \Omega )$ when $q=p/(p-1).$
\end{cor}



\bigskip

Let us close this introduction by returning to the problem of determining the classes of homeomorphisms $\mathcal F$ and the corresponding classes $\mathbf \Gamma$ so that a given Jordan curve admits a reflection $f\in \mathcal F$ precisely when $\Gamma\in \mathbf \Gamma.$
Previously only the case of Jordan curves satisfying the three point condition was understood. We now understand also the case where one of the two associated Jordan domains is $p$-hyperbolic. The methods of this paper suggest that a solution is also feasible when one of the domains is quasiconvex. Moreover, it seems to us that, if one were to start from a class $\mathcal F,$ one would need 
an invariant, for example a class of function spaces whose seminorms behave nicely under composition with elements of $\mathcal F$. In the case of $p$-morphisms this is the $L^p$-energy, which for $p=2$ is the usual Dirichlet energy. Many interesting questions remain to be addressed.

\section*{Notation and terminology}
\label{sec:notation}

\subsection{Metrics}

In the complex plane $\C$ we use the standard Euclidean metric $|\cdot - \cdot|$ and in the Riemann sphere $\widehat \C$ the chordal metric, that is, metric of $\bS^2$ in $\R^3$. We denote the chordal metric also $|\cdot-\cdot|$. 

Given a domain $\Omega$ in $\C$, we denote $d_\Omega$ the \emph{inner metric in $\Omega$}, that is, 
\[
d_\Omega(x,y) = \inf_\gamma \ell(\gamma),
\]
where $\gamma \colon [0,1]\to \Omega$ is a curve from $\gamma(0)=x$ to $\gamma(1)=y$, and $\ell(\gamma)$ is the length of $\gamma$ in the Euclidean metric. In what follows, we use the (inner) diameter $\diam_\Omega$ and (inner) distance $\dist_\Omega$ given by the inner metric $d_\Omega$ of $\Omega$. Note that, these notions extend for sets in the Euclidean closure $\overline{\Omega}$ of $\Omega$.

\subsection{Jordan curves and domains}

Recall that the image $\Gamma$ of an embedding $\S^1 \to \widehat \C$ is called a \emph{Jordan curve} and, by the Jordan curve theorem, the set $\widehat \C\setminus \Gamma$ has exactly two components, both homeomorphic to the open unit disk $\D$ and having $\Gamma$ as boundary. These components of $\widehat \C\setminus \Gamma$ are called \emph{Jordan domains}. 

By the Riemann mapping theorem, for each Jordan domain $\Omega$ in $\widehat \C$, there exists a conformal map $\D \to \Omega$. Moreover, for a Jordan domain $\Omega$, every conformal homeomorphism $\varphi\colon \D \to \Omega$ has a homeomorphic extension $\overline{\D} \to \overline{\Omega}$ by the Caratheodory--Osgood theorem, see e.g.\;\cite{P1991}. 

We fix a family $\{\gamma_z\colon z\in \partial \Omega\}$ of hyperbolic rays in $\Omega$ with respect to a base point $z_0\in \Omega$ as follows. Let $\varphi \colon \D\to \Omega$ be a conformal map for which $\varphi(0)=z_0$. Then, for each $z\in \partial \Omega$, we let $\gamma_z \colon [0,1]\to \overline{\Omega}$ to be the ray defined by $t \mapsto \varphi(\gamma_{\varphi^{-1}(z)}(t))$, where $\gamma_{\varphi^{-1}(z)}$ is the hyperbolic ray in $\D$ from the origin to $\gamma_{\varphi^{-1}(z)}\in \partial \mathbb D$. Note that the family $\{\gamma_z\colon z\in \partial \Omega\}$ is independent of the choice of the map $\varphi$. Here and in what follows, we suppress the base point $z_0$ in the notation.

Let $\Omega$ be a Jordan domain in $\C$, $\Gamma=\partial \Omega$, and $z_\Omega$ its base point. The \emph{shadow projection $S_{\Gamma}\colon \Omega \setminus \{z_\Omega\} \to \Gamma$ with respect to $\Gamma$} is the projection along hyperbolic rays onto $\Gamma$, that is, for $w\in \Omega$, we have that $z = S_\Gamma(w)$ is the unique point in $\Gamma$ for which $w = \gamma_z(t)$ for some $t\in (0,1)$. Note that, since hyperbolic geodesics are independent of the choice of the parametrization $\varphi$, also so is the shadow projection. Since $\varphi$ extends to the boundary as a homeomorphism, the shadow projection $S_{\Gamma}\colon \Omega\setminus \{z_\Omega\} \to \Gamma$ is continuous.

For each $w\in \Omega$, we denote $\Gamma^\perp_\Omega(w)$ the hyperbolic geodesic joining $w$ and its shadow $S_{\Gamma}(w)$; note that $\Gamma^\perp_\Omega(w)$ is also a subarc of the unique hyperbolic ray from $z_0$ to $S_\Gamma(w)$. Given a connected set $Q \subset \Omega$ and $x\in S_\Gamma(Q)$, we denote $\Gamma^\perp_\Omega(x,\,Q)$ the segment in $\Gamma^\perp_\Omega(x)$ between $x$ and $Q$. Finally, for $z\in \Gamma$, we denote by $\Gamma^\perp_\Omega(z)$ the hyperbolic ray in $\Omega$ ending at $z$.  

Given a Jordan curve $\Gamma \subset \C$, we denote $\tilde \Omega_\Gamma$ the component of $\widehat\C\setminus \Gamma$, which is the bounded component in $\C$, and $\Omega_\Gamma$ the component of $\widehat \C \setminus \Gamma$ containing $\infty$. For both domains $\tilde \Omega_\Gamma$ and $\Omega_\Gamma$, we use the notation $S_\Gamma$ for the shadow projections $\tilde \Omega_\Gamma\setminus \{z_{\tilde \Omega_\Gamma}\} \to \Gamma$ and $\Omega_\Gamma\setminus \{z_{\Omega_\Gamma}\} \to \Gamma$; here $z_{\Omega_\Gamma} = \infty$. However, to simplify notation, we denote $\tilde \Gamma^\perp(\tilde w) = \Gamma^\perp_{\tilde \Omega_\Gamma}(\tilde w)$ and $\Gamma^\perp(w) = \Gamma^\perp_{\Omega_\Gamma}(w)$ for $\tilde w\in \tilde\Omega_\Gamma$ and $w\in \Omega_\Gamma$, respectively. We extend the related notations similarly.

\section{Preliminaries}
\label{sec:preliminaries}

For $p\in (1,2)$ and $C\ge 1$, we define $\subharm(p,C)$ to the collection of all $p$-subhyperbolic domains with constant $C>0$ bounded by Jordan curves in $\C$, that is, $G\in \subharm(p,C)$ if $\partial G \subset \C$ is a Jordan curve and $G$ satisfies \eqref{eq:p-subhyperbolic} with constant $C$. We also denote $\Jordan(p,C)$ the collection of all Jordan curves $\Gamma$ in $\C$ for which the bounded complementary domain $\tilde \Omega_\Gamma$ belongs to $\subharm(p,C)$

\subsection{Conformal geometry}

Recall that for points $z_1$ and $z_2$ in $\D$, their hyperbolic distance is
\[
d_h(z_1,\,z_2)=\inf_{\az}\int_{\az} \frac {2}{1-|z|^2}\,|dz|,
\]
where the infimum is over all rectifiable curves $\az$ joining $z_1$ to $z_2$ in $\mathbb D$. The hyperbolic geodesics in $\D$ are arcs of (generalized) circles that intersect the unit circle orthogonally. 

Recall that the hyperbolic metric is preserved under conformal maps, that is, for a conformal map $\varphi \colon \D \to \D$ and a pair of points $x$ and $y$ in $\D$ we have
\[
d_h(\varphi(x),\varphi(y)) = d_h(x,y).
\]
The hyperbolic geodesics in $\mathbb D$ are arcs of (generalized) circles that intersect the unit circle orthogonally. In particular, the segment $\gamma_z \colon [0,1]\to \overline{D}$, $t\mapsto tz$, is a hyperbolic ray (in $\D$) from the origin to $z\in \partial \D$.

We begin this section by recording the classical Gehring--Hayman inequality.

\begin{lem}[\cite{GH1962}]\label{GH thm}
	Let $\varphi \colon \mathbb D \to \Omega$ be a conformal map, $x,\,y\in \mathbb D$ be two distinct points, $\gamma_{x,\,y}$ be the hyperbolic geodesic in $\mathbb D$ connecting $x$ and $y$, and  $\az_{x,\,y}$ be any curve connecting $x$ and $y$ in $\mathbb D.$ Then we have
	$$\ell(\varphi(\gamma_{x,\,y}))\le C \ell(\varphi(\az_{x,\,y})),$$
	where $C$ is an absolute constant. 
\end{lem}

Let us recall the definition of a Whitney-type set. 
\begin{defn}
\label{whitney-type set}
Let $\Omega$ be a Jordan domain in $\C$ and $\lambda \ge 1$. A bounded connected set $A \subset \Omega$ is a \emph{$\lambda$-Whitney-type set in $\Omega$} if 
\begin{enumerate}
\item $\frac {1}{\lambda} \diam(A)\le \dist(A,\,\partial\Omega)\le {\lambda } \diam(A)$, and
\item there exists a disk of radius $\frac {1}{\lambda}\diam(A)$ contained in $A$.
\end{enumerate}
\end{defn}

Whitney-type sets are preserved by conformal maps in the following sense; we refer to Gehring \cite[Theorem 11]{G1962} for a proof.

\begin{lem}\label{whitney preserving}
Suppose $\varphi\colon \Omega \to \Omega'$ is conformal, where $\Omega,\Omega'\subset \C$ are domains conformally equivalent to the unit disk or its complementary domain and $A\subset \Omega$ is a $\lambda$-Whitney-type set. Then $\varphi(A)\subset \Omega'$ is a $\lambda'$-Whitney-type set with $ \lambda'=\lambda'(\lambda)$.  
\end{lem}

As a corollary, we have the following bilipschitz property.

\begin{lem}
	\label{lemma:biLip}
	Let $\varphi\colon \Omega \to \Omega'$ be a conformal map, where $\Omega,\Omega'\subset \mathbb R^2$ are domains conformally equivalent to $\D$, and let $A\subset \Omega$ be a $\lambda$-Whitney-type set. Then the restriction of the conformal map $\varphi|_{A} \colon A\to \varphi(A)$ is an $L$-bilipschitz map up to a dilation factor with $L=L(\lambda)$. 
\end{lem}

\subsection{Subhyperbolic domains}

 Domains satisfying \eqref{eq:p-subhyperbolic} are quasiconvex, quantitatively. Recall that a (path-connected) set $A\subset \widehat\C$ is said to be {\it $C$-quasiconvex for $C\ge 1$} if for every two points $z_1,\,z_2\in A$ there exists a curve $\gamma\subset A$ joining $z_1$ and $z_2$ such that 
 \begin{equation}
 \label{eq:qconvex}
 \ell(\gamma)\le C|z_1-z_2|.
 \end{equation}
 We refer to \cite{la1985} and to the proof of \cite[Theorem 2.15]{GM1985}; see also Lemma~\ref{GH thm}.

 \begin{lem}
 	\label{lemma:qconvex}
 	Let $\Gamma\in \Jordan(p,C_0)$. Then $\tilde \Omega_{\Gamma}$ is $C_{\mathrm{qc}}$-quasiconvex with $C_{\mathrm{qc}}=C_{\mathrm{qc}}(p,C_0)$. More precisely, there exists $C_{\mathrm{qch}}=C_{\mathrm{qch}}(p,C_0)$ having the property that, for each $z_1$ and $z_2$ in $\tilde \Omega_\Gamma$, there exist a hyperbolic geodesic $\gamma$ joining $z_1$ and $z_2$ in $\tilde \Omega_{\Gamma}$ for which
 	\begin{equation}\label{eq:h-qconvex}
 	\ell(\gamma) \le C_{\mathrm{qch}} |z_1-z_2|.
 	\end{equation}
 	
 \end{lem}

The subhyperbolic domains have the following self-improvement property.

\begin{lem}[{\cite[Theorem 2.6]{S2010}}]
\label{self improve}
Let $1<p<2$ and $C\ge 0$. Then there exists $\ez=\ez(p,C)>0$ having the property that, for every $G\in \subharm(p,C)$ and $1<s<p+\ez$, we have $G\in \subharm(s,C')$, where $C'=C'(p,C)>0$.
\end{lem}

For the forthcoming discussion, we record also the fact that the subhyperbolic
length of a hyperbolic geodesic is controlled by a snowflaked Euclidean metric of the domain. 

\begin{lem}[{\cite[Lemma 4.2]{KRZ2015}}]\label{GM2}
Let $\Gamma \in \Jordan(p,C)$ and let $z\in \Gamma$. Then, for any pair of points $w,w'\in \tilde \Gamma^\perp(z)$, we have
\begin{align}\label{curve conditionGamma}
\int_{\gamma[w,w']}\, \dist(z,\,\Gamma)^{ 1-p}\,dz \le C(p,C)|w-w'|^{2-p},
\end{align}
where $\gamma[w,w']$ denotes the subarc of $\tilde \Gamma^\perp(z)$ joining $w$ and $w'$. 
\end{lem}

\begin{cor}
\label{cor:Pankka3}
Let $\Gamma \in \Jordan(p,C_0)$ be a Jordan curve in $\C$. Then, for any $\tilde x \in \tilde \Omega_{\Gamma}$, we have
\[
\int_{\tilde \Gamma^\perp(\tilde x)}\dist(w,\Gamma)^{1-p} \ds(w) \sim \ell(\tilde \Gamma^\perp(\tilde x))^{2-p},
\]
where the constants depend only on $p$ and $C_0$ 
\end{cor}
\begin{proof}
By Lemma~\ref{GM2} we have 
$$\int_{\tilde \Gamma^\perp(\tilde x)}\, \dist(z,\,\Gamma)^{ 1-p}\,dz \le C(p,C)|w-w'|^{2-p}\le C(p,C) \ell(\tilde \Gamma^\perp(\tilde x))^{2-p}.$$
This implies one direction. For the other direction, notice that for any $w\in \tilde \Gamma^\perp(\tilde x)$ we have
$$\ell(\tilde \Gamma^\perp(\tilde x))\ge \dist(w,\Gamma).$$
Since $p>1$, we obtain the other direction easily. 
\end{proof}

\subsection{John domains}

Let us recall the definition of John domains. 

\begin{defn}[John domain]\label{def:John}
	An open  subset $\Omega\subset \hat{\mathbb C}$ is called a John domain provided it satisfies 
	the 
	following condition:
	There exist a distinguished point $x_0 \in \Omega$ and a constant $J>0$ such that, for every 
	$x\in\Omega$, 
	there is a curve $\gamma: [0,\,l(\gamma)] \to \Omega$ parameterized by  arc
	length, such that 
	$\gamma(0)= x$, $\gamma(l(\gamma))= x_0$
	and
	\[
	\dist(\gamma(t),\, \R^2\setminus\Omega)\ge Jt. 
	\]
	Such a  curve $\gamma$ is called a $J$-John curve, $J$ is called a John 
	constant, and we refer to a John domain with a John constant $J$ by a $J$-John domain and to $x_0$ by a John center of $\Omega.$
\end{defn}

By definition the bounded complementary component $\tilde \Omega_\Gamma$ of a Jordan curve $\Gamma \in \Jordan(p,C)$ is subhyperbolic. The unbounded component $\Omega_\Gamma$ is a John domain. 

\begin{lem}[{\cite[Theorem 4.5]{NV1991}, \cite[Theorem 4,1]{GHM1989}}]
	\label{comp GM}
	Let $\Gamma \in \Jordan(p,C)$ be a Jordan curve in $\C$. Then the component $\Omega_\Gamma$ of $\widehat \C \setminus \Gamma$ is a $J$-John domain with $J=J(p,C)$, and the hyperbolic rays in the complementary domain are John curves. Especially, $\Gamma$ is of area zero.
\end{lem}

We state the diameter estimates for shadows in \cite[Lemma 4.3]{KRZ2015}. 
\begin{lem}
	\label{lemma:shadow1}
	Let $\Omega\subset \mathbb R^2$ be a Jordan domain with boundary $\Gamma$ and $\lambda \ge 1$. Then, for any $\lambda$-Whitney-type set $A\subset \Omega$ with some $\lambda \ge 1$, we have 
	\[
	\diam(A)\le C(\lambda) \diam(S_{\Gamma}(A)). 
	\]
\end{lem}

\begin{lem}
	\label{lemma:shadow2}
	Let $\Omega\subset \mathbb R^2$ be a Jordan $J$-John domain with boundary $\Gamma$ and $\lambda \ge 1$. Then, for any $\lambda$-Whitney-type set $A\subset \Omega$ with some $\lambda \ge 1$, we have 
	\[
	\diam_{\Omega}(S_{\Gamma}(A))\sim \diam(S_{\Gamma}(A))\sim \diam(A)
	\]
	with the constant depends only on $J$ and $\lambda$. 
\end{lem}

\begin{cor}
	\label{cor:shadow2}
	Let $\Gamma\in \Jordan(p,C_0)$. Then, for any $\lambda$-Whitney-type set $A\subset \Omega_\Gamma$ for $\lambda \ge 1$, we have 
	\[
	\diam(S_{\Gamma}(A))\sim \diam(A)
	\]
	with the constant depends only on $C_0$ and $\lambda$. 
\end{cor}

We also record the following useful lemma. Recall that  a homeomorphism $\varphi\colon \mathbb D\to \Omega $ is 
{\it{quasisymmetric with respect to the
		inner distance}} if there is a homeomorphism $\eta:[0,\infty)\to [0,\infty)$ so that
$$|z-x|\le t|y-x|\mbox{ implies }\dist_{\Omega}(\varphi(z),\varphi(x))\le 
\eta(t)\dist_{\Omega}(\varphi(y),\varphi(x))$$
for each triple $z,x,y$ of points in $\mathbb D.$
It is clear from the definition that the inverse of a quasisymmetric map is also quasisymmetric. 
Roughly speaking the homeomorphism $\varphi$ maps round objects to round objects (with respect to the inner distance). 

\begin{lem}[\cite{H1989}, Theorem 3.1]\label{quasisymmetry}
	Let $\Omega\subset \mathbb R^2$ be a simply connected domain, and 
	$\varphi\colon \mathbb D\to \Omega $ be a 
	conformal map. Then $\Omega$ is John  if and only if $\varphi$ 
	is quasisymmetric with 
	respect to the inner distance. This statement is quantitative in the sense
	that the John constant and the function $\eta$ in quasisymmetry depend
	only on each other and $\diam(\Omega)/\dist(\varphi(0),\partial\Omega).$
	Especially, if $\Omega$ is John with constant $J$ and
	$\varphi(0)=x_0,$ where $x_0$ is the distinguished point, then,
	for any disk $B\subset \mathbb D$, $f(B)$ is a John domain with 
	the John constant only depending on $J$.   
\end{lem}

A immediate corollary is the following.
\begin{cor}\label{points}
Let $\Gamma\in \Jordan(p,C_0)$ and $\varphi_{\Gamma}\colon \mathbb C\setminus D \to \mathbb C\setminus {\tilde \Omega}_{\Gamma}$ be an extended conformal map. 
Then for any subarc $\Gamma'\subset \Gamma$ with $z\in \Gamma'$, letting $z'\in \Gamma^\perp(z)$ so that for some $c\ge 1$, 
$$ \ell(\Gamma^\perp(z,\,z'))= c \diam(\Gamma'),$$
we have 
$$\frac 1 {\beta}\diam(\varphi_{\Gamma}^{-1}(\Gamma')) \le  \ell(\varphi_{\Gamma}^{-1}(\Gamma^\perp(z,\,z')))\le \beta \diam(\varphi_{\Gamma}^{-1}(\Gamma')),$$
where $\beta=\beta(C_0,\,c)$ is an increasing function of $c$.  
\end{cor}
\begin{proof}
Let $z''\in \Gamma'$ so that	
$$\dist_{\Omega_\Gamma}(z,\,z'')>\frac 1 2 \diam_{\Omega_\Gamma}(\Gamma')$$
given by the triangle inequality. Then we have
$$\dist_{\Omega_\Gamma}(z,\,z'')\sim \diam_{\Omega_\Gamma}(\Gamma')\sim \diam(\Gamma')$$
by Lemma~\ref{lemma:shadow2}.
	
	Recall that $\Omega_{\Gamma}$ is John. Since $\varphi_\Gamma$ is conformal and maps the exterior of the disk onto the exterior of a Jordan domain, 
$\varphi$ extends conformally to the point at infinity, mapping it to the point
at infinity. Hence, modulo two rotations of the Riemann sphere, we may identify
$\varphi_\Gamma$ with a conformal map from the unit disk onto a bounded John 
domain in $\mathbb C.$  
	Then by Lemma~\ref{quasisymmetry},  $\varphi$ is $\eta$-quasisymmetric with respect to the internal metric of $\Omega$, and we have
	$$     \frac{|\varphi_{\Gamma}^{-1}(z)-\varphi_{\Gamma}^{-1}(z')|}{|\varphi_{\Gamma}^{-1}(z)-\varphi_{\Gamma}^{-1}(z'')|}\le \eta^{-1}\left( \frac{\dist_{\Omega_\Gamma}(z,\,z')}{\dist_{\Omega_\Gamma}(z,\,z'')}\right),$$
which gives one direction of the corollary as we are in the exterior of the unit disk. The other direction follows similarly. 
\end{proof}

This corollary also implies the following lemma. 
\begin{lem}\label{lemma:Pankka2}
	Let $\Omega$ be a Jordan domain which is $J$-John, $\Gamma=\partial \Omega$, and let $R\subset \Omega$ be a compact connected set. Let $c\ge 1$ be a constant having the property that, for each $z\in S_{\Gamma}(R)$, we have
	\begin{equation}\label{assum 3}
	\ell(\Gamma^\perp(z,\,R))\le c \diam(S_{\Gamma}(R)) 
	\end{equation}
	for some constant $c > 0$. Then 
	\[
	\dist(R,\,\Gamma)\ls\diam(R),
	\]
	where the constant depends only on $J$ and $c$. 
\end{lem}
\begin{proof}
Let $\varphi_{\Gamma}\colon \mathbb C\setminus D \to \overline {\Omega}_{\Gamma}$ be the extended conformal map. 
By Corollary~\ref{points}, we have 
$$\ell(\varphi_{\Gamma}^{-1}(\Gamma^\perp(z,\,R)))\ls \diam(\varphi_{\Gamma}^{-1}(S_{\Gamma}(R)))$$
for any $z\in \in S_{\Gamma}(R)$. Then by the fact that $\varphi_{\Gamma}^{-1}(\Gamma^\perp(z,\,R))$ radial segments, we conclude that 
$$\dist(\varphi_{\Gamma}^{-1}(R),\,\varphi_{\Gamma}^{-1}(\Gamma))\ls\diam(\varphi_{\Gamma}^{-1}(R)). $$
Now by the fact that $\varphi_{\Gamma}$ is quasisymmetric again we obtain the desired result.
\end{proof}

\begin{cor}
	\label{cor:JJ-corollary}
	Let $\Gamma \in \Jordan(p,C_0)$ and $R \subset \Omega_\Gamma$ a compact connected set and 
	\[
	c = \sup_{x\in S_\Gamma(R)} \frac{\ell(\Gamma^\perp(x,R))}{\diam(S_\Gamma(R))}.
	\]
	Then 
	\[
	\dist(R,\,\Gamma)\ls\diam(R),
	\]
	where the constant depends only on $p$, $C_0$, and $c$. 
\end{cor}

The following lemma introduces a doubling property on the boundary of a Jordan John domain, i.e.~a domain which is both Jordan and John. We use it in the form of the ensuing corollary.

\begin{lem}[{\cite[Lemma 4.6]{KRZ2015}}]
\label{lemma:finite division}
Let $\Omega\subset \widehat{\C}$ be a Jordan John domain. Then for each $C>0$ there are at most $N=N(C,J)$ pairwise disjoint subarcs $\{\gamma_k\}_{k=1}^N$ of a curve $\Gamma\subset \partial \Omega$ satisfying
\[
\diam(\Gamma) \le C\diam(\gamma_k).
\]
\end{lem}

\begin{cor}
\label{cor:finite division}
Let $\Gamma \in \Jordan(p,C_0)$ be a Jordan curve in $\C$. Then, for each $C>0$, there are at most $N=N(C,p,C_0)$ pairwise disjoint subarcs $\{\gamma_k\}_{k=1}^N$ of an arc $\gamma \subset \Gamma$ satisfying
\[
\diam(\gamma) \le C\diam(\gamma_k).
\]
\end{cor}


\section{Construction of a $p$-reflection: the idea}

By the self-improving property of subhyperbolic domains (Lemma \ref{self improve}), 
we may reformulate Theorem \ref{thm:main1} as follows.

\begin{thm}\label{weak main thm}
Let $1 < r < p<2$, $\Gamma\subset \hat\C$ be a Jordan curve, and let $\Omega$ and $\widetilde\Omega$ be complementary components of $\Gamma$ in $\hat\C$. 
If $\widetilde \Omega$ is $p$-subhyperbolic then there exists an $r$-reflection from $\widetilde \Omega$ to $\Omega$, quantitatively.
\end{thm}

The statement is quantitative in the following sense: \emph{if $\tilde \Omega$ belongs to $\subharm(p,C)$, then the reflection $f\colon \widehat \C \to \widehat \C$ from $\widetilde \Omega$ to $\Omega$ satisfies \eqref{pkvasi} with constant $K=K(r,p,C)\ge 1$.}

The proof of Theorem~\ref{weak main thm} spans over the remaining sections and consists of two main stages. The first stage, which we divide into four steps, consists of finding natural partitions $\sQ_\Gamma^\refi$ and $\tilde \sQ_\Gamma^\refi$ of domains $\Omega_\Gamma$ and $\tilde \Omega_\Gamma$, respectively. These partitions consist of topological rectangles which are in a natural one-to-one correspondence in terms of their adjacency. In the second stage of the proof we use the geometric properties of the obtained rectangles to find $p$-morphisms between corresponding rectangles and then glue the individual maps together to a $p$-morphic reflection.

Since the existence of good partitions $\sQ_\Gamma$ and $\tilde \sQ_\Gamma$ takes the bulk of the proof, we describe here the steps we take to develop them and postpone the description of the strategy to create $p$-morphisms to Section \ref{sec:proof-end-game}.

In what follows, we formulate most of the statements in terms of the Jordan curve $\Gamma$. When working in the complex plane $\C$, the domains $\widetilde \Omega$ and $\Omega$ are understood so that $\widetilde \Omega$ is the bounded complementary component $\tilde \Omega_\Gamma$ of $\Gamma$ in $\C$.

\emph{Step 1:} We begin by fixing a neighborhood $\tilde A$ of $\Gamma$ in $\widetilde \Omega$ and a partition $\tilde \sQ_\Gamma$ of $\tilde A$ based on Whitney-type sets. 

\emph{Step 2a:}
In this step, we fix a topological reflection $h_\Gamma \colon \tilde A \to \Omega$ over $\Gamma$ to obtain a topologically equivalent partition $\sQ_\Gamma$ of a neighborhood $A=h_\Gamma(\tilde A)$ of $\Gamma$ in $\Omega$. The role of the embedding $h_\Gamma$ is  merely to induce a partition 
$\sQ_\Gamma=h_\Gamma(\tilde \sQ_\Gamma)$ of $A$. 

\emph{Step 2b:} In this step, we introduce the partition $\sQ_\Gamma$ of $A$ and also compare the metric properties of the elements of $\sQ_\Gamma$ with the corresponding properties of elements of $\tilde \sQ_\Gamma$. Whereas the initial partition $\tilde\sQ_\Gamma$ of the neighborhood of $\Gamma$ in $\tilde \Omega$ consists of Whitney-type set, its reflection $\sQ_\Gamma$ does not a priori have this property.

\emph{Step 3:} We use the metric information from the previous step to subdivide elements of $\sQ_\Gamma$ into topological rectangles whose diameters are comparable to their distances to $\Gamma$, quantitatively. This produces a new partition $\sQ^\refi_\Gamma$ of a neighborhood of $\Gamma$ in $\Omega$. 

\emph{Step 4:} We refine the partition $\tilde \sQ_\Gamma$ according to the adjacency structure of the partition $\sQ^\refi_\Gamma.$ 
Technically this is done by reflecting partitions induced by $\sQ_\Gamma^\refi$ on horizontal edges of rectangles of $\tilde \sQ_\Gamma$ to edges of $\tilde \sQ_\Gamma$. Then, we consider two reparametrizations of these edge partitions and use their interpolation to build the refinement $\tilde \sQ_\Gamma^\refi$. The partition $\tilde \sQ^\refi$ consists of rectangles, which are in natural $1$-to-$1$ correspondence to rectangles in $\sQ^\refi_\Gamma$.

\newcommand{\sW}{\mathscr W}

\emph{Step 5:} Since the rectangles in $\sQ_\Gamma^\refi$ are not a priori bilipschitz equivalent to Euclidean rectangles, we pass to a further partition 
$\sW_\Gamma^\refi$ whose elements have this property and are close to corresponding elements of $\sQ_\Gamma^\refi$.

After these preliminary steps, we prove Theorem \ref{weak main thm} in Section \ref{sec:proof-end-game}.

\begin{rem}
A small philosophical remark regarding the strategy of the construction of these partitions is in order. In what follows, after introducing a new partition, we consider metric properties of its elements. The recurring theme in these considerations is that we would like to have, for an element $Q$ of the partition, that the distance $\dist(Q,\Gamma)$ controls the Euclidean diameter $\diam(Q)$ of $Q$ and the diameter $\diam(S_\Gamma(Q))$ of the shadow on $\Gamma$. Since the constructions of the partitions stem from initial data given by hyperbolic geodesics, we simultaneously aim for similar control of $\ell(\Gamma^\perp(z)\cap Q)$ and $\ell(\Gamma^\perp(z),Q)$ for $z\in S_\Gamma(Q)$ in terms of $\dist(Q,\Gamma)$. 

The step 3 above can now be understood from the point of view that, for $\sQ_\Gamma$, we do not have uniform estimates for this Whitney-type data, which leads us to pass to the refinement $\sQ_\Gamma^\refi$. 

Note also that, that the elements of $\tilde \sQ_\Gamma$ are topological rectangles in an obvious manner and this same property is carried over to subsequent partitions $\sQ_\Gamma$, $\sQ_\Gamma^\refi$, and $\tilde \sQ_\Gamma^\refi$. Therefore we refer to the elements in these partitions as rectangles. 
\end{rem}

\begin{convention}
In the forthcoming sections, all implicit constants are either absolute or depend only on parameters $p\in (1,2)$ and $C_0\ge 1$. From now on, we consider $p$ and $C_0$ as fixed, but arbitrary, parameters. Although, also the Jordan curve $\Gamma$ could be considered as fixed for once and for all, we have decided to emphasize the role of $p$ and $C_0$ in the statements, and state the results for Jordan curves in $\Jordan(p,C_0)$. 

\end{convention}

\section{Step 1: Partition {$\tilde {\sQ}_\Gamma$}}

We begin by fixing a conformal map
$\tilde \varphi_\Gamma: \mathbb D\to \tilde \Omega_{\Gamma}$ 
so that $\tilde \varphi(0)$ is a point furthest away from the $\Gamma.$ Then
$\tilde \varphi_\Gamma$ extends homeomorphically up to boundary by the Carath\'edory-Osgood theorem and we refer also to this extension by $\tilde \varphi_\Gamma.$ 


Let $\Gamma \in \Jordan(p,C_0)$. We denote 
\[
\tilde A_\Gamma = \tilde \Omega_\Gamma\setminus \tilde \varphi_\Gamma(B(0,1/2)).
\]
Then $\Gamma$ is a boundary component of $\tilde A_{\Gamma}$.

Let $\sP$ be the standard dyadic partition of the annulus $\D\setminus B(0,1/2)$, that is, 
\[
\sP=\{ P_{k,j} \subset \D \colon k\in \N,\ 0\le j\le 2^{k+1}-1\}, 
\]
where
\[
P_{k,j} = \{ re^{i\theta} \in \D \colon r\in [1-2^{-k},1-2^{-(k+1)}],\ \theta\in [j 2^{-k} \pi, (j+1)2^{-k}\pi]\}
\]
for each $k\in \N$ and $0 \le j \le 2^{k+1}-1$. We set $\tilde \sQ_\Gamma$ to be the image of $\sP$ under $\tilde \varphi_\Gamma$, that is,
\[
\tilde\sQ_\Gamma = \{ \tilde\varphi_\Gamma(P) \colon P \in \sP\}.
\]
See Figure~\ref{map}.

\begin{figure}[ht]
\begin{overpic}[scale=0.4, unit=1mm]{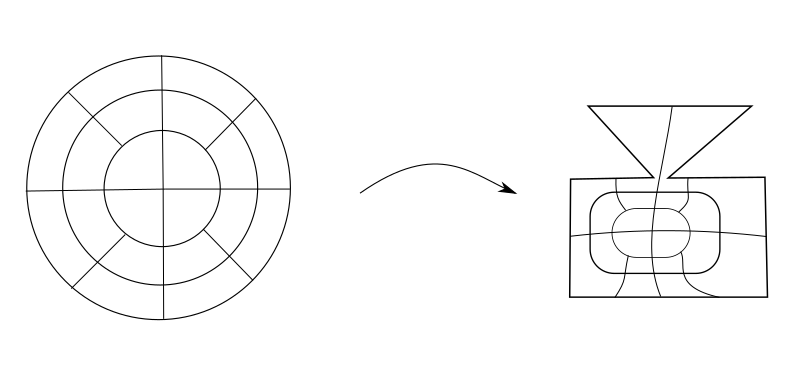}
\put(60,35){$\tilde \varphi_{\Gamma}$} \put(22,3){$\mathbb D$} \put(90,3){$\tilde \Omega_{\Gamma}$}  \put(31,29){$P_{ij}$}
\end{overpic}
\caption{A conformal map $\tilde \varphi_{\Gamma}$ transfers a Whitney-type decomposition $\{P_{k,\,j}\}$ of the unit disk to a Whitney-type one in the $(2-p)$-subhyperbolic domain $\tilde \Omega$.}
\label{map}
\end{figure}
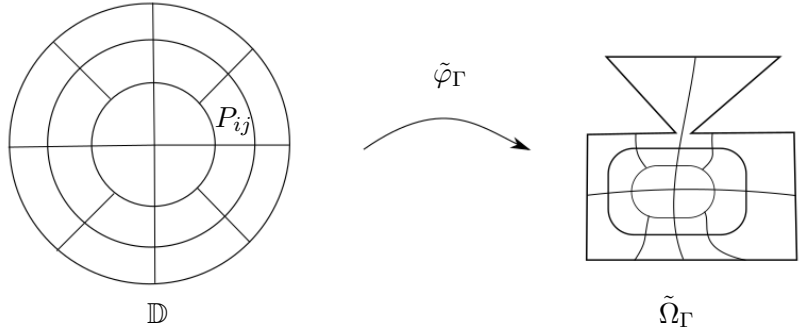

\subsection{Main features of ${\tilde{\sQ}_\Gamma}$}

It is fairly easy to see that rectangles in $\tilde \sQ_\Gamma$ are Whitney-type sets with a constant depending only on data (Lemma \ref{lemma:tilde_sQ_Whitney1} below). Therefore, since the hyperbolic rays in $\tilde \Omega_\Gamma$ are quasigeodesics, we further gather from the Whitney data that there exists, for each $\tilde Q\in \sQ_\Gamma,$ a short hyperbolic geodesic connecting $S_\Gamma(\tilde Q)$ and $\tilde Q$ (Corollary \ref{cor:tilde_Q_distance}).

\begin{rem}
 It is, however, crucial for the forthcoming discussion to observe that, for $z\in S_\Gamma(\tilde Q)$, the length $\ell(\tilde \Gamma^\perp(z,\tilde Q))$ is not controlled from above by $\dist(\tilde Q, \Gamma)$. More precisely, the function $\tilde \sQ_\Gamma \to (0,\infty)$, 
\begin{equation}
\label{eq:tilde_Q_blow-up}
\tag{{$\tilde{\sQ}_\Gamma$}}
\tilde Q \mapsto \max_{z\in S_\Gamma(\tilde Q)} \frac{\ell(\tilde \Gamma^\perp(z,\tilde Q))}{\dist(\tilde Q, \Gamma)}
\end{equation}
need not be (even) bounded. This is the crucial feature of $\tilde \sQ_\Gamma$, which we tackle repeatedly in forthcoming sections. Note that a large ratio in \eqref{eq:tilde_Q_blow-up} also leads to a large ratio for $\diam(S_\Gamma(\tilde Q))$ and $\dist(\tilde Q,\Gamma)$. 
\end{rem}

\begin{lem}
\label{lemma:tilde_sQ_Whitney1}
Let $\Gamma \in \Jordan(p,C_0)$. Then there exists an absolute constant $\lambda\ge 1$ for which each $\tilde Q\in \tilde \sQ_\Gamma$ is a $\lambda$-Whitney-type set and there exists an absolute constant $C>0$ for which
\begin{equation}
\label{eq:dist_diam_tilde_Omega}
\dist(\tilde Q, S_\Gamma(\tilde Q) ) \le C  \diam(\tilde Q).
\end{equation}
\end{lem}

\begin{proof}
Let $P\in \sP$ be such that $\tilde \varphi_\Gamma(P) = \tilde Q$. The first claim follows immediately from Lemma~\ref{whitney preserving} and the fact that each $P\in  \sP$ is of $4$-Whitney-type. 

For the second claim, we observe that, by the conformal invariance of the conformal capacity, we have that
\[
1\ls \capa(P,S_{\bS^1}(P); \mathbb D)= \capa( \tilde Q,S_{\Gamma}(\tilde Q);\Omega), 
\]
where the constant of comparability is absolute. Here $\capa(E,F;G)$ is the conformal capacity of continua $E$ and $F$ in the closure of the domain $G$; see \cite{Vaisala-book} for details.
Thus, by e.g.~\cite[Lemma 2.10]{KRZ2015}, we have that
\[
\dist_{\tilde \Omega_\Gamma}(\tilde Q,S_{\Gamma}(\tilde Q))\ls \diam_{\tilde \Omega_\Gamma}(\tilde Q),
\] 
where the constants depend only on $\lambda$. Here $\diam_{\tilde \Omega_\Gamma}(\tilde Q)$ is the inner diameter of $\tilde Q$ in $\tilde \Omega_\Gamma$. Since $\tilde Q$ is a Whitney-type set, we further have, by a simple covering argument, that $\diam_{\tilde \Omega_\Gamma}(\tilde Q) \sim \diam(\tilde Q)$, where the constant depends only on $\lambda$; see e.g.\ the proof of \cite[Lemma 2.5]{KRZ2015}.

Since $\tilde \Omega_\Gamma$ is quasiconvex by Lemma \ref{lemma:qconvex}, we have that 
$$\dist(\tilde Q,S_{\Gamma}(\tilde Q))\sim \dist_{\tilde \Omega_\Gamma}(\tilde Q,S_{\Gamma}(\tilde Q)).$$
The claim follows.
\end{proof}

\begin{cor}
\label{cor:tilde_Q_distance}
Let $\Gamma\in \Jordan(p,C_0)$ and $\tilde Q\in \tilde \sQ_\Gamma$. Then there exists a point $z_{\tilde Q}\in S_\Gamma(\tilde Q)$ for which 
\[
\ell( \tilde\Gamma^\perp  (z_{\tilde Q},\tilde Q)) \le C \dist(\tilde Q, \Gamma)
\]
where $C=C(p,C_0)>0$.
\end{cor}

\begin{proof}
By \eqref{eq:dist_diam_tilde_Omega}, there exist points $\tilde y\in \tilde Q$ and $z\in S_\Gamma(\tilde Q)$ for which
\begin{equation*}
|\tilde y-z|\ls \diam(\tilde Q),
\end{equation*}
where the constant depends only on $p$ and $C_0$. Let $\tilde x\in \tilde \Gamma^\perp(z)\cap \tilde Q$ be the other end point of $\tilde \Gamma^\perp(z,\tilde Q)$. Then, by the triangle inequality,
\[
|\tilde x-z|\ls |\tilde y-z| + \diam(\tilde Q)  \ls \diam(\tilde Q), 
\]
Since $\tilde \Gamma^\perp(z,\tilde Q)$ is the hyperbolic geodesic joining $\tilde x$ to $z$, it is a quasigeodesic by Lemma \ref{lemma:qconvex}, that is, $\ell(\tilde \Gamma^\perp(z,\tilde Q)) \sim |\tilde x-z|$. Thus
\[
\ell(\tilde \Gamma^\perp(z,\tilde Q)) \ls \diam(\tilde Q),
\]
where the constants depend only on $p$ and $C_0$.
\end{proof}

\subsection{Intersection length estimates for ${\tilde{\sQ}_\Gamma}$}

By Lemma \ref{lemma:biLip}, any two topological rectangles in $\mathscr P$ of $\tilde \sQ_\Gamma$ are uniformly bilipschitz to each other. Thus the width of $\tilde Q\in \sQ_\Gamma$ in terms of intersection length with hyperbolic geodesics is comparable to the Euclidean diameter of $\tilde Q$. We record this fact as follows.

\begin{cor}
\label{cor:intersection_distance1}
Let $\Gamma\in \Jordan(p,C_0)$, $\tilde Q\in \tilde \sQ_\Gamma$, and $z\in S_\Gamma(\tilde Q)$. Then there exists an absolute constant $C>0$ for which 
\begin{equation}
\label{eq:ID1}
\frac{1}{C} \diam(\tilde Q) \le \ell(\tilde\Gamma^\perp(z) \cap \tilde Q)\le C \diam(\tilde Q). 
\end{equation}
\end{cor}

\begin{proof}
Let $P \in \sP$ and $v\in \bS^1$ for which $\tilde \varphi_\Gamma(P) =\tilde Q$ and $\tilde\varphi(v)=z$. Then $\tilde \varphi_\Gamma([0,v]) = \tilde \Gamma^\perp(z)$ and $\tilde \varphi_\Gamma(P \cap [0,v]) = \tilde Q\cap \Gamma^\perp(z)$. Since
\[
\frac{1}{4} \diam(P) \le \ell(P \cap [0,v]) \le 4 \diam(P),
\]
we have, by Lemma~\ref{lemma:biLip}, that  
\[
\ell(\tilde Q\cap \Gamma^\perp(z))\sim \diam(\tilde Q),
\]
where the constant is absolute.
\end{proof}

In terms of the subhyperbolic metric, the intersection length estimate takes the following form.

\begin{cor}
\label{cor:intersection_distance2}
Let $\Gamma\in \Jordan(p,C_0)$, $\tilde Q\in \tilde \sQ_\Gamma$, and $z\in S_\Gamma(\tilde Q)$. Then
\begin{equation}
\label{eq:ID2}
\int_{ \Gamma^\perp(z) \cap  \tilde  Q} \dist(w, \Gamma)^{1-p} \dw \sim \dist(\tilde Q, \Gamma)^{2-p},
\end{equation}
where the constants are absolute. 
\end{cor}

\begin{proof}
By Lemma~\ref{lemma:tilde_sQ_Whitney1}, $\diam \tilde Q \sim \dist(\tilde Q, \Gamma)$ with an absolute constant. By Corollary \ref{cor:intersection_distance1}, $\ell(\tilde \Gamma^\perp(z)\cap \tilde Q) \sim \dist(\tilde Q,\Gamma)$. Thus
\begin{align*}
\int_{\Gamma^\perp(z) \cap  \tilde  Q} \dist(w, \Gamma)^{1-p} \dw \sim \dist(\tilde Q, \Gamma)^{1-p} \dist(\tilde Q, \Gamma) = \dist(\tilde Q,\Gamma)^{2-p}.
\end{align*}
\end{proof}

\subsection{Estimates for shadows}
\label{sec:tilde_sQ_shadows}

We record two diameter estimates for shadows of rectangles in $\tilde \sQ_\Gamma$. These estimates are used in Section \ref{sec:sQ-Gamma-refined} to obtain the refinement $\sQ^\refi_\Gamma$ of the partition $\sQ_\Gamma$ in $\Omega_\Gamma$.

The first observation is that, for a square $\tilde Q$ in $\tilde \sQ_\Gamma$ which has a small shadow, the diameter of the shadow $S_\Gamma(\tilde Q)$ and the square itself $\tilde Q$ have comparable diameter. 

\begin{lem}\label{lem:small shadow}
\label{lem:small_shadow}
Let $\Gamma\in \Jordan(p,C_0)$. Then there exist $C=C(p,C_0)>0$ with the following property. Suppose that, for each $\tilde Q \in \tilde \sQ_\Gamma$ and any $z\in S_\Gamma(\tilde Q)$, we have 
\begin{equation}\label{assum 1}
\diam (S_\Gamma(\tilde Q)) \le \ell(\tilde \Gamma^\perp(z,\tilde Q)).
\end{equation} 
Then 
\begin{equation}
\label{eq:shadow-dist}
 \frac 1 C \dist(\tilde Q, \Gamma)\le \diam (S_\Gamma(\tilde Q))\le C \dist(\tilde Q, \Gamma).
\end{equation}
\end{lem}

\begin{proof}
The lower bound in \eqref{eq:shadow-dist} is independent of the assumption \eqref{assum 1}. Indeed, since $\tilde Q$ is a Whitney-type set by Lemma~\ref{lemma:tilde_sQ_Whitney1}, we have, by Lemma~\ref{lemma:shadow1} and  Lemma~\ref{lemma:tilde_sQ_Whitney1} again, that 
\[
\diam(S_\Gamma(\tilde Q))\gs \diam(\tilde Q)\sim \dist(\tilde Q, \Gamma),
\]
where the constants depend only on $p$ and $C_0$. 

The upper bound follows immediately from Corollary \ref{cor:tilde_Q_distance}. Indeed, by assumption, we have
\[
\diam(S_\Gamma(\tilde Q)) \le \ell(\tilde \Gamma^\perp(z,\tilde Q)) \ls \dist(\tilde Q,\Gamma),
\]
where the constants depend only on $p$ and $C_0$.
\end{proof}

For subarcs in the shadow, we have the following comparability estimate.

\begin{lem}
\label{lemma:ax1}
Let $\Gamma\in \Jordan(p,C_0)$. Then there exists $c_0=c_0(p, C_0)\ge 1$ for the following. Let $\tilde Q\in \tilde \sQ_\Gamma$. Let also $I \subset S_\Gamma(\tilde Q) \subset \Gamma$ be an arc and suppose there exists a point $z_0\in I$ for which 
\begin{equation}
\label{eq:choice1}
\diam(I)=\ell(\tilde \Gamma^\perp(z_0, \tilde Q)).
\end{equation}
Then, for each $z\in I$, 
\begin{equation}
\label{eq:shadow_and_length1}
\frac{1}{c_0}\le \frac {\diam(I)}{\ell(\tilde \Gamma^\perp(z, \tilde Q))} \le  c_0.
\end{equation}
\end{lem}

\begin{proof}
Let $z\in I$ and let $\tilde x\in \Gamma^\perp(z)$ be the point for which $\tilde \Gamma^\perp(\tilde x') = \tilde \Gamma(z,\tilde Q)$. Let also $\tilde x_0\in \Gamma^\perp(z_0)$ be the point satisfying $\tilde \Gamma^\perp(\tilde x_0) = \tilde \Gamma^\perp(z_0,\tilde Q)$.

Since $\tilde \Gamma^\perp(z,\tilde Q)$ is a quasigeodesic in the Euclidean metric, we have that
\begin{align*}
\ell(\Gamma^\perp(z,\tilde Q)) &\ls |z-\tilde x| \le |z-z_0| + |z_0 - \tilde x_0| + |\tilde x_0 - \tilde x| \\
&\le \diam(I) + \ell(\tilde\Gamma^\perp(z_0,\tilde Q))+\diam(\tilde Q) \\
&\ls \diam(I) + \dist(\tilde Q, \Gamma) \\
&\ls \diam(I)
\end{align*}
by Lemma \ref{lemma:tilde_sQ_Whitney1} and Corollary \ref{cor:tilde_Q_distance}.

The other direction follows immediately from Corollary \ref{cor:tilde_Q_distance}. Indeed, 
\begin{align*}
\diam(I) &= \ell(\tilde \Gamma^\perp(z_0,\tilde Q)) \ls \dist(\tilde Q,\Gamma) \ls \ell(\tilde \Gamma^\perp(z,\tilde Q)).
\end{align*}
The claim follows.
\end{proof}


\section{Step 2a: Stable reflection over $\Gamma$}
\label{vitonen}

We now leave the partition $\tilde \sQ_\Gamma$ for a while and discuss, as a preparatory step for the construction of a partition $\sQ_\Gamma$ in $\Omega_\Gamma$, a reflection of $\tilde A_\Gamma$ over the Jordan curve $\Gamma$. As will be soon apparent, this initial reflection will reflect hyperbolic rays in $\tilde \Omega_\Gamma$ to (reparametrized) hyperbolic rays in $\Omega_\Gamma$ respecting the distance data \eqref{eq:tilde_Q_blow-up}. The following lemma gives the existence of the requested reflection.

\begin{lem}
\label{lemma:bijection}
Let $\Gamma \in \Jordan(p,C_0)$. Then there exists an embedding $h_\Gamma \colon \tilde A_\Gamma \to \Omega_\Gamma$ satisfying, for each $\tilde x\in \tilde A_\Gamma$,
\begin{equation}
\label{eq:h}
\int_{\tilde \Gamma^\perp(\tilde x)}\dist(w,\Gamma)^{1-p} \dw=\ell\left(\Gamma^\perp(h_\Gamma(\tilde x))\right)^{2-p}.
\end{equation}
Furthermore, $h_\Gamma$ extends to a homeomorphism $\hat h_\Gamma \colon \mathrm{cl} \tilde A_\Gamma \to \mathrm{cl}(h_\Gamma \tilde A_\Gamma)$ satisfying $\hat h_\Gamma|_\Gamma = \id_\Gamma$.
\end{lem}

Let us now fix a conformal map $\varphi_\Gamma:\mathbb D\setminus \{0\}\to \Omega_\Gamma,$ where we view $\Omega_\Gamma$ as a subset of $\mathbb C.$

\begin{proof}
For each $z\in \Gamma$, the length function $\ell_z \colon \tilde \Gamma^\perp(z)\cap \tilde \Omega_\Gamma \to [0,\infty)$,
\begin{equation}
\label{eq:h2}
\tilde x \mapsto \int_{\tilde \Gamma^\perp(\tilde x)} \dist(w,\Gamma)^{1-p} \dw,
\end{equation}
is continuous and strictly monotone. By Lemma~\ref{GM2} and the fact that the hyperbolic ray $\Gamma^\perp(z)\cap  \Omega_\Gamma$ has infinite (Euclidean) length in $\C$, the map $h_\Gamma$ is well-defined.

To show that $h_\Gamma$ is an embedding,  we observe first that the mapping $\Gamma\times (0,1] \to \C$, defined by $(z,t) \mapsto \varphi_\Gamma(t \varphi_\Gamma^{-1}(z))$ for $t\ge 1$ and by $(z,t)\mapsto  \tilde\varphi_\Gamma(t \tilde \varphi_\Gamma^{-1}(z))$ for $t\le 1$, is an embedding. The continuity and injectivity of $h_\Gamma$ follow then from the continuity and monotonicity of the length functional \eqref{eq:h2} on hyperbolic rays $\tilde \gamma_z$ for $z\in \Gamma$. The continuity of $h_\Gamma^{-1}$ is analogous. 
\end{proof}

\begin{rem} To simplify our notation we refer also to $\hat h_\Gamma$ by 
$h_\Gamma$ in what follows.
Since $h_\Gamma$ is uniquely determined by the data which we have associated to $\Gamma$, we call $h_\Gamma$ in Lemma \ref{lemma:bijection} the \emph{stable reflection over $\Gamma$}. 
\end{rem}

\subsection{Main features of the reflection $h_\Gamma$}

In the following lemma, we record three properties of the reflection $h_\Gamma$ which are repeatedly used in the following sections. The first two are given in terms of the lengths of hyperbolic geodesics. 

Heuristically, the first claim states that $h_\Gamma$ is bilipschitz, as a map of hyperbolic geodesics. This yields, in particular, that properties of the control function \eqref{eq:tilde_Q_blow-up} are essentially preserved under $h_\Gamma$. 

The second claim states that a similar comparability property also holds in terms of the distance to the boundary; this uses the fact that $\Omega_\Gamma$ is a John domain. The third claim states that the mapping $h_\Gamma$ is (essentially) repelling in terms of the boundary $\Gamma$. Note that all distances and lengths are with respect to Euclidean metric.

\begin{lem}
\label{lemma:zxz}
Let $\Gamma \in \Jordan(p,C_0)$ and let $h_\Gamma \colon \tilde A_\Gamma \to \Omega_\Gamma$ be a stable reflection over $\Gamma$. Then there exists $L_{\Gamma}=L_{\Gamma}(p,C_0)\ge 1$ satisfying, for each $\tilde x\in \tilde A_\Gamma$,
\begin{equation}
\label{comparable length}
\frac{1}{L_{\Gamma}} \ell\left(\tilde \Gamma^\perp(\tilde x)\right) \le \ell\left(\Gamma^\perp(h_\Gamma(\tilde x))\right) \le L_{\Gamma} \ell\left(\tilde \Gamma^\perp(\tilde x)\right),
\end{equation}
\begin{equation}
\label{eq:not_zxz} 
\frac{1}{L_\Gamma} \dist(h_\Gamma(\tilde x),\Gamma) \le \ell\left(\Gamma^\perp(h_\Gamma(\tilde x))\right) \le L_\Gamma \dist(h_\Gamma(\tilde x),\Gamma), 
\end{equation}
and
\begin{equation}
\label{zxz}
\dist(\tilde x,\Gamma)\le L_\Gamma \dist(h_\Gamma(\tilde x),\Gamma).
\end{equation}
\end{lem}

\begin{proof}
Let $\tilde x\in \tilde A_\Gamma$. By \eqref{eq:h} and Corollary \ref{cor:Pankka3}, we have that
\begin{equation*}
\ell\left(\Gamma^\perp(h_\Gamma(\tilde x))\right)^{2-p} = \int_{\Gamma^\perp(\tilde x)} \dist(w,\Gamma)^{1-p} \ds(w)\sim \ell\left(\Gamma^\perp(\tilde x)\right)^{2-p},
\end{equation*}
where the constants depend only on $p$ and $C_0$. This implies \eqref{comparable length}.

By Corollary~\ref{comp GM}, the domain $\Omega_\Gamma$ is a John domain with a constant depending only on $C_0$ and hyperbolic rays in $\Omega_\Gamma$  are John curves. Thus, we obtain that
\begin{equation*}
\ell\left(\Gamma^\perp(h_\Gamma(\tilde x))\right)\ls \dist(h_\Gamma(\tilde x),\Gamma) \le |h_\Gamma(\tilde x)-S_\Gamma(\tilde x)|\le \ell\left(\Gamma^\perp(h_\Gamma(\tilde x))\right),
\end{equation*}
where the constants depend only on $p$ and $C_0$. This yields \eqref{eq:not_zxz}. 

For \eqref{zxz}, we observe that, by \eqref{comparable length} and \eqref{eq:not_zxz}, we have the estimate
\begin{equation*}
\dist(\tilde x,\Gamma) \le |\tilde x-S_\Gamma(\tilde x)|\le \ell(\Gamma^\perp(\tilde x)) 
\sim \ell\left(\Gamma^\perp(h_\Gamma(\tilde x))\right) \sim \dist(h_\Gamma(\tilde x),\Gamma),
\end{equation*}
where the constants depend only on $p$ and $C_0$. 
\end{proof}


\section{Step 2b: Partition $\sQ_\Gamma$}

We are now ready to use the reflection $h_\Gamma$ to construct a partition in a neighborhood of $\Gamma$ in $\Omega_\Gamma$.

We define 
\[
A_\Gamma = h_\Gamma(\tilde A_\Gamma).
\]
The partition $\sQ_\Gamma$ of $A_\Gamma$ is the image of $\tilde \sQ_\Gamma$ under $h_\Gamma$, that is, 
\[
\sQ_\Gamma = \{ h_\Gamma(\tilde Q) \colon \tilde Q \in \tilde \sQ_\Gamma\}.
\]

\subsection{Main features of $\sQ_\Gamma$}

Clearly, for each $Q\in \sQ_\Gamma$, we have 
\[
S_\Gamma(Q) = S_\Gamma(h_\Gamma^{-1}(Q)).
\]
Thus it is natural to compare the lengths of the hyperbolic geodesics $\ell(\Gamma^\perp(z,Q))$ and $\ell(\tilde \Gamma^\perp(z,\tilde Q))$ for $z\in S_\Gamma(\tilde Q)$. The characteristic property of $\sQ_\Gamma$ is that these lengths are comparable. More precisely, it is an immediate consequence of \eqref{comparable length} that, for any $Q\in \sQ_\Gamma$ and any $z\in S_\Gamma(Q)$, we have
\begin{equation}
\label{smallest length}
\ell(\Gamma^\perp(z,Q)) \sim \ell(\tilde \Gamma^\perp(z,h_\Gamma^{-1}(Q)),
\end{equation}
where the constants depend only on $p$ and $C_0$. Note that, by Corollary \ref{cor:tilde_Q_distance}, we also have that there exists $z_Q\in S_\Gamma(Q)$ for which
\begin{equation}
\label{smallest length2}
\ell(\Gamma^\perp(z_Q,Q)) \sim \dist(Q,\Gamma).
\end{equation}
Indeed, it suffices to observe that we may take $z_Q = h_\Gamma(z_{\tilde Q})$ and that $\dist(\tilde Q, \Gamma) \le \dist(Q,\Gamma)$ by the expansion property of $h_\Gamma$ (i.e.~\eqref{zxz}).

\begin{rem}
This immediately yields that the control function in \eqref{eq:tilde_Q_blow-up} is comparable to the control defined on $\sQ_\Gamma $ by setting
\begin{equation}
\label{eq:Q_blow-up}
\tag{{$\sQ_\Gamma$}}
Q\mapsto \max_{z\in S_\Gamma(Q)} \frac{\ell(\Gamma^\perp(z,Q))}{\dist(h_\Gamma^{-1}(Q), \Gamma)}.
\end{equation}
In this respect the partitions $\tilde \sQ_\Gamma$ and $\sQ_\Gamma$ are symmetric with respect to $\Gamma$. Note, however, that the rectangles in $\sQ_\Gamma$ need not resemble rectangles in $\tilde \sQ_\Gamma$. For example, the rectangles in $\sQ_\Gamma$ need not be Whitney sets with uniform constant. See Figure~\ref{rough_reflection}. 
\end{rem}

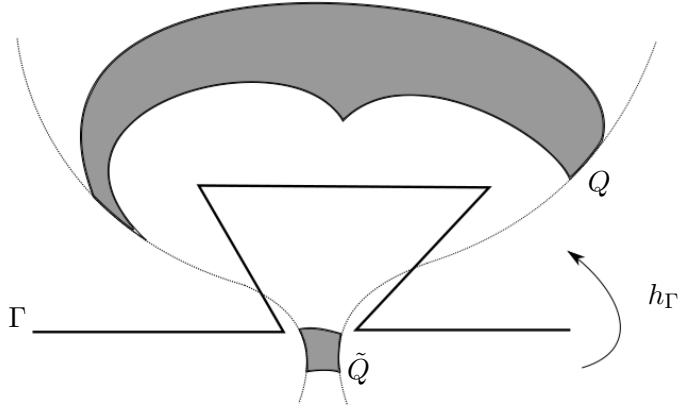
\begin{figure}[ht]
\begin{overpic}[scale=0.4,unit=1mm]{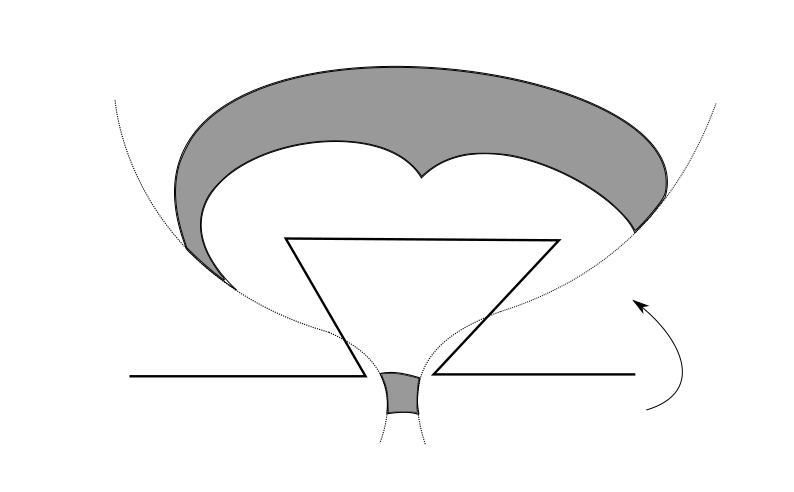}  
\put(60,10){$\tilde Q$} \put(100,20){$h_{\Gamma}$} \put(92,35){$Q$} \put(15,17){$\Gamma$}
\end{overpic}
\caption{The rough reflection $h_{\Gamma}$ maps hyperbolic rays $\tilde \Gamma^\perp\subset \tilde \Omega$ to hyperbolic rays $\Gamma^\perp$ in $\Omega$, which induces a decomposition in $\Omega$ from that of $\tilde \Omega$.}
\label{rough_reflection}
\end{figure}

\subsection{Metric properties of $\sQ_\Gamma$}

We now discuss in more detail the metric properties of the rectangles in $\sQ_\Gamma$. We have collected the main estimates to the following proposition. Note that, by construction, $S_\Gamma(h_\Gamma(\tilde Q)) = S_\Gamma(\tilde Q)$ for each $\tilde Q\in \tilde \sQ_\Gamma$. 

\begin{prop}\label{prop:width bound}
Let $\Gamma \in \Jordan(p,C_0)$. Then there exists constants of comparabilty depending only on $p$ and $C_0$ having the following properties. Let $\tilde Q\in \tilde \sQ_\Gamma$, and $Q=h_\Gamma(\tilde Q)$. Then, for each $z\in S_\Gamma(Q)$, 
\begin{equation}
\label{eq:Q-length}
\ell(\Gamma^\perp(z)\cap Q) \gs \dist(\tilde Q,\Gamma)
\end{equation}
and
\begin{equation}\label{bound on width}
\ell(\Gamma^\perp(z)\cap Q) \ls \dist(\Gamma^\perp(z)\cap Q,\Gamma).
\end{equation}
Furthermore, we have
\begin{equation}
\label{eq:width}
\min_{z\in S_{\Gamma}(\tilde Q)} \ell( \Gamma^\perp(z) \cap Q) \sim \dist(\tilde Q, \Gamma).
\end{equation}
\end{prop}

\begin{rem}
As it is easy to notice, we are more interested in comparing metric quantities in terms of $\dist(\tilde Q,\Gamma)$ rather than in terms of $\dist(Q,\Gamma)$. The reason for this is that in the following sections we consider refinements of $\sQ_\Gamma$ and their reflections in $\tilde \Omega_\Gamma$.
\end{rem}

We prove Proposition \ref{prop:width bound} in steps. As a preliminary step, we record a power law for the relative length and distance for the intersection of a hyperbolic ray with a rectangle of $\tilde Q_{\Gamma}$. We state this fact in terms of the partition $\tilde \sQ_\Gamma$ and our reflection $h_\Gamma$.

\begin{lem}\label{lemma:Q1}
Let $\Gamma \in \Jordan(p,C_0)$, $\tilde Q\in \tilde \sQ_\Gamma$, and $Q=h_\Gamma(\tilde Q)$. Then, for each $z\in S_\Gamma(Q)$, 
\begin{align}\label{width}
\frac{\ell(\Gamma^\perp(z)\cap Q)}{\dist(\tilde Q,\Gamma)} \sim \left( \frac{\dist(\Gamma^\perp(z)\cap Q,\Gamma)}{\dist(\tilde Q,\Gamma)}\right)^{p-1},
\end{align}
where the constants depend only on $p$ and $C_0$. 
\end{lem}

\begin{proof}
Let $z\in S_\Gamma(\tilde Q)$ and let $\tilde x_1\in \tilde Q$ be the other end point of $\tilde \Gamma^\perp(z,\tilde Q)$. Then $\tilde x_1$ is also an end point of $\Gamma^\perp(z)\cap \tilde Q$. Let then $\tilde x_2\in \tilde Q$ be the end point of $\tilde \Gamma^\perp(z)\cap Q$ which is not $\tilde x_1$. 

Then, by definition of $h_\Gamma$ and Taylor's formula $(x+y)^\az-x^\az\sim x^{\az-1}y$ for $0<y<x$ and $\az> 1$, we have that
\begin{align*}
&\ell(\Gamma^\perp(z)\cap Q) = \ell(\Gamma^\perp(h_\Gamma(\tilde x_2)))- \ell(\Gamma^\perp(h_\Gamma(\tilde x_1)))  \\
& = \left(\int_{\Gamma^\perp(\tilde x_2)} \dist(w,\Gamma)^{1-p} \dw \right)^{\frac{1}{2-p}} - \left(\int_{\Gamma^\perp(\tilde x_1)} \dist(w,\Gamma)^{1-p} \dw \right)^{\frac{1}{2-p}} \\
&\sim \left( \int_{\Gamma^\perp(\tilde x_1)} \dist(w, \Gamma)^{1-p} \dw \right)^{\frac {p-1}{2-p}} \int_{\Gamma^\perp(z)\cap \tilde Q} \dist(w, \Gamma)^{1-p} \dw \\
&\sim \,\ell\left( \Gamma^\perp(h_\Gamma(\tilde x_1))\right)^{p-1} \int_{\Gamma^\perp(z)\cap \tilde Q} \dist(w, \Gamma)^{1-p} \dw \\
&\sim \,\ell\left( \Gamma^\perp(h_\Gamma(\tilde x_1))\right)^{p-1} \dist(\tilde Q,\Gamma)^{2-p} 
\sim \left( \dist(h_\Gamma(\tilde x_1),\Gamma)\right)^{p-1} \dist(\tilde Q, \Gamma)^{2-p},
\end{align*}
where we used Corollary \ref{cor:intersection_distance2} and \eqref{eq:not_zxz} in Lemma \ref{lemma:zxz} in the last two steps.
\end{proof}

\begin{lem}\label{lemma:Q2}
Let $\Gamma \in \Jordan(p,C_0)$, $\tilde Q\in \tilde \sQ_\Gamma$, and $Q=h_\Gamma(\tilde Q)$. Let $z\in S_\Gamma(Q)$, and let $x\in \Gamma^\perp(z)\cap Q$ be the unique point for which $\Gamma^\perp(x) = \Gamma^\perp(z,Q)$. Then 
\begin{equation}\label{represent dist}
\dist(\Gamma^\perp(z) \cap Q, \Gamma)\sim \dist(x,\Gamma),
\end{equation}
where the constants depend only on $p$ and $C_0$.
\end{lem}

\begin{proof}
Let $x'\in \Gamma^\perp(z)\cap Q$ be a point which realizes the distance of $\Gamma^\perp(z)\cap Q$ to $\Gamma$, that is, $\dist(x',\Gamma) = \dist(\Gamma^\perp(z) \cap Q, \Gamma)$, and let $\tilde x' = h_\Gamma^{-1}(x')$. Let also $\tilde x = h_\Gamma^{-1}(x)$. It suffices to show that 
\begin{equation}
\label{eq:x-x'}
\dist(x,\Gamma) \sim \dist(x',\Gamma).
\end{equation}

We observe first that $\ell(\Gamma^\perp(\tilde x)) \sim \ell(\Gamma^\perp(\tilde x'))$. Indeed, 
\begin{align*}
\ell(\Gamma^\perp(\tilde x))\le \ell(\Gamma^\perp( \tilde x'))+\ell(\tilde \Gamma^\perp(z) \cap \tilde Q)
\le \ell(\tilde \Gamma^\perp(\tilde x'))+\diam(\tilde Q)\ls \ell(\tilde \Gamma^\perp(\tilde x')), 
\end{align*}
since $\diam(\tilde Q)\ls \dist(\tilde Q, \Gamma) \le \ell(\tilde \Gamma^\perp(\tilde x'))$. The other direction is similar.

Hence, together with \eqref{eq:not_zxz} and \eqref{comparable length}, we obtain that
\begin{align*}
\dist(h_\Gamma( \tilde x),\Gamma) &\sim \ell(\tilde\Gamma^\perp(h_\Gamma(\tilde x))) \sim \ell(\tilde\Gamma^\perp(\tilde x)) \sim \ell(\tilde\Gamma^\perp(\tilde x')) \\
&\sim \ell(\Gamma^\perp(h_\Gamma(\tilde x'))) \sim \dist(h_\Gamma(\tilde x'),\Gamma). 
\end{align*}
Thus \eqref{eq:x-x'} holds. This completes the proof.
\end{proof}

Now we are ready to prove Proposition~\ref{prop:width bound}.

\begin{proof}[Proof of Proposition~\ref{prop:width bound}]
By \eqref{zxz}, we have
\[
\dist(\tilde Q,\,\Gamma)\le \dist(\Gamma^\perp(z)\cap  \tilde Q,\Gamma)\ls \dist(\Gamma^\perp(z)\cap h(\tilde Q),\Gamma).
\]
Then by Lemma \ref{lemma:Q1} together with $1<p<2$, recalling $Q=h(\tilde Q)$ we have
\begin{multline*}
\dist(\tilde Q,\,\Gamma) \ls  {\ell(\Gamma^\perp(z)\cap Q)} \\
 \sim    {\dist(\Gamma^\perp(z)\cap Q,\Gamma)}^{p-1}{\dist(\tilde Q,\Gamma)}^{2-p}\ls \dist(\Gamma^\perp(z)\cap Q,\Gamma). 
\end{multline*} 
This proves \eqref{eq:Q-length} and \eqref{bound on width}. 

In order to prove \eqref{eq:width}, we first note that by Corollary~\ref{cor:tilde_Q_distance}, there exists $z_{\tilde{Q}}\in S_{\Gamma} (\tilde Q)$ so that
$$\ell(\tilde \Gamma^\perp(z_{\tilde{Q}},\,\tilde{Q}))\ls \dist(\tilde Q, \Gamma).$$
Then by \eqref{comparable length} and \eqref{represent dist} we conclude
\begin{multline*}
\dist(\tilde Q, \Gamma)\gs  \ell(\tilde \Gamma^\perp(z_{\tilde{Q}},\,\tilde{Q}))\\
\sim \ell( \Gamma^\perp(z_{\tilde{Q}},\,{Q}))
\gs \dist(x,\,\Gamma)\sim \dist(\Gamma^\perp(z) \cap Q, \Gamma), 
\end{multline*}
where $x$ is the point as in Lemma~\ref{lemma:Q2}. Plugging this into \eqref{bound on width}, together with \eqref{eq:Q-length} we obtain \eqref{eq:width}.
\end{proof}


\section{Step 3: Refined partition ${\sQ_\Gamma^\refi}$}
\label{sec:sQ-Gamma-refined}

In this section we define the final partition $\sQ_\Gamma^\refi$ of $A_\Gamma$. The general strategy is to subdivide the rectangles of $\sQ_\Gamma$ by subdividing the shadows of rectangles in $\sQ_\Gamma$, we call the obtained partition a shadow refinement of $\sQ_\Gamma$. More precisely, a partition $\sR$ of $A_\Gamma$ is a \emph{shadow refinement of $\sQ_\Gamma$} if for each $R\in \sR$ there exists $Q\in \sQ_\Gamma$ and a closed subarc $J \subset S_\Gamma(Q)$ for which $R = Q\cap S_\Gamma^{-1}(J)$. Given $R\in \sR$, we call the rectangle $Q\in \sQ_\Gamma$ containing $R$ the \emph{parent of $R$} and $R$ a \emph{child of $Q$}. 

We denote by $\sC(Q)$ the family of children of $Q\in \sQ_\Gamma$, that is, 
\[
\sC(Q) = \{ R\in \sR \colon R\subset Q\}.
\]

We extend this terminology, already at this point, by saying that a rectangle $\tilde Q\in \tilde \sQ_\Gamma$ is a \emph{parent of $h_\Gamma(\tilde Q) \in \sQ_\Gamma$} and a \emph{grandparent for each child of $h_\Gamma(\tilde Q)$}. Similarly, we say that a child $R\in \sR$ is a \emph{grandchild of $\tilde Q$}.

\begin{rem}
Notice that, clearly, each rectangle in a shadow refinement $\sR$ is a topological rectangle having two sides which are hyperbolic geodesics and two sides which are contained in the boundary of its parent. 
\end{rem}

The fundamental reason for the refinement is to fix the issues related to unboundedness of the functions \eqref{eq:tilde_Q_blow-up} and \eqref{eq:Q_blow-up}. We do this by passing to a partition for which the diameters of the shadows are comparable to the length of the hyperbolic rays towards $\Gamma$.  More precisely, we balance the partition $\sQ_\Gamma$ by refining it. 

\begin{defn}
A shadow refinement $\sR$ of $\sQ_\Gamma$ is \emph{$C$-balanced} if, for each $R\in \sR$ and each $z\in S_\Gamma(R)$, it holds
\[
\frac{1}{C} \diam S_\Gamma(R) \le \ell(\Gamma^\perp(z,R)) \le C \diam S_\Gamma(R).
\]
\end{defn}

The following proposition gives the existence of a quantitative balanced shadow refinement $\sQ_\Gamma^\refi$. 

\begin{prop}
\label{prop:existence_of_balanced_shadow_refinement}
Let $\Gamma \in \Jordan(p,C_0)$. Then there exists a constant $C_\refi = C_\refi(p,C_0)\ge 1$ and a $C_\refi$-balanced shadow refinement $\sQ_\Gamma^\refi$ of $\sQ_\Gamma$.
\end{prop} 

\begin{rem}
Although we derive several metric properties for the partition $\sQ_\Gamma^\refi$ in the forthcoming section, the fact that the partition $\sQ_\Gamma^\refi$ is balanced should be considered its main feature. However, a small warning is in order. The fact that $\sQ_\Gamma^\refi$ is balanced yields neither diameter or distance bounds for grandchildren of $\tilde Q\in \sQ_\Gamma^\refi$ in terms of $\dist(\tilde Q, \Gamma)$ nor quantitative upper bounds for the number of children of $Q\in \sQ_\Gamma$. We discuss these properties in more detail later in this section.
\end{rem}

For the partition $\sQ_\Gamma^\refi$, we introduce a corresponding balanced partition of shadows. We say, for $Q\in \sQ_\Gamma$, that an arc $I \subset S_\Gamma(Q)$ is \emph{$C$-balanced with respect to $Q$ for $C\ge 1$} if 
\[
\frac{1}{C} \diam(I) \le  \ell(\Gamma^\perp(z,Q)) \le C \diam(I)
\]
for each $z\in I$. Similarly, a partition $\{I_1,\ldots, I_k\}$ of $S_\Gamma(Q)$ is \emph{$C$-balanced partition of $S_\Gamma(Q)$ with respect to $Q$ for $C\ge 1$} if each $I_i$ is $C$-balanced with respect to $Q$. 

The following lemma gives the existence of an appropriate balanced partition of shadows of rectangles of $\sQ_\Gamma$.

\begin{lem}
\label{lemma:shadow_partition}
Let $\Gamma\in \Jordan(p,C_0)$. Then, for each $Q\in \sQ_\Gamma$, there exists $k_Q\in \N$ and a $C^\refi$-balanced partition $\sI_\Gamma(Q) = \{I_1,\ldots, I_{k_Q}\}$ of $S_\Gamma(Q)$, where $C^\refi=C^\refi(p,C_0)\ge 1$.
\end{lem}

Proposition \ref{prop:existence_of_balanced_shadow_refinement} follows now immediately from this lemma.

\begin{proof}[Proof of Proposition \ref{prop:existence_of_balanced_shadow_refinement}]
For each $Q\in \sQ_\Gamma$, let $\sI_\Gamma(Q)$ be a $C^\refi$-balanced partition of $S_\Gamma(Q)$ for $C^\refi\ge 1$ as in Lemma \ref{lemma:shadow_partition}, and define 
\[
\sQ_\Gamma^\refi = \{ S_\Gamma^{-1}(I) \cap Q \colon I \in \sI_\Gamma(Q)\}.
\]
Then $\sQ_\Gamma^\refi$ is $C_\refi$-balanced partition with $C_\refi = C^\refi$.
\end{proof}

For the proof of Lemma \ref{lemma:shadow_partition}, we assume from now on that the Jordan curve $\Gamma$ is oriented and that this orientation determines, for each pair $x,y\in \Gamma$ of distinct points, a unique positively oriented segment $\Gamma[x,y]$ from $x$ to $y$ on $\Gamma$. Given points $x_0,\ldots, x_k$ on $\Gamma$, we write $x_0 < x_1 < \cdots < x_k$ if the segments $\Gamma[x_{i-1},x_i]$ have mutually disjoint interiors for $i=1,\ldots, k$.

\begin{proof}[Proof of Lemma \ref{lemma:shadow_partition}]
Let $\tilde Q\in \tilde \sQ_\Gamma$ be such that $h_\Gamma(\tilde Q) = Q$ and let $a<b$ in $\Gamma$ be the end points of $S_{\Gamma}(\tilde Q)$, that is, $\Gamma[a,b] = S_\Gamma(\tilde Q)$.

We fix points $a=x_0< \cdots < x_k = b$ inductively as follows. Let $x_0=a$ and suppose that points $x_0< \cdots < x_i < b$ have been chosen. Then let $x_{i+1}\in \Gamma[x_i,b]$ be either the first point after $x_i$ (in the order of orientation) satisfying
\begin{equation}
\label{eq:choice}
\diam(\Gamma[x_i,x_{i+1}])=\ell(\Gamma^\perp(x_i,\tilde Q))
\end{equation}
or let $x_{i+1} = b$ if no such point exists. In the latter case, we set $k = i+1$ and stop the induction process. To see that the induction stops, notice that, for each $z\in \Gamma[a,b]$, we have, by Corollary \ref{cor:tilde_Q_distance}, that $\ell(\Gamma^\perp(z,\tilde Q)) \gs \dist(\tilde Q, \Gamma)$. Since $\Gamma[a,b]$ has finite diameter, we conclude from Corollary \ref{cor:finite division} that there exists an index $i\in \N$ for which $x_{i+1}=b$. Also in this case, we set $k=i+1$.

For each $i\le k-1$, the condition \eqref{eq:choice} is satisfied and hence, by Lemma \ref{lemma:ax1}, the arcs $\Gamma[x_i,x_{i+1}]$ are $C$-balanced with $C=C(p,C_0)$. Thus it remains to consider the case $i=k_Q-1$ and its subcase that the segment $\Gamma[x_{k-1},x_{k}]$ does not satisfy \eqref{eq:choice}. We consider two cases.

Suppose first that $k=1$. Then 
\[
\diam S_\Gamma(\tilde Q) = \diam \Gamma[a,b] \le \ell(\tilde \Gamma^\perp(a,\tilde Q)).
\]
In this case, we have, by Lemma \ref{lem:small shadow} and Corollary \ref{cor:tilde_Q_distance}, that $\Gamma[a,b]$ is $C$-balanced with $C=C(p,C_0)$ and the claim follows.

Suppose now that $k\ge 2$. Then, by Lemma \ref{lemma:ax1}, the arc $\Gamma[x_{k-2},x_k]$ is $C'$-balanced with $C'=C'(p,C_0)$. Thus the partition 
\[
\{ \Gamma[x_0,x_1],\ldots, \Gamma[x_{k-3},x_{k-2}], \Gamma[x_{k-2},x_{k}]\}
\]
of $S_\Gamma(Q)$ is $C$-balanced with respect to $Q$ with $C=C(p,C_0)$. 
\end{proof}

\begin{figure}[ht]
\begin{overpic}[scale=0.4,unit=1mm]{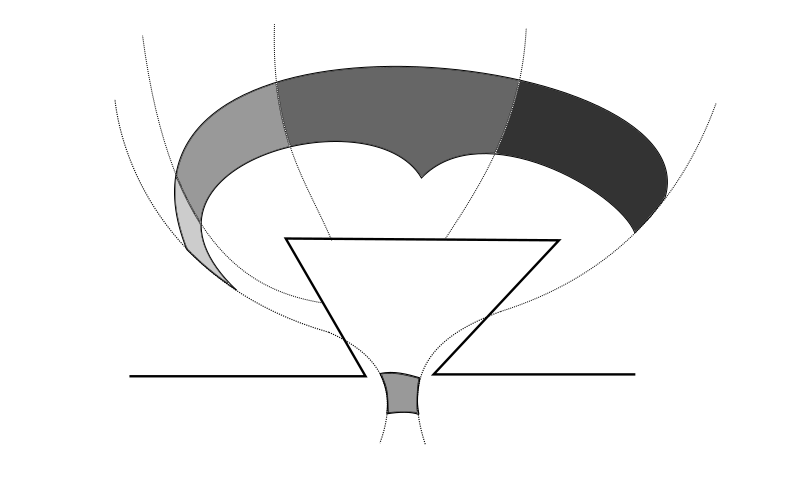} 
\put(61,10){$\tilde Q$} \put(20,45){$R_1$} \put(30,58){$R_2$} \put(55,62){$R_3$} \put(90,53){$R_4$}
\end{overpic}
\caption{An illustration of the decomposition of $Q\in \sQ_\Gamma$ into balanced topological rectangles $R_1,\ldots, R_4$.}
\label{Q_decompose}
\end{figure}
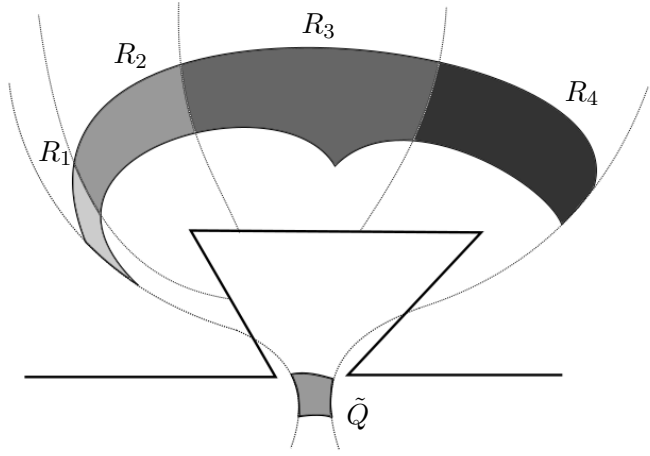

\subsection{Properties of $\sQ_\Gamma^\refi$}

Our first observation is an intersection length estimate for rectangles of $\sQ_\Gamma^\refi$. Note that, since $\sQ_\Gamma^\refi$ is a shadow refinement of $\sQ_\Gamma$, we have, for each $R\in \sQ_\Gamma^\refi$ and its parent $Q\in \sQ_\Gamma$, that 
$$\Gamma^\perp(z) \cap R = \Gamma^\perp(z)\cap Q$$
for each $z\in \Gamma$.

 We begin with the comparability of the diameter of the shadow of $R$ and the distance of $R$ to $\Gamma$.

\begin{lem}
\label{lemma:Qi1}
Let $\Gamma \in \Jordan(p,C_0)$. Then, for each $R\in \sQ_\Gamma^\refi$, we have
\begin{equation}
\label{eq:Qi1a}
\diam(S_\Gamma(R)) \sim \dist(R,\Gamma),
\end{equation}
where the constants depend only on $p$ and $C_0$.
\end{lem}
\begin{proof}
Let $R\in \sQ_\Gamma^\refi$, let $Q$ be a parent of $R$ in $\sQ_\Gamma$ containing $R$, and let $\tilde Q \in \tilde \sQ_\Gamma$ be the rectangle for which $Q = h_\Gamma(\tilde Q)$. Let $x\in R$ be a point for which $\dist(x,\Gamma) = \dist(R,\Gamma)$, and let $z=S_\Gamma(x)$. 

Since $\Gamma^\perp(x) = \Gamma^\perp(z,R)$ and $\Gamma^\perp(z)$ is a John curve in $\Omega \subset \widehat \C$ by Lemma \ref{comp GM}, we have that $\ell(\Gamma^\perp(x)) \sim \dist(x,\Gamma)$ in the Euclidean metric, since $\Gamma^\perp(x) \subset \overline{A_\Gamma}$. Thus, by Lemma \ref{lemma:shadow_partition}, we have that
\[
\diam(S_\Gamma(R)) \sim \ell(\Gamma^\perp(z, Q)) = \ell(\Gamma^\perp(z,R)) \sim \dist(x,\Gamma) = \dist(R,\Gamma).
\]
The claim follows.
\end{proof}

The second diameter estimate is comparability of the diameter and the distance. 


\begin{lem}
\label{lemma:Qi2}
Let $\Gamma \in \Jordan(p,C_0)$. Then, for each $R\in \sQ^\refi_\Gamma$, we have
\[
\diam (R) \sim \dist(R,\Gamma),
\]
where the constants depend only on $p$ and $C_0$.
\end{lem}

We need the following auxiliary lemma. 
\begin{lem}\label{lemma:R-intersection}
	Let $\Gamma\in\Jordan(p,C_0)$. Then, for each $R\in \sQ_\Gamma^\refi$ and $z\in S_{\Gamma}(R)$, we have
	\[
	\ell(\Gamma^\perp(z,\, R)) \sim \dist(R,\Gamma),
	\]
	where the constants depend only on $p$ and $C_0$. 
\end{lem}
\begin{proof}
	Since each $R$ is balanced and $z\in S_{\Gamma}(R)$, 
	$$\ell(\Gamma^\perp(z,\, R)) \sim \diam( S_{\Gamma}(R))$$
	and then the claim is an immediate consequence of Lemma~\ref{lemma:Qi1}. 
\end{proof}

\begin{proof}[Proof of Lemma~\ref{lemma:Qi2}]
One direction follows from the facts that $\Omega_\Gamma$ is a John domain and that $R$ has small shadow on $\Gamma$. Indeed, by Lemma \ref{lemma:shadow_partition}, we have that 
\[
\ell(\Gamma^\perp(z,R)) \sim \diam S_\Gamma(R)
\]
for each $z\in S_\Gamma(R)$. Thus, by Lemma \ref{cor:JJ-corollary}, 
\[
\dist(R, \Gamma) \ls \diam (R).
\]

Towards the other direction, let $x_1,x_2\in R$ be points for which $|x_1-x_2| =\diam R$ and, for $i=1,2$, let $z_i = S_\Gamma(x_i)$. Then
\begin{align*}
\diam(R) &= |x_1-x_2| \le |x_1 - z_1| + |z_1 - z_2| + |z_2 - x_2| \\
&\le \ell(\Gamma^\perp(x_1)) + \diam(S_\Gamma(R)) + \ell(\Gamma^\perp(x_2)).
\end{align*}
By Lemma \ref{lemma:Qi1}, $\diam(S_\Gamma(R))  \ls \dist(R,\Gamma)$. Thus it suffices to show that $\ell(\Gamma^\perp(x_i)) \ls \dist(R,\Gamma)$.

Let $Q\in\sQ_\Gamma$ be the parent of $R$. Observe that $\Gamma^\perp(x_i,R) = \Gamma^\perp(x_i,Q)$. Since $S_\Gamma(R)$ is $R$-balanced by construction, we have, by Lemmas \ref{lemma:R-intersection} and \ref{lemma:shadow_partition}, that
\begin{align*}
\ell(\Gamma^{\perp}(x_i)) &\le \ell(\Gamma^{\perp}(S_\Gamma(x_i),R)) + \ell(\Gamma^{\perp}(x_i)\cap R) \\
&\ls \ell(\Gamma^\perp(S_{\Gamma}(x_i), Q)) + \dist(R,\Gamma) \\
&\ls \diam(S_\Gamma(R)) + \dist(R,\Gamma) \ls \dist(R,\Gamma).
\end{align*}
The claim follows.
\end{proof}

Similarly as for rectangles in $\sQ_\Gamma$, the intersection length of hyperbolic geodesics with rectangles in $\sQ_\Gamma^\refi$ is in terms of a power law. We record this in the following form.

\begin{cor}
\label{cor:size of Qi}
Let $\Gamma\in\Jordan(p,C_0)$. Then, for each $R\in \sQ_\Gamma^\refi$, $z\in S_{\Gamma}(R)$ and $\tilde Q\in \tilde \sQ_\Gamma$ satisfying $R\subset h_{\Gamma}(\tilde Q)$,  we have
\[
\ell(\Gamma^\perp(z)\cap R)\sim \left(\frac{\dist(R,\Gamma)}{\dist(\tilde Q,\Gamma)}\right)^{p-1}\dist(\tilde Q,\Gamma),
\]
where the constants depend only on $p$ and $C_0$. 
\end{cor}
\begin{proof}
By Lemma \ref{lemma:Q1}, we have that 
$$\ell(\Gamma^\perp(z)\cap R) = \ell(\Gamma^\perp(z)\cap Q)\sim \dist(\Gamma^\perp(z)\cap h_\Gamma(\tilde Q), \Gamma)^{p-1} \dist(\tilde Q, \Gamma)^{2-p}.$$
Then via \eqref{represent dist} and Lemma~\ref{comp GM}, we further have
\begin{align*}
&\dist(\Gamma^\perp(z)\cap h_\Gamma(\tilde Q), \Gamma)^{p-1} \dist(\tilde Q, \Gamma)^{2-p}\\
=&\dist(\Gamma^\perp(z)\cap R, \Gamma)^{p-1} \dist(\tilde Q, \Gamma)^{2-p}
\sim \ell(\Gamma^\perp(z,\,R))^{p-1} \dist(\tilde Q, \Gamma)^{2-p}.
\end{align*} 
Now applying Lemma~\ref{lemma:R-intersection}, we eventually obtain 
\begin{align*}
\ell(\Gamma^\perp(z)\cap R) \sim \ell(\Gamma^\perp(z,\,R))^{p-1} \dist(\tilde Q, \Gamma)^{2-p}\sim \dist(R,\,\Gamma))^{p-1} \dist(\tilde Q, \Gamma)^{2-p}.
\end{align*}
The claim follows.
\end{proof}

\subsection{Number of large children}

Due to the properties of the initial reflection $h_\Gamma$ the rectangles in $\sQ_\Gamma^\refi$ are large compared to their grandparents. We formulate this, at this, point as follows.

\begin{lem}
\label{lemma:R-is-large}
Let $\Gamma \in \Jordan(p,C_0)$. Then, for each $R\in \sQ_\Gamma^\refi$, 
\[
\diam(S_\Gamma(R)) \gs \dist(\tilde Q, \Gamma),
\]
where $\tilde Q$ is the grandparent of $R$ and the constants depend only on $p$ and $C_0$.
\end{lem}

\begin{proof}
Let $Q=h_\Gamma(\tilde Q)$ be the parent of $R$. Let $z\in S_\Gamma(R)$. Since $\sQ_\Gamma^\refi$ is balanced, we have, by \eqref{smallest length} and Corollary \ref{cor:tilde_Q_distance}, that 
\[
\diam(S_\Gamma(R)) \sim \ell(\Gamma^\perp(z,R)) = \ell(\Gamma^\perp(z,Q)) \sim \ell(\tilde \Gamma^\perp(z,\tilde Q)) \gs \dist(\tilde Q, \Gamma).
\]
The claim follows.
\end{proof}

An immediate version of this lemma is the following estimate. 

\begin{cor}\label{R upper bound}
Let $\Gamma\in\Jordan(p,C_0)$. Then, for each $R\in \sQ_\Gamma^\refi$, $z\in S_{\Gamma}(R)$ and $\tilde Q\in \tilde \sQ_\Gamma$ satisfying $R\subset h_{\Gamma}(\tilde Q)$,  we have
\[
\ell(\Gamma^\perp(z)\cap R)\ls  \dist(R,\Gamma),
\]
where the constant depends only on $p$ and $C_0$. 
\end{cor}
\begin{proof}
	By Lemma~\ref{lemma:R-is-large} and Lemma~\ref{lemma:Qi1}, we have
	$$\dist(\tilde Q, \Gamma)\ls \ell(\Gamma^\perp(z)\cap R)\sim \dist(R,\,\Gamma).$$
	This together with Corollary~\ref{cor:size of Qi} and $p<2$ gives the desired estimate.  
\end{proof}

Although there is no upper bound for the number of children one rectangle in $\sQ_\Gamma$ can have, we obtain a uniform upper bound for the number of children of fixed relative size. For the statement, we give the following definitions.

\begin{defn}
A rectangle $R\in \sQ_\Gamma^\refi$ is \emph{$k$-large for $k\in \Z$} if  
\[
2^k \le \frac{\diam(S_\Gamma(R))}{\dist(\tilde Q,\Gamma)} < 2^{k+1},
\]
where $\tilde Q$ is the grandparent of $R$. For $Q\in \sQ_\Gamma$, we denote $\sC_k(Q) \subset \sQ_\Gamma^\refi$ the family of $k$-large children of $Q$.
\end{defn}

\begin{prop}
\label{prop:back-to-life}
Let $\Gamma \in \Jordan(p,C_0)$. Then there exists $C_\# = C_\#(p,C_0)\ge 1$ having the property that, for each $Q\in \sQ_\Gamma$, the number of $k$-large children of $Q$ is at most $C_\#$. 
\end{prop}

The proof is by volume counting in a given scale and follows from estimates for Whitney-type sets, although rectangles are not Whitney-type sets. We record this estimate as a lemma for separate use.

\begin{lem}
\label{lemma:R_and_ball}
Let $\Gamma \in \Jordan(p,C_0)$. Then there exists $\lambda^\refi=\lambda^\refi(p,C_0)\ge 1$ with the following property. For each $R\in \sQ_\Gamma^\refi$, there exists a  $\lambda^\refi$-Whitney-type set $B_R \subset \Omega_\Gamma$, meeting $R$ and satisfying $S_\Gamma(B_R) \subset S_\Gamma(R)$.
\end{lem}

\begin{proof}
Let $z_1,\,z_2$ be the end points of $S_\Gamma(R)$.
As $R$ is balanced, we have
$$\ell(\Gamma^\perp(z_1,\,R))\sim \diam(S_\Gamma(R)) \sim  \ell(\Gamma^\perp(z_2,\,R)).$$
Thus by Corollary~\ref{points} we have
\begin{equation}\label{exterior of disk}
\ell(\varphi_{\Gamma}^{-1}(\Gamma^\perp(z_1,\,R)))\sim \diam(\varphi_{\Gamma}^{-1}(S_\Gamma(R))) \sim  \ell(\varphi_{\Gamma}^{-1}(\Gamma^\perp(z_2,\,R))).
\end{equation}

Let $B$ be the ball which is tangent to the radial segment $\varphi_{\Gamma}^{-1}(\Gamma^\perp(z_1))$ at its end point in  $\mathbb C\setminus \overline{\mathbb D},$ and whose radius is 
$$\min\left\{ \frac 1 2 \ell(\varphi_{\Gamma}^{-1}(\Gamma^\perp(z_1))), \diam(\varphi_{\Gamma}^{-1}(S_\Gamma(R)))  \right\}.$$
Then since $\varphi_{\Gamma}^{-1}(\Gamma^\perp(z_i)), i=1,\,2$ are radial segments, by the geometry of the exterior of the unit disk and \eqref{exterior of disk}, we have
$$\diam(B)\sim \dist(B,\,\mathbb D), \quad B\cap \varphi_{\Gamma}^{-1}(R)\neq \emptyset.$$
Moreover, the shadow of $B$ in $\mathbb C\setminus \overline{\mathbb D}$ is contained in $\varphi_{\Gamma}^{-1}(S_\Gamma(R)).$
Then by Lemma~\ref{whitney preserving}, $\varphi_{\Gamma}(B)$ is a $\lambda^{\refi}$-Whitney-type set meeting $R,$ and the shadow of $\varphi_{\Gamma}(B)$ is contained in $S_\Gamma(R)$, where $\lambda^{\refi}=\lambda^{\refi}(p,\,C_0)$. 
\end{proof}

\begin{proof}[Proof of Proposition \ref{prop:back-to-life}]
Let $k\in \Z$, $Q\in \sQ_\Gamma$, and fix $\tilde Q\in \tilde \sQ_\Gamma$ for which $h_\Gamma(\tilde Q) = Q$. 

For each $R\in \sC_k(Q)$, let $B_R \subset \Omega_\Gamma$ be a $\lambda^\refi$-Whitney-type set intersecting $R$ given by Lemma~\ref{lemma:R_and_ball}. By Lemma \ref{lemma:Qi1}, $\diam(S_\Gamma(R)) \sim \dist(R,\Gamma)$. Since $\diam(B_R) \sim \dist(R,\Gamma)$, there exists $\lambda'=\lambda'(p,C_0)\ge 1$ for which $\diam(B_R) \ge 2^k \dist(\tilde Q, \Gamma)/\lambda'$.

Let $\tilde x\in \tilde Q$. We show that there exists $\mu^\refi=\mu^\refi(p,C_0)$ for which each $R\in \sC_k(Q)$ satisfies
\[
B_R \subset B(\tilde x, 2^k \mu^\refi \dist(\tilde Q, \Gamma))
\]
for each $R\in \sC_k(Q)$. Then $\# \sC_k(Q) \le \mu^\refi$, since the Whitney-type sets $B_R$ are mutually disjoint; notice that the shadows of distinct $R,R'$ can only
intersect at end points and recall that the shadow of $B_R$ is contained in the shadow of $R.$

It suffices to show that there exists $\mu^\refi=\mu^\refi(p,C_0)$ for which $R \subset B(\tilde x, 2^k \mu^\refi \dist(\tilde Q, \Gamma))$ for each $R\in \sC_k(Q)$. Let $y\in R$ and $z = S_\Gamma(y)$. Then  by Corollary~\ref{R upper bound}, \eqref{comparable length} and the facts that $\tilde Q$ is of Whitney-type and that $R$ is balanced, 
\begin{align*}
|y-\tilde x| &\le |y-z| + |z - \tilde x| 
\le \ell(\Gamma^\perp(y)) + \ell(\tilde \Gamma^\perp(z),\tilde Q) + \diam(\tilde Q) \\
&\ls \ell(\Gamma^\perp(z)\cap R) + \ell(\Gamma^\perp(z),Q) + \ell(\Gamma^\perp(z), Q) + \dist(\tilde Q,\Gamma) \\
&\ls \dist(R,\Gamma) + \ell(\Gamma^\perp(z),Q) + \ell(\tilde \Gamma^\perp(z),\tilde Q) \\
&\ls \ell(\Gamma^\perp(z, Q)) = \ell(\Gamma^\perp(z,R)) \sim \diam(S_\Gamma(R)).
\end{align*}
The claim follows.
\end{proof}


\subsection{Neighboring rectangles in $\sQ^\refi_\Gamma$}

We now record, as corollaries, some properties of pairs of intersecting 
rectangles in $\sQ^\refi_\Gamma$. We call such rectangles either \emph{adjacent} or \emph{neighbors}.

As the first property we record an almost immediate lemma according to which shadows of adjacent rectangles of $\sQ_\Gamma^\refi$ have comparable diameters. Note that the same holds also for diameters of the adjacent rectangles themselves.

\begin{lem}
\label{lemma:good partition}
Let $\Gamma\in \Jordan(p,C_0)$ and let $R$ and $R'$ be neighbors in $\sQ_\Gamma^\refi$, i.e.\;$R\cap R'\ne \emptyset$. Then there exists a constant $C=C(p,C_0)>0$ for which 
\[
\frac{1}{C} \le \frac{\diam(S_\Gamma(R))}{\diam(S_\Gamma(R'))} \le C. 
\] 
\end{lem} 

\begin{rem}
\label{rmk:good partition}
Notice that, by Lemmas \ref{lemma:good partition}, \ref{lemma:Qi1} and \ref{lemma:Qi2}, the diameters of intersecting rectangles are comparable, quantitatively.
\end{rem}
 
\begin{proof}[Proof of Lemma \ref{lemma:good partition}]
Since rectangles $R$ and $R'$ have non-empty intersection, we may fix $z\in S_\Gamma(R)\cap S_\Gamma(R')$. Since the partition $\sQ_\Gamma^\refi$ is balanced, we have that $\diam(S_\Gamma(R)) \sim \ell(\Gamma^\perp(z,R))$ and $\diam(S_\Gamma(R')) \sim \ell(\Gamma^\perp(z,R'))$. Thus
\[
\frac{\diam S_\Gamma(R)}{\diam S_\Gamma(R')} \ls \frac{\ell(\Gamma^\perp(z,R))}{\ell(\Gamma^\perp(z,R'))} = \frac{\ell(\Gamma^\perp(z,R))}{\ell(\Gamma^\perp(z,R))} =1.
\]
The claim now follows by symmetry.
\end{proof}

Lemma \ref{lemma:good partition} gives a corollary on the number of neighbors of a rectangle in $\sQ_\Gamma^\refi$.

\begin{cor}
\label{cor:finite_shadow_neighbors}
Let $\Gamma\in \Jordan(p,C_0)$. Then there exists a constant $C_\#=C_\#(p,C_0)$ having the property that, for each $R\in \sQ_\Gamma$, we have
\[
\#\{R' \in \sQ_\Gamma^\refi \colon R'\cap R \neq \emptyset\} \le C_\#.
\]
\end{cor}
\begin{proof}
Let $R\in \sQ_\Gamma^\refi$ and let $Q \in \sQ_\Gamma$ be the parent of $R$ in $\sQ_\Gamma$. 

Since $Q$ has at most $8$ neighbors in $\sQ_\Gamma$, it suffices to show that each neighbor $Q'$ of $Q $ in $\sQ_\Gamma$ contains uniformly bounded number of rectangles $R'$ of $\sQ_\Gamma^\refi$, which meet $R$ and satisfy $S_\Gamma(R') \subset S_\Gamma(R)$. 

Let $Q'$ be a neighbor of $Q$ in $\sQ_\Gamma$ and $R'\subset Q'$ an rectangle of $\sQ_\Gamma^\refi$ meeting $R$. Then, by Lemma~\ref{lemma:good partition}, we have that
\[
\diam(S_\Gamma(R))  \sim \diam(S_\Gamma(R')),
\]
where the constants depend only on $p$ and $C_0$. Thus the number rectangles $R'\in\sQ_\Gamma^\refi$, which meet $R$ and satisfy $S_\Gamma(R') \subset S_\Gamma(R)$, is uniformly bounded by the doubling property of $\Gamma$, that is, Corollary \ref{cor:finite division}. The claim follows.
\end{proof}


\section{Step 4: Refined partition $\tilde{\sQ}_\Gamma^\refi$}

In this section we reflect the partition $\sQ_\Gamma^\refi$ into the 
domain $\tilde \Omega_\Gamma$. This reflection is not topological and, in particular, the partition $\tilde \sQ_\Gamma^\refi$ is not the partition $\{ h_\Gamma^{-1}(R) \colon R\in \sQ_\Gamma^\refi\}$ of $\tilde A$. Instead of this, we subdivide elements of $\tilde \sQ_\Gamma$ according to the refinement $\sQ_\Gamma^\refi$ of $\sQ_\Gamma$ and merely reflecting the relative metric properties of $\sQ_\Gamma^\refi$. The immediate outcome will be that rectangles in $\tilde \sQ_\Gamma^\refi$ are in a natural correspondence with rectangles in $\sQ_\Gamma^\refi$. From the metric estimates we also deduce that rectangles in $\sQ_\Gamma^\refi$ are uniformly bilipschitz to Euclidean rectangles.

Technically, the partition $\tilde \sQ_\Gamma^\refi$ is obtained as an interpolation of two partitions, geometric and combinatorial, on horizontal edges of rectangles in $\tilde \sQ_\Gamma$. Since the construction of these partitions spans over the rest of this section, we assume now the existence of these partitions and give the synthetic construction of the partition $\tilde \sQ_\Gamma^\refi$ in terms of the properties of our geometric and combinatorial edge partitions.

The organization of this section is as follows. We first discuss horizontal edges of $\tilde Q\in \tilde \sQ_\Gamma^\refi$ in $\Omega_\Gamma$ and initial partitions on these edges. Then we return to discuss the initial horizontal edges in $\tilde \Omega_\Gamma$ and their geometric partitions. After these preliminaries, we define geometric partitions on horizontal edges in $\tilde \Omega_\Gamma$ (Section \ref{sec:GP-HE}) and combinatorial partitions (Section \ref{sec:CP-HE}). The bilipschitz equivalence of these partitions is proven in Section \ref{sec:GP-CP} and the partition $\tilde \sQ_\Gamma^\refi$ is constructed in Section \ref{sec:sQ_Gamma-refi}

\subsection{Horizontal edges}

We first define  upper and lower horizontal edges of rectangles in $\tilde \sQ_\Gamma$. Heuristically, the upper horizontal edge of $\tilde Q$ is the edge in which the hyperbolic rays from $S_\Gamma(\tilde Q)$ meet $\tilde Q$. The lower horizontal edge is the edge, where the same hyperbolic geodesics leave $\tilde Q$. More formally, we have the following definitions.

Let $\tilde Q\in \tilde \sQ_\Gamma$. We say that the arc 
\[
\horz(\tilde Q) = \{ \tilde x \in \tilde Q \colon \tilde \Gamma^\perp(x) = \tilde\Gamma^\perp(S_\Gamma(\tilde x);\tilde Q)\}
\]
is the \emph{upper horizontal edge of $\tilde Q$}. Similarly, we say that the arc
\[
\horz^\flat(\tilde Q) = \{ \tilde x\in \tilde Q \colon \tilde \Gamma^\perp(\tilde x)\cap \tilde Q = \tilde \Gamma^\perp(S_\Gamma(\tilde x)) \cap \tilde Q\}
\]
is the \emph{lower horizontal edge of $\tilde Q$}. 

The upper and lower horizontal edges $\horz(Q)$ and $\horz^\flat(Q)$ of a rectangle $Q\in \sQ_\Gamma$ are defined analogously. Notice that, we may equivalently define $\horz(Q) = h_\Gamma(\horz(h_\Gamma^{-1}(Q)))$ and $\horz^\flat(Q) = h_\Gamma(\horz^\flat(h_\Gamma^{-1}(Q)))$.

For $\tilde Q\in \tilde \sQ_\Gamma$, the upper horizontal edge $\horz(\tilde Q)$ has the property that there exists a (unique) pair $\tilde Q^+$ and $\tilde Q^-$ of rectangles in $\tilde \sQ_\Gamma$ for which 
\[
\tilde Q\cap \tilde Q^+ \cup \tilde Q^- = \horz(\tilde Q).
\]
We call $\tilde Q^+$ and $\tilde Q^-$ the \emph{upper neighbors of $\tilde Q$}. We further have that 
\[
\horz(\tilde Q) = \horz^\flat(\tilde Q^+) \cup \horz^\flat(\tilde Q^-).
\]

Rectangles in $\sQ_\Gamma$ have an analogous property. Also in this case, we call the unique pair the \emph{upper neighbors}. The same property holds, naturally, also for rectangles in $\sQ_\Gamma$.

\subsubsection{Geometry of horizontal edges $\horz(\tilde Q)$}

Since the rectangles in $\tilde \sQ_\Gamma$ are uniformly bilipschitz, up to the change of scale, to rectangles in $\sP$, and hence to the unit square, we may consider length data of the horizontal edges in place of diameters. We record two lemmas, which will be used to reparametrize the edges.

\begin{lem} 
\label{lemma:horz-tilde-Q-bilip}
Let $\Gamma \in \Jordan(p,C_0)$. Then there exists $L=L(p,C_0)$ for the following. For each $\tilde Q\in \tilde \sQ_\Gamma$, there exists $L$-bilipschitz homeomorphisms $\pi_{\tilde Q} \colon \horz(\tilde Q) \to [0,\diam(\horz(\tilde Q))]$ and $\pi^\flat_{\tilde Q} \colon \horz^\flat(\tilde Q) \to [0,\diam(\horz^\flat(\tilde Q))]$.
\end{lem}

\begin{proof}
By Lemma \ref{lemma:tilde_sQ_Whitney1}, $\tilde Q$ is Whitney-type with a constant depending only on $p$ and $C_0$. Let $P=\tilde \varphi_\Gamma^{-1}(\tilde Q)$. Then, by Lemma \ref{lemma:biLip}, $\tilde \varphi_\Gamma|_P \colon P \to \tilde Q$ is bilipschitz, with a constant depending only on $p$ and $C_0$, up to scaling.  Thus $\horz(\tilde Q)$ is a curve having length comparable to $\diam(\horz(\tilde Q)$, where the constants depend only on $p$ and $C_0$. The other case is analogous. The claim follows.
\end{proof}

The same argument also gives the bilipschitz equivalence of lower and upper horizontal edges. Since the proof is analogous, we omit the details.

\begin{lem}
\label{lemma:upper-lower-edge}
Let $\Gamma \in \Jordan(p,C_0)$ and $\tilde Q\in \tilde \sQ_\Gamma$. Then there exist $L$-bilipschitz homeomorphisms $\tau^\flat_{\tilde Q^\pm} \colon \horz(\tilde Q^\pm) \to \horz_\flat(\tilde Q^\pm)$, where $L=L(p,C_0)\ge 1$. \qed
\end{lem}

\subsection{Initial partitions on horizontal edges}
\label{sec:initial-edge-partitions}

Let $\tilde Q\in \tilde \sQ_\Gamma$. We call the partition
\[
\tilde \Sigma(\tilde Q) = \{ h_\Gamma^{-1}(R) \cap \horz(\tilde Q) \colon R\in \sC(h_\Gamma(\tilde Q))\}
\]
the \emph{initial partition of $\horz(\tilde Q)$} and similarly, the partition
\[
\tilde \Sigma^\flat(\tilde Q) = \{ h_\Gamma^{-1}(R) \cap \horz^\flat(\tilde Q) \colon R\in \sC(h_\Gamma(\tilde Q))\}
\]
the \emph{initial partition of $\horz^\flat(\tilde Q)$}; recall that $\sC(Q)$ is the collection of children of $Q\in \sQ_\Gamma$.

On the horizontal edge $\horz(\tilde Q)$ we also have a competing partition, 
\[
\tilde \Sigma^\pm(\tilde Q) = \{ h_\Gamma^{-1}(R) \cap \horz(\tilde Q) \colon R \in \sC(\tilde Q^+)\cup \sC(\tilde Q^-)\} = \Sigma^\flat(\tilde Q^+) \cup \Sigma^\flat(\tilde Q^-),
\]
where $\tilde Q^+$ and $\tilde Q^-$ are the upper neighbors of $\tilde Q$.

\begin{rem}
The partitions $\tilde \Sigma_\comb(\tilde Q)$ of $\horz(\tilde Q)$ and
$\tilde \Sigma_\geom(\tilde Q)$ of $\horz^\flat(\tilde Q)$ stem from these partitions. Heuristically, $\tilde \Sigma_\geom(\tilde Q)$ is a reparametrization of elements of $\tilde \Sigma^\flat(\tilde Q)$ and $\tilde \Sigma_\comb(\tilde Q)$ is a double reparametrization of $\tilde \Sigma(\tilde Q)$ taking partitions $\tilde \Sigma_\geom(\tilde Q^+)$ and $\tilde \Sigma_\geom(\tilde Q^-)$ into account. The parameter $r$ in Theorem \ref{weak main thm} plays a role in the construction of these partitions.
\end{rem}

For $Q\in \sQ_\Gamma$, the initial partitions $\Sigma(Q)$, $\Sigma^\flat(Q)$, and $\Sigma^\pm(Q)$ are defined analogously, that is,
\begin{align*}
\Sigma(Q) &= \{ R \cap \horz(Q) \colon R\in \sC(Q)\}, \\
\Sigma^\flat(Q) &= \{ R \cap \horz^\flat(Q) \colon R\in \sC(Q)\},\quad \text{and} \\
\Sigma^\pm(Q) &= \{ R \cap \horz(Q) \colon R\in \sC(Q^+)\cup \sC(Q^-)\}.
\end{align*}
Notice that we may also equivalently define these partitions as images of partitions $\tilde \Sigma(h_\Gamma^{-1}(Q))$, $\tilde \Sigma^\flat(h_\Gamma^{-1}(Q))$, and $\tilde\Sigma^\pm(h_\Gamma^{-1}(Q)$ under $h_\Gamma$.

\subsubsection{Geometry and combinatorics of initial partitions $\Sigma(Q)$}

Before moving the discussion to geometric partitions $\tilde \Sigma_\geom(\tilde Q)$ on horizontal edges of rectangles $\tilde Q\in \tilde \sQ_\Gamma$, we record two properties of the horizontal edges of rectangles in $\sQ_\Gamma^\refi$. The first is that the diameters of horizontal edges of rectangles in $\sQ_\Gamma^\refi$ are comparable to the diameters of their shadows.

\begin{lem}\label{upper boundary diam}
Let $\Gamma \in \Jordan(p,C_0)$. Then, for each child $R\in  \sQ_\Gamma^\refi$ of $Q\in \sQ_\Gamma$, we have 
\[
\diam(R\cap \horz(Q)) \sim \diam(S_\Gamma(R)) \sim \diam(R\cap \horz^\flat(Q))
\]
where the constants depend only on $p$ and $C_0$.
\end{lem}
\begin{proof}
	Let $\varphi_{\Gamma}\colon \mathbb C\setminus \mathbb D \to \overline{\Omega_\Gamma}$ be an extended conformal map, and let $z\in S_\Gamma(R)\subset \Gamma$. 
As $R$ is balanced, we have 
$$\diam(S_\Gamma(R))\sim \ell(\Gamma^\perp(z,\,R)).$$
Then by Corollary~\ref{points}, 
	$$\ell(\varphi_{\Gamma}^{-1}(\Gamma^\perp(z,\,R)))\sim \diam(\varphi_{\Gamma}^{-1}(S_\Gamma(R)))$$
for any $z\in S_\Gamma(R)\subset \Gamma.$

Then since each $\varphi_{\Gamma}^{-1}(\Gamma^\perp(z,\,R))$ is a radial segment, we have
	$$\diam(\varphi_{\Gamma}^{-1}(R\cap \horz(Q)))\sim \diam(\varphi_{\Gamma}^{-1}(S_\Gamma(R))).$$
	Now by the quasisymmetry $\varphi_{\Gamma}$ again, we obtain
	$$\diam(R\cap \horz(Q)) \sim \diam(S_\Gamma(R)) $$
	 The second part follows from a similar argument; notice that Corollary~\ref{R upper bound} gives
	 $$\diam(S_\Gamma(R))\sim \ell(\Gamma^\perp(z,\,R\cap \horz^\flat(Q))).$$
\end{proof}

The second property we record stems from the observation that, \emph{a priori}, the arcs in $\Sigma(Q)$ and $\Sigma^\pm(Q)$ have no relation: an arc in $\Sigma^\pm(Q)$ may meet one or several arcs in $\Sigma(Q)$ and may be contained or may contain an arc in $\Sigma(Q)$. 

For each $Q\in \sQ_\Gamma$ and $\tau\in \Sigma^\pm(Q)$, we denote $\Sigma(Q;\tau)$ the family of all arcs in $\Sigma(Q)$ which meet $\tau$, that is
\[
\Sigma(Q;\tau) = \{\sigma \in \Sigma(Q) \colon \sigma \cap \mathrm{int}\,\tau \ne \emptyset\};
\]
where the interior of $\tau$ is understood with respect to $\horz(Q)$; recall that the edges we consider are closed.

We call 
\[
\Diam\left( \Sigma (Q;\tau) \right) = \sum_{\sigma \in \Sigma(Q;\tau)} \diam(\sigma \cap \tau)
\]
the \emph{total diameter of the collection $\Sigma(Q;\tau)$}.  

\begin{rem}
\label{rmk:C_sharp}
By Corollary \ref{cor:finite_shadow_neighbors}, each rectangle in $\sQ_\Gamma^\refi$ meets at most $C_\#$ rectangles in $\sQ_\Gamma^\refi$. The same bound holds also for the number of elements in $\Sigma(Q; \tau)$. Thus 
\begin{equation}
\label{eq:Diam-diam}
\diam(\tau) \le \Diam(\Sigma(Q; \tau)) \le C_\# \diam(\tau),
\end{equation}
where the lower bound follows from the triangle inequality. Also notice that, although arcs in $\Sigma(Q;\tau)$ are essentially mutually disjoint, the diameter is not an additive function and hence there is no equality.
\end{rem}

\subsection{Geometric partition of horizontal edges}
\label{sec:GP-HE}

As already mentioned, our staring point for this discussion on partitions on $\horz(\tilde Q)$ for $\tilde Q\in \tilde \sQ_\Gamma$ is the initial partition 
\[
\tilde \Sigma(\tilde Q) = \left\{h_\Gamma^{-1}(R)\cap \horz^\flat(\tilde Q) \colon R \in \sC(h_\Gamma(\tilde Q)) \right\}.
\]
defined in Section \ref{sec:initial-edge-partitions}.

To obtain the so-called geometric partition $\tilde \Sigma_\geom(\tilde Q)$ on $\horz^\flat(\tilde Q)$, we reparametrize the edge $\horz(\tilde Q)$ by a homeomorphism $\horz^\flat(\tilde Q)\to \horz^\flat(\tilde Q)$. From this point onwards, we take advantage of the parameter $1<r<p$ in Theorem \ref{weak main thm} in construction of partitions. The requested homeomorphism is provided by the following lemma.

\begin{lem}\label{chop interval}
Let $\Gamma\in \Jordan(p,C_0)$ and $r<p$. Then there exists $C_\horz=C_\horz(p,C_0,r)>0$ for the following. For each $\tilde Q\in \tilde \sQ_\Gamma$, there exists a homeomorphism $\tilde \kappa_{\tilde Q} \colon \horz^\flat(\tilde Q) \to \horz^\flat(\tilde Q)$ fixing the end points of $\horz^\flat(\tilde Q)$ and having the property that, for each $\tilde \sigma \in \tilde \Sigma^\flat(\tilde Q)$, 
\begin{equation}
\label{eq:geom-part-cond}
\frac{1}{C_\horz} \left( \frac{\dist(\tilde Q, \Gamma)}{\diam(h_\Gamma(\tilde \sigma))}\right)^{\frac{p-r}{r-1}} \le \frac{\diam(\tilde \kappa_{\tilde Q}(\tilde \sigma))}{\dist(\tilde Q,\Gamma)}  \le C_\horz \left( \frac{\dist(\tilde Q, \Gamma)}{\diam(h_\Gamma(\tilde \sigma))}\right)^{\frac{p-r}{r-1}}.
\end{equation}
\end{lem}

\begin{rem}
Notice that \eqref{eq:geom-part-cond} is equivalent to 
\[
\diam(\tilde \kappa_{\tilde Q}(\tilde \sigma)) \sim \dist(\tilde Q, \Gamma)^{\frac{p-1}{r-1}} \diam(h_\Gamma(\tilde \sigma))^{-\frac{p-r}{r-1}}.
\]
\end{rem}

We formalize the geometric partition now as follows.

\begin{defn}
For $\tilde Q\in \tilde \sQ_\Gamma$, a partition $\tilde \Sigma_\geom(\tilde Q)$ of $\horz(\tilde Q)$ is an \emph{$r$-geometric partition} if 
\[
\tilde \Sigma_\geom(\tilde Q) = \{ \tilde \kappa_{\tilde Q}(\tilde \sigma) \colon \tilde \sigma \in \tilde \Sigma^\flat(\tilde Q)\},
\]
where $\tilde \kappa_{\tilde Q} \colon \horz(\tilde Q) \to \horz(\tilde Q)$ is a homeomorphism as in Lemma \ref{chop interval}.
\end{defn}

\begin{proof}[Proof of Lemma \ref{chop interval}] 
By Lemma \ref{lemma:horz-tilde-Q-bilip}, the arc $\horz^\flat(\tilde Q)$ is uniformly bilipschitz equivalent to the interval $[0,\dist(\tilde Q, \Gamma)]$. Thus the claim follows from a simple piece-wise affine reparametrization of $\horz^\flat(\tilde Q)$ if we show that
\[
\sum_{\tilde \sigma \in \tilde \Sigma(\tilde Q)} \left( \frac{\dist(\tilde Q, \Gamma)}{\diam(h_\Gamma(\tilde \sigma))}\right)^{\frac{p-r}{r-1}} \sim 1,
\]
where constants depend only on $p$, $C_0$, and $r$.

We observe first, by definition of $\tilde \Sigma^\flat(\tilde Q)$ and Lemma \ref{upper boundary diam}, that 
\begin{align*}
\sum_{\tilde \sigma \in \tilde \Sigma(\tilde Q)} \left( \frac{\dist(\tilde Q, \Gamma)}{\diam(h_\Gamma(\tilde \sigma))}\right)^{\frac{p-r}{r-1}} 
&= \sum_{R \in \sQ_\Gamma^\refi|_Q} \left( \frac{\dist(\tilde Q, \Gamma)}{\diam(R\cap \horz^\flat(h_\Gamma(\tilde Q)))}\right)^{\frac{p-r}{r-1}} \\
&\sim \sum_{R \in \sQ_\Gamma^\refi|_Q} \left( \frac{\dist(\tilde Q, \Gamma)}{\diam(S_\Gamma(R))}\right)^{\frac{p-r}{r-1}}.
\end{align*}
By Proposition \ref{prop:back-to-life}, we have that
\begin{align*}
\sum_{R \in \sQ_\Gamma^\refi|_Q} \left( \frac{\dist(\tilde Q, \Gamma)}{\diam(S_\Gamma(R))}\right)^{\frac{p-r}{r-1}} &= \sum_{k\in \Z} \sum_{R\in \sC_k(Q)} \left( \frac{\dist(\tilde Q, \Gamma)}{\diam(S_\Gamma(R))}\right)^{\frac{p-r}{r-1}} \\
&\le \sum_{k\in \Z} \sum_{R\in \sC_k(Q)} \left( \frac{1}{2^k}\right)^{\frac{p-r}{r-1}} 
\ls 1,
\end{align*}
where constants depend only on $p$, $C_0$, and $r$. 

Towards the other direction, we observe that, by Corollary \ref{cor:tilde_Q_distance}, there exists $z_0\in S_\Gamma(Q)$ for which 
\[
\ell(\tilde \Gamma(z_0),\tilde Q) \ls \dist(\tilde Q, \Gamma).
\]
Let now $\tilde \sigma_0 \in \tilde \Sigma(\tilde Q)$ be such that $\tilde \Gamma(z)\cap \tilde \sigma_0 \ne \emptyset$ and let $R_0\in \sQ_\Gamma^\refi$ be a child of $Q$ for which $h_\Gamma(\tilde \sigma_0) = R_0 \cap \horz^\flat(h_\Gamma(\tilde Q))$. 
We have, by Lemmas \ref{lemma:Qi2} and \ref{lemma:Qi1}, that
\begin{align*}
\sum_{\tilde \sigma \in \tilde \Sigma(\tilde Q)} \left( \frac{\dist(\tilde Q, \Gamma)}{\diam(h_\Gamma(\tilde \sigma))}\right)^{\frac{p-r}{r-1}} 
&\ge \left( \frac{\dist(\tilde Q, \Gamma)}{\diam(R_0 \cap \horz^\flat(h_\Gamma(\tilde Q))}\right)^{\frac{p-r}{r-1}} \\
&\ge \left( \frac{\dist(\tilde Q, \Gamma)}{\diam(R_0)}\right)^{\frac{p-r}{r-1}} 
\sim \left( \frac{\dist(\tilde Q, \Gamma)}{\diam(S_\Gamma(R_0))}\right)^{\frac{p-r}{r-1}}.
\end{align*}
Since $\sQ_\Gamma^\refi$ is balanced, we have, by \eqref{comparable length} and the assumption on $z_0$, that  
\begin{align*}
\left( \frac{\dist(\tilde Q, \Gamma)}{\diam(S_\Gamma(R_0))}\right)^{\frac{p-r}{r-1}}
&\sim \left( \frac{\dist(\tilde Q, \Gamma)}{\ell(\tilde \Gamma^\perp(z_0),Q)}\right)^{\frac{p-r}{r-1}} 
\sim \left( \frac{\dist(\tilde Q, \Gamma)}{\ell(\tilde \Gamma^\perp(z_0),\tilde Q)}\right)^{\frac{p-r}{r-1}} 
\sim 1.
\end{align*}
The claim follows.
\end{proof}

\subsection{Combinatorial partitions of horizontal edges}
\label{sec:CP-HE}

We now discuss the so-called combinatorial partition $\tilde \Sigma_\comb(\tilde Q)$ of $\horz(\tilde Q)$ for $\tilde Q\in \tilde \sQ_\Gamma$. The idea is to redistribute the partition $\tilde \Sigma(\tilde Q)$ starting from partition $\tilde \Sigma^\pm(\tilde Q)$ defined in Section \ref{sec:initial-edge-partitions}.

This is done in two steps as follows. For the first step, we define a partition 
\[
\tilde \Sigma^\pm_\geom(\tilde Q) = \tilde \Sigma_\geom(\tilde Q^+) \cup \tilde \Sigma_\geom(\tilde Q^-).
\]
of $\horz(\tilde Q)$ and let $\tilde \kappa^\pm_{\tilde Q} \colon \horz(\tilde Q) \to \horz(\tilde Q)$ be the concatenation of homeomorphisms $\tilde \kappa_{\tilde Q^+} \colon \horz^\flat(\tilde Q^+) \to \horz^\flat(\tilde Q^+)$ and $\tilde \kappa_{\tilde Q^-} \colon \horz^\flat(\tilde Q^-) \to \horz^\flat(\tilde Q^-)$ associated to the partitions $\tilde \Sigma_\geom(\tilde Q^+)$ and $\tilde \Sigma_\geom(\tilde Q^-)$. Notice that, in practice, $\tilde \kappa^\pm_{\tilde Q}$ moves the end points of arcs in $\tilde \Sigma^\pm(\tilde Q)$ to obtain $\tilde \Sigma^\pm_\geom(\tilde Q)$. 

For notational convenience, we define also an auxiliary geometric partition
\[
\tilde \Sigma_\geom(\tilde Q) = \{ \tilde \kappa^\pm_{\tilde Q}(\tilde \tau) \colon \tau \in \tilde \Sigma(\tilde Q)\}
\]
of $\horz(\tilde Q)$; notice that $\horz(\tilde Q)$ is not a lower horizontal edge of any rectangle.

Now, in the second step, we move the end points of the arcs in $\tilde \Sigma(\tilde Q)$ inside arcs in $\tilde \Sigma^\pm_\geom(\tilde Q)$. We define, for each $\tilde \tau \in \tilde \Sigma^\pm_\geom(\tilde Q)$, 
\[
\tilde \Sigma_\geom(\tilde Q; \tilde \tau) = \{ \tilde \sigma \in \tilde \Sigma_\geom(\tilde Q) \colon \tilde \sigma \cap \mathrm{int}\tilde \tau \ne \emptyset\}.
\]
Let also
\[
h_{\tilde Q}^\geom = h_\Gamma \circ (\tilde \kappa^\pm_{\tilde Q})^{-1} \colon \horz(\tilde Q) \to \horz(h_\Gamma(\tilde Q)).
\]

\begin{lem}
\label{lemma:combinatorial-reparametrization}
Let $\Gamma \in \Jordan(p,C_0)$ and $r<p$. Let $\tilde Q\in \tilde \sQ_\Gamma$ and $Q=h_\Gamma(\tilde Q)$. Then there exists a homeomorphism $\tilde \kappa^\comb_{\tilde Q} \colon \horz(\tilde Q) \to \horz(\tilde Q)$ satisfying
\begin{equation}
\label{subdivision tilde sigma}
\frac{\ell(\tilde\kappa^\comb_{\tilde Q}(\tilde \sigma \cap \tilde \tau))}{\ell(\tilde \tau)} = \frac{\diam(h_{\tilde Q}^\geom(\tilde \sigma\cap \tilde \tau))}{\Diam(\Sigma(h_\Gamma(\tilde Q);h_{\tilde Q}^\geom(\tilde \tau)))}
\end{equation}
for each $\tilde \tau\in \tilde \Sigma^\pm_\geom(\tilde Q)$ and $\tilde \sigma \in \tilde \Sigma_\geom(\tilde Q; \tilde \tau)$, and moreover
\[
\tilde \kappa^\comb_{\tilde Q}|_{\tilde \tau} = \id
\]
if $\# \tilde \Sigma(\tilde Q;\tilde \tau) = 1$.
\end{lem}

\begin{proof}
Let $\tilde \tau \in \tilde\Sigma^\pm_\geom(\tilde Q)$. It suffices to consider the case that $\tilde \Sigma(\tilde Q;\tilde \tau)$ contains more than one member. Let $Q=h_\Gamma(\tilde Q)$ and $\tau = h_\Gamma(\tilde \tau)$. Since 
\[
\sum_{\tilde \sigma \in \tilde \Sigma(\tilde Q;\tilde \tau)} \frac{\diam(h_{\tilde Q}^\geom(\tilde \sigma \cap \tilde \tau))}{\Diam(\Sigma(h_\Gamma(\tilde Q);h_{\tilde Q}^\geom(\tilde \tau))}
= \sum_{\sigma \in \Sigma(Q; \tau)} \frac{\diam(\sigma \cap \tau)}{\Diam(\Sigma(Q;\tau)} = 1,
\]
we observe that the question is merely on the existence of a homeomorphism $\tilde \tau \to \tilde \tau$ which preserves the end points and is a constant speed reparametrization of $\tilde \sigma\cap \tilde \tau$ for each $\tilde \sigma$. The claim follows.
\end{proof}

\begin{defn}
A partition $\tilde \Sigma_\comb(\tilde Q)$ is a \emph{combinatorial partition of $\horz(\tilde Q)$} if 
\[
\tilde \Sigma_\comb(\tilde Q) = \{ (\tilde \kappa^\comb_{\tilde Q}\circ \tilde \kappa^\pm_\geom)(\tilde \sigma) \colon \tilde \sigma \in \tilde \Sigma(\tilde Q)\},
\]
where $\tilde \kappa^\comb_{\tilde Q} \colon \horz(\tilde Q) \to \horz(\tilde Q)$ is a homeomorphism as in Lemma \ref{lemma:combinatorial-reparametrization} and $\tilde \kappa^\pm_\geom\colon \horz(\tilde Q) \to \horz(\tilde Q)$ is a homeomorphism associated to $\tilde \Sigma^\pm_\geom(\tilde Q)$.
\end{defn}

\begin{rem}
We emphasize that, although the partition $\tilde \Sigma_\comb(\tilde Q)$ stems from partitions $\tilde \Sigma^\pm(\tilde Q)$ and $\tilde \Sigma^\pm_\geom(\tilde Q)$, the partition $\tilde \Sigma(\tilde Q)$ is a redistribution of $\tilde \Sigma(\tilde Q)$. In particular, $\tilde \Sigma_\comb(\tilde Q)$ and $\tilde \Sigma_\geom(\tilde Q)$ has the same number of elements, i.e.~the number of children of $h_\Gamma(\tilde Q)$. 
\end{rem}

\subsection{Uniform bilipschitz equivalence of edge partitions}
\label{sec:GP-CP}

In what follows, we denote for each $\tilde Q\in \tilde\sQ_\Gamma$ the reparametrizing homeomorphisms simply by 
\[
\tilde \kappa_\geom = \kappa_{\tilde Q} \colon \horz^\pm(\tilde Q) \to \horz^\pm(\tilde Q)
\]
and 
\[
\tilde \kappa_\comb = \tilde \kappa^\comb_{\tilde Q}\circ \tilde \kappa^\pm_\geom \colon \horz(\tilde Q) \to \horz(\tilde Q).
\]
Notice that, for each $\tilde Q$, these homeomorphisms induce bijections 
\[
\sC(\tilde Q) \to \tilde \Sigma_\geom(\tilde Q),\quad 
\tilde R \mapsto \tilde \kappa_\geom(\tilde R\cap \horz^\flat(\tilde Q)),
\]
and 
\[
\sC(\tilde Q) \to \tilde \Sigma_\comb(\tilde Q),\quad 
\tilde R \mapsto \tilde \kappa_\comb(\tilde R\cap \horz(\tilde Q)).
\]

The partitions $\tilde\Sigma_\geom(\tilde Q)$ and $\tilde\Sigma_\comb(\tilde Q)$ of $\horz^\flat(\tilde Q)$ and $\horz(\tilde Q)$ are bilipschitz equivalent, quantitatively. We state this as follows.

\begin{prop}\label{prop:homotopy}
Let $\Gamma\in \Jordan(p,C_0)$ and $r<p$. Then there exists a constant $C_{\mathrm{GC}}=C_{\mathrm{GC}}(p,C_0,r)$ with the property that, for each $\tilde Q\in \tilde \sQ_\Gamma$ and $\tilde R\in \sC(\tilde Q)$,
\begin{equation}
\label{eq:sigma-comparison}
\frac{1}{C_{\mathrm{GC}}} \ell(\tilde \kappa_\geom(\tilde R\cap \horz^\flat(\tilde Q))) 
\le \ell(\tilde \kappa_\comb(\tilde R \cap \horz(\tilde Q))) 
\le C_{\mathrm{GC}}\ell(\tilde \kappa_\geom(\tilde R\cap \horz^\flat(\tilde Q)))
\end{equation}
and
\begin{equation}
\label{eq:sigma-size}
\ell(\kappa_\geom(\tilde R\cap \horz^\flat(\tilde Q))) \sim \dist(\tilde Q,\Gamma)^{\frac{p-1}{r-1}} \dist(R,\Gamma)^{\frac{p-r}{r-1}}.
\end{equation}
\end{prop}

\begin{proof}
Let $\tilde \sigma = \tilde R\cap \horz(\tilde Q)$ and $\tilde \sigma^\flat = \tilde R \cap \horz^\flat(\tilde Q)$. Let also $R=h_\Gamma(\tilde R)$ and $Q=h_\Gamma(\tilde Q)$.

By Lemmas \ref{upper boundary diam} and \ref{lemma:Qi1}, we have 
\[
\diam(h_\Gamma(\tilde \sigma)) = \diam(R \cap \horz(Q)) \sim \diam(S_\Gamma(R)) \sim \dist(R,\Gamma)
\]
and
\[
\diam(h_\Gamma(\tilde \sigma^\flat)) = \diam(R \cap \horz^\flat(Q)) \sim \diam(S_\Gamma(R)) \sim \dist(R,\Gamma),
\]
where the constants depend only on $p$, $C_0$, and $r$. 

Since $\horz(\tilde Q)$ is uniformly bilipschitz equivalent to an interval by Lemma \ref{lemma:horz-tilde-Q-bilip}, we have, by the definition of $\kappa_\geom$ (see Lemma \ref{chop interval}), that
\begin{equation}
\label{eq:ell-k_geom}
\begin{split}
\ell(\kappa_\geom(\tilde \sigma^\flat)) &\sim \diam(\kappa_\geom(\tilde \sigma^\flat)) \\
&\sim \left( \frac{\dist(\tilde Q, \Gamma)}{\diam(h_\Gamma(\tilde \sigma^\flat))}\right)^{\frac{p-r}{r-1}} \dist(\tilde Q, \Gamma) \\
&\sim \dist(\tilde Q,\Gamma)^{\frac{p-1}{r-1}} \dist(R,\Gamma)^{\frac{p-r}{r-1}}.
\end{split}
\end{equation}
This proves \eqref{eq:sigma-size}. It remains to show the same estimate for $\ell(\kappa_\comb(\tilde \sigma))$.

\medskip
We begin with the proof of the upper bound \eqref{eq:sigma-comparison}. By Corollary, \ref{cor:finite_shadow_neighbors}, $\tilde \sigma$ meets at most $C_\#$ arcs in $\tilde \Sigma^\pm(\tilde Q)$, where $C_\#=C_\#(p,C_0)$. 

For each $\tilde \tau \in \tilde \Sigma^\pm(\tilde Q)$, let $R_{\tilde \tau}\in \sQ_\Gamma^\refi$ be the rectangle for which $\tilde \tau = h_\Gamma^{-1}(R_{\tilde \tau})\cap \horz(\tilde Q)$. By Lemma \ref{lemma:good partition}, $\diam(S_\Gamma(R_{\tilde \tau})) \sim \diam(S_\Gamma(R))$. Thus, by the definition of $\kappa^\pm_{\tilde Q}$ and Lemma \ref{upper boundary diam}, we have that 
\begin{equation}
\label{smp}
\begin{split}
\ell(\tilde \kappa^\pm_{\tilde Q}(\tilde \tau)) &\sim \dist(\tilde Q,\Gamma)^{\frac {p-1}{r-1}}\dist(  R_{\tilde \tau},\Gamma)^{\frac {r-p}{r-1}} \sim \dist(\tilde Q,\Gamma)^{\frac {p-1}{r-1}}\dist(R,\Gamma)^{\frac {r-p}{r-1}}.
\end{split}
\end{equation}

Thus, by \eqref{smp} and \eqref{eq:ell-k_geom}, we have that
\begin{equation}
\label{eq:comb-geom-upper}
\ell(\tilde \kappa_\comb(\tilde \sigma)) \le \sum_{\tilde \tau \cap \tilde \sigma \ne \emptyset} \ell(\tilde \kappa^\pm_{\tilde Q}(\tilde \tau)) \sim \ell(\tilde \kappa_\geom(\tilde \sigma^\flat)).
\end{equation}
This completes the proof of the upper bound in \ref{eq:sigma-comparison}.

For the lower bound in \eqref{eq:sigma-comparison}, we consider three cases.

\noindent
\emph{Case 1:} Suppose there exists $\tilde \tau \in \tilde \Sigma^\pm_\geom(\tilde Q)$ contained in $\tilde \sigma$, that is, $\tilde \Sigma_\geom(\tilde Q; \tilde \tau) = \{ \tilde \sigma\}$. Then $\kappa^\comb_{\tilde Q}$ is the identity on $\tilde \tau$ and $\tilde \kappa^\comb = \tilde \kappa^\pm_{\tilde Q}$ on $\tilde \tau$. Thus 
\[
\ell(\tilde \kappa_\comb(\tilde \sigma)) \ge \ell(\tilde \kappa_\comb(\tilde \tau)) = \ell(\tilde \kappa^\pm_{\tilde Q}(\tilde \tau)) \sim \ell(\tilde \kappa_\geom(\tilde \sigma^\flat))
\]
by \eqref{eq:ell-k_geom} and \eqref{smp}.

\medskip
\noindent
\emph{Case 2:} Suppose that there exists $\tilde \tau \in \tilde \Sigma^\pm_\geom(\tilde Q)$ for which $\tilde \sigma \subset \tilde \tau$. Let $\tilde \sigma_\geom = \tilde \kappa^\pm_{\tilde Q}(\tilde \sigma)$ and $\tilde \tau_\geom = \tilde \kappa^\pm_{\tilde Q}(\tilde \tau)$.

Then, by the construction of $\tilde \kappa^\comb_{\tilde Q}$ in Lemma \ref{lemma:combinatorial-reparametrization} and \eqref{eq:Diam-diam}, we have that
\begin{align*}
\ell(\tilde \kappa_\comb(\tilde \sigma)) &= \ell(\tilde \kappa^\pm_{\tilde Q}(\tilde \sigma_\geom)) = \ell(\tilde \kappa^\pm_{\tilde Q}(\tilde \sigma_\geom \cap \tilde \tau_\geom)) \\
&= \frac{\diam(h_{\tilde Q}^\geom(\tilde \sigma_\geom\cap \tilde \tau_\geom))}{\Diam(\Sigma(h_\Gamma(\tilde Q);h_{\tilde Q}^\geom(\tilde \tau_\geom)))} \ell(\tilde \tau_\geom) \\
&\sim \frac{\diam(h_{\tilde Q}^\geom(\tilde \sigma_\geom\cap \tilde \tau_\geom))}{\diam(h_{\tilde Q}^\geom(\tilde \tau_\geom))}\ell(\tilde \tau_\geom)\\
&= \frac{\diam(h_\Gamma(\tilde \sigma))}{\diam(h_\Gamma(\tilde \tau))}\ell(\tilde \tau_\geom) 
= \frac{\diam(h_\Gamma(\tilde \sigma))}{\diam(h_\Gamma(\tilde \tau))}\ell(\kappa^\pm_{\tilde Q}(\tilde \tau)).
\end{align*}
Since $h_\Gamma(\tilde \sigma)$ and $h_\Gamma(\tilde \tau)$ meet, they are neighbors in $\sQ_\Gamma^\refi$. Thus by applying Lemma \ref{upper boundary diam} twice and Lemma \ref{lemma:good partition}, we obtain that
\[
\diam(h_\Gamma(\tilde \sigma)) \sim \diam(S_\Gamma(h_\Gamma(\tilde \sigma))) \sim \diam(S_\Gamma(h_\Gamma(\tilde \tau))) \sim \diam(h_\Gamma(\tilde \tau)).
\]
Thus, by \eqref{smp} and \eqref{eq:ell-k_geom}, we have that
\begin{equation}
\label{eq:as-in-this}
\ell(\tilde \kappa_\comb(\tilde \sigma))  \sim \ell(\kappa^\pm_{\tilde Q}(\tilde \tau)) \sim  \dist(\tilde Q,\Gamma)^{\frac {p-1}{r-1}}\dist(R,\Gamma)^{\frac {r-p}{r-1}} \sim \ell(\kappa_\geom(\tilde \sigma^\flat)).
\end{equation}
This proves the uniform lower bound in this case.

\medskip
\noindent
\emph{Case 3:} Suppose finally that $\tilde \sigma$ neither contains nor is contained in an arc in $\tilde \Sigma^\pm(\tilde Q)$. Since $\horz(\tilde Q)$ is an arc, we conclude that there exists exactly two rectangles in $\tilde \tau\in \sQ_\Gamma^\refi$ which cover the interior of $\tilde \sigma$ in $\horz(\tilde Q)$. Moreover, one of them, say $\tilde\tau \in \tilde \Sigma^\flat(\tilde Q)$, covers more than half of $\tilde \sigma$ in the sense that $\diam(h_\Gamma(\tilde \sigma\cap \tilde \tau))\ge \diam(h_\Gamma(\tilde \sigma))/2$. 

We may now repeat the argument of Case 2 almost verbatim. Indeed, let again $\tilde \sigma_\geom=\tilde \kappa^\pm_{\tilde Q}(\tilde \sigma)$ and $\tilde \tau_\geom = \tilde \kappa^\pm_{\tilde Q}(\tilde \tau)$. Then
\begin{align*}
\ell(\tilde \kappa_\comb(\tilde \sigma)) 
&\sim \frac{\diam(h_{\tilde Q}^\pm(\tilde \sigma_\geom\cap \tilde \tau_\geom))}{\diam(h_{\tilde Q}^\pm(\tilde \tau_\geom))}\ell(\tilde \tau_\geom)\\
&= \frac{\diam(h_\Gamma(\tilde \sigma\cap \tilde \tau))}{\diam(h_\Gamma(\tilde \tau))}\ell(\kappa_\geom(\tilde \tau)) \\
&\ge \frac{1}{2}\frac{\diam(h_\Gamma(\tilde \sigma))}{\diam(h_\Gamma(\tilde \tau))}\ell(\kappa_\geom(\tilde \tau)) 
\sim \ell(\kappa_\geom(\tilde \tau)).
\end{align*}
Thus the comparability follows as in \eqref{eq:as-in-this}. This proves the last lower bound and completes the proof of the proposition.
\end{proof}

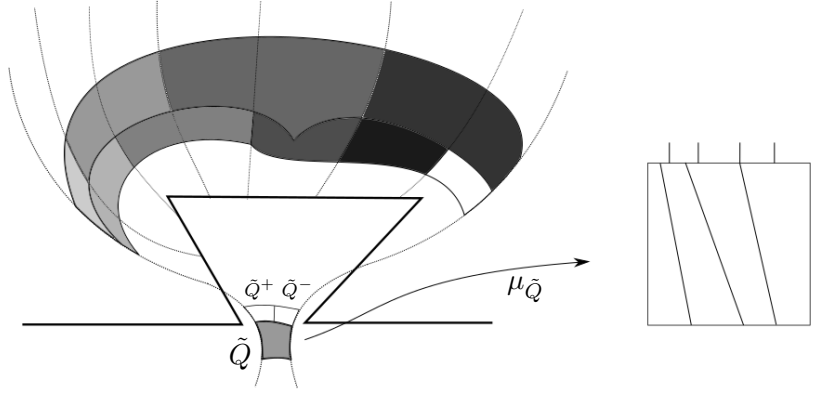
\begin{figure}[ht]
\begin{overpic}[scale=0.35,unit=1mm]{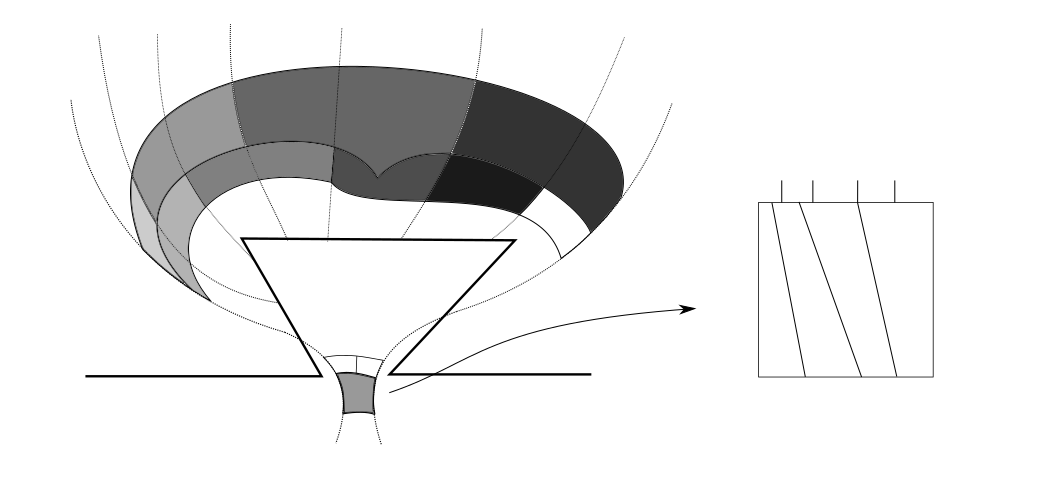} 
\put(38,9){$\tilde Q$}
\put(40,18){\tiny $\tilde Q^+$}
\put(45,18){\tiny $\tilde Q^-$}
\put(75,19){$\mu_{\tilde Q}$}
\end{overpic}
\caption{The decomposition of $\tilde Q$ is based on the geometric partition on its lower horizontal edge and the combinatorial partition of its upper horizontal edge, which is further based on the geometric partition of $\tilde Q^{\pm}$. This partition is obtained by passing to the image rectangle of $\tilde Q$ under 
the bilipschitz map $\mu_{\tilde Q}.$
}
\label{tilde_Q_decompose}
\end{figure}

\subsection{Interpolation refinement $\tilde{\sQ}_\Gamma^\refi$}
\label{sec:sQ_Gamma-refi}

\newcommand{\hull}{\mathrm{hull}}
\newcommand{\pr}{\mathrm{pr}}
\newcommand{\sA}{\mathsf{A}}
\newcommand{\sU}{\mathsf{U}}
\newcommand{\sL}{\mathsf{L}}

The partition $\tilde \sQ_\Gamma^\refi$ is given by interpolation of geometric and combinatorial partitions on horizontal edges of rectangles of $\tilde \sQ_\Gamma$. 

As a model case, consider a Euclidean rectangle $R$ and partitions $\sU=\{u_1,\ldots, u_m\}$ and $\sL=\{l_1,\ldots, l_m\}$ of opposite faces of $R$ into the same number of arcs. Then, by reordering the segments in $\sL$ if necessary, the convex hulls $\{ u_k \star l_k \colon k=1,\ldots, m\}$ form a partition of $R$. Recall that the \emph{convex hull (or the join) $A\star B$ of the subsets $A$ and $B$ in $\R^2$} is the set 
\[
A\star B = \{ tx+(1-t)y \in \R^2 \colon x\in A,\ y \in B,\ t\in [0,1]\}.
\]

Since the rectangles in $\tilde \sQ_\Gamma$ are uniformly bilipschitz to Euclidean rectangles by Lemma \ref{lemma:biLip}, we may use the geometric partitions on upper horizontal edges and combinatorial partitions on lower horizontal edges to obtain a similar partition for each rectangle in $\tilde \sQ_\Gamma$, quantitatively. More precisely, we define $\tilde \sQ_\Gamma^\refi$ as follows.

Let $L=L(p,C_0)\ge 1$ be the constant in Lemma \ref{lemma:biLip}, that is, for each $\tilde Q\in \tilde \sQ_\Gamma$, we may fix $\mu_{\tilde Q} \colon \tilde Q \to [0,\diam(\tilde Q)]^2$ a corner preserving $L$-bilipschitz homeomorphism for which $\mu_{\tilde Q}(\horz(\tilde Q)) = [0,\diam(\tilde Q)]\times \{\diam(\tilde Q)\}$ and $\mu_{\tilde Q}(\horz^\flat(\tilde Q)) = [0,\diam(\tilde Q)]\times \{0\}$. 

For each $\tilde R =h_\Gamma(\tilde R)$, where $R\in \sC(h_\Gamma(\tilde Q))$, let 
\[
H(\tilde R) = \mu_{\tilde Q}^{-1}\left( \mu_{\tilde Q}(\kappa_\comb(\tilde R \cap \horz(\tilde Q)) \star \mu_{\tilde Q}(\kappa_\geom(\tilde R \cap \horz^\flat(\tilde Q)))\right).
\]
Notice that, for each $\tilde Q\in \tilde \sQ_\Gamma$, the collection 
\[
\tilde \sQ_\Gamma^\refi(\tilde Q) = \{ H(\tilde R) \colon h_\Gamma(\tilde R)\in \sC(h_\Gamma(\tilde Q))\}
\]
is a partition of $\tilde Q$.
\begin{defn}
The partition 
\[
\tilde \sQ_\Gamma^\refi = \bigcup_{\tilde Q\in \tilde \sQ_\Gamma} \tilde \sQ_\Gamma^\refi(\tilde Q)
\]
of $\tilde A$ is the \emph{interpolation refinement of $\tilde \sQ_\Gamma$}. We call rectangles in $\tilde \sQ_\Gamma^\refi(\tilde Q)$ the children of $\tilde Q\in \tilde \sQ_\Gamma$ and denote $\sC(\tilde Q) = \tilde \sQ_\Gamma^\refi(\tilde Q)$.
\end{defn}
Notice that, in this definition, we tacitly ignore the dependence of the partition on the choice of bilipschitz maps $\mu_{\tilde Q}$.

\newcommand{\Rect}{\mathrm{Rect}}

By Proposition \ref{prop:homotopy}, the rectangles in $\tilde \sQ_\Gamma^\refi$ are uniformly bilipschitz to the Euclidean rectangles. We record this as the following corollary. For the statement, for each $\tilde Q\in \tilde \sQ_\Gamma$ and $\tilde R\in \sC(\tilde Q)$, let 
\[
\Rect_\Gamma(\tilde R) = [0,\,\dist(\tilde Q, \Gamma)^{\frac{p-1}{r-1}}\dist(R, \Gamma)^{\frac {r-p}{r-1}}]\times [0,\,\dist(\tilde Q, \Gamma)] \subset \R^2,
\]
where $R\in \sC(h_\Gamma(\tilde Q))$ is the unique rectangle for which $R\cap \horz^\flat(h_\Gamma(\tilde Q)) = h_\Gamma(\kappa_\geom^{-1}(\tilde R\cap \horz^\flat(\tilde Q)))$.

\begin{cor}
\label{cor:wz Qi}
Let $\Gamma \in \Jordan(p,C_0)$ and $1 < r < p$. Then there exists $L=L(p,C_0,r)\ge 1$ for the following. For each $\tilde R\in \tilde \sQ_\Gamma^\refi$, there exist an $L$-bilipschitz homeomorphism 
\[
\tilde \Xi_{\tilde R} \colon \tilde R \to \Rect_\Gamma(\tilde R),
\]
which maps the corners of $\tilde R$  to the corners of $\Rect_\Gamma(\tilde R)$.
\end{cor}

\subsubsection{Structure of $\tilde \sQ^\refi_\Gamma$}

As the definition of the model rectangle $\Rect_\Gamma(\tilde R)$ for $\tilde R\in \tilde \sQ_\Gamma^\refi$ suggests, there is a natural bijection $\iota^\refi_\Gamma\colon \tilde \sQ^\refi_\Gamma \to \sQ^\refi_\Gamma$ from the decomposition $\tilde \sQ^\refi_\Gamma$ of $\tilde A$ to the decomposition $\sQ^\refi_\Gamma$ of $A$ obtained as follows. Let $\iota_\Gamma \colon \tilde \sQ_\Gamma\to \sQ_\Gamma$ be the bijection $\tilde Q \mapsto h_\Gamma(Q)$. Then $\iota_\Gamma$ induces a graph isomorphism from the adjacency graph of $\tilde \sQ_\Gamma$ to the adjacency graph of $\sQ_\Gamma$. Further, $\iota_\Gamma$ refines to a graph isomorphism $\iota_\Gamma^\refi \colon \tilde \sQ_\Gamma^\refi \to \sQ_\Gamma^\refi$, where $\iota_\Gamma^\refi(\tilde R)\in \sQ_\Gamma^\refi$ is uniquely defined by the relation 
\[
h_\Gamma(\kappa_\geom^{-1}(\tilde R\cap \horz(\tilde Q))) = \iota_\Gamma^\refi(\tilde R) \cap \horz(h_\Gamma(\tilde Q)).
\]

Note that this graph isomorphism $\iota_\Gamma^\refi$ has the property that, for rectangles $R,R'\in \tilde \sQ_\Gamma^\refi$, images $\iota^\refi_\Gamma(\tilde R)$ and $\iota^\refi_\Gamma(\tilde R')$ meet in an edge (in a corner) if and only if rectangles $\tilde R$ and $\tilde R'$ meet in an edge (in a corner). Note also, that for each $\tilde Q\in \tilde \sQ^\refi_\Gamma$, $\iota^\refi_\Gamma(\tilde \sQ^\refi_\Gamma|_{\tilde Q}) = \sQ^\refi_\Gamma|_{Q}$, that is, $\iota^\refi_\Gamma$ maps children to children.

\section{Step 5: Lipschitz partition $\sW_{\Gamma}^\refi$ of $\hat A_\Gamma$}

As the final step before the construction of the reflection, we reshape the rectangles of the partition $\sQ_\Gamma^\refi$. The reason is that, whereas the vertical edges of $Q\in \sQ_\Gamma^\refi$ are subarcs of hyperbolic rays, we have no length control on the horizontal edges $\horz(Q)$ and $\horz^\flat(Q)$. The partition $\sW_\Gamma^\refi$ fixes this issue. Due to reshaping we need to pass from $A_\Gamma$ to another collar $\hat A_\Gamma$ of $\Gamma$ in $\Omega_\Gamma$. 

More precisely, we find a new partition $\sW_\Gamma^\refi$ with rectangles which are bilipschitzly Euclidean and close to the corresponding rectangles in $\sQ_\Gamma^\refi$ in the Hausdorff distance. 
Recall that the \emph{Hausdorff distance $\dist_H(A,A')$ of the sets $A$ and $A'$ in $\C$} is 
\[
\dist_{H}(A,A') = \max\{\,\sup_{x \in A} \inf_{y \in A'} d(x,y), \sup_{y \in A'} \inf_{x \in A} d(x,y)\}.
\]

The partition $\sW_\Gamma^\refi$ is given by the following proposition. For the statement, we define also, for $R\in \sQ_\Gamma^\refi$,
\[
\Rect_\Gamma( R)= [0,\,\dist(R,\,\Gamma)]  \times [0,\, \dist(R,\,\Gamma)^{p-1}  \dist(\tilde R,\,\Gamma)^{2-p}] \subset \C. 
\]
\begin{prop}
\label{prop:inside lip}
Let $\Gamma\in \Jordan(p,C_0)$ and $\epsilon\in (0,1/9)$. Then there exist $L_\epsilon=L_\epsilon(p,C_0,\epsilon)\ge 1$ and a domain $\hat A_\Gamma \subset \Omega_\Gamma$, for which $\Gamma$ is a boundary component of $\mathrm{cl}(\hat A_\Gamma)$, and a partition $\sW^{\refi}_\Gamma$ of $\hat A_\Gamma$ into rectangles having the following properties: 
\begin{enumerate}
\item there exists a graph isomorphism $\eta^\refi_\Gamma \colon \sQ_\Gamma^\refi \to \sW_\Gamma^\refi$ of adjacency graphs, 
\item for each $R\in \sQ^\refi_\Gamma$, the rectangle $W_R = \eta^\refi_\Gamma(R)$ satisfies the following properties:
\begin{enumerate}
\item $W_R$ has the same vertical edges as $R$,
\item $W_R \cap R \ne \emptyset$,
\item $\dist_{H}(W_R, R) < \varepsilon \dist(\tilde R,\Gamma)$, where $R=\iota_\Gamma^\refi(\tilde R)$, \label{item:need-to-mention}
\item there exists an $L_\epsilon$-bilipschitz map $W_R \to \Rect_\Gamma(R)$ preserving horizontal edges.
\end{enumerate}
\end{enumerate}
\end{prop}

The proof is based on a lemma on normalized families. 

\subsection{Normalized family}

\newcommand{\nbh}{\mathscr N}
\newcommand{\cN}{\mathcal N}

We define first the \emph{$\sQ_\Gamma^\refi$-neighborhood $\nbh^\refi(R)\subset \R^2$ of $R\in \sQ_\Gamma^\refi$} by
\[
\nbh^\refi(R)=\bigcup_{\substack{R'\cap R \neq \emptyset\\ R'\in \sQ^\refi_\Gamma}} R'. 
\]

Recall now that we have tacitly fixed in Section \ref{vitonen} a conformal map $\varphi_\Gamma \colon \D\setminus \{0\} \to \Omega_\Gamma$. By fixing a conformal map $\D \to \bH$, where $\bH=\R\times (0,\infty)$ is the upper half-plane model of the hyperbolic space, we obtain a conformal map $\phi \colon A_\Gamma \to \bH$, which maps hyperbolic rays $\Gamma(z)\cap A_\Gamma$ for $z\in \Gamma$ to vertical segments in $\bH$. We denote, for each $R\in \sQ_\Gamma^\refi$, $\phi_R^1 = \phi|_{\nbh^\refi(R)} \colon \nbh^\refi(R) \to \bH$ the restriction of $\phi$ to the $\sQ_\Gamma^\refi$-neighborhood of $R$. By Lemma \ref{lemma:biLip}, each $\phi_R^1$ is $L$-bilipschitz with an absolute constant $L$. Note that the image of $\phi_R^1$ is a finite union of topological rectangles whose vertical edges are vertical lines in $\bH$. 

For each $R\in \sQ_\Gamma^\refi$ and each $R'\in \sQ_\Gamma^\refi$ contained in $\nbh^\refi(R)$, let $x_{\phi_R(R')} \in \C$ be the Euclidean barycenter of $\phi_R^1(R')$; a point $x\in \C$ is a Euclidean bacycenter of a bounded set $A\subset \C$ if $x$ is the center of the minimal closed disk containing $A$.

To define normalized copies of topological rectangles in $\tilde \sQ_\Gamma^\refi$ in $\bH$, let first $C_\Box=C_\Box(p,C_0)$ be the comparability constant in Corollary \ref{cor:size of Qi}, that is, a constant satisfying
\[
\frac{1}{C_\Box} \le 
\frac{\ell(\Gamma^\perp(z)\cap h_{\Gamma}(\tilde Q))}{\dist(R,\,\Gamma)^{p-1}\dist(\tilde Q,\,\Gamma)^{2-p}} \le C_\Box
\]
for all $R\in \sQ_\Gamma^\refi$ and $\tilde Q\in \tilde \sQ_\Gamma$ satisfying $R\subset h_\Gamma(\tilde Q)$.

\begin{defn}
For each $R\in \sQ_\Gamma^\refi$, the \emph{$C_\Box$-normalized copy of $R$} is the topological rectangle
\[
R^\Box = (\phi_R^2 \circ \phi_R^1)(R) \subset \C,
\]
where $\phi_R^2 \colon \R^2 \to \R^2$ is the linear map
\[
(x,y) \mapsto  C_\Box \left( \frac{x-x_R}{\dist(R, \Gamma)}, \left(\frac{\dist(R,\Gamma)}{\dist(\tilde R, \Gamma)}\right)^{1-p} \frac{y-y_R}{\dist(\tilde R, \Gamma)}\right).
\]
where $R=\iota_\Gamma^\refi(\tilde R)$. The map $\phi_R = \phi_R^2 \circ \phi_R^1 \colon \nbh^\refi(R) \to \C$ is called a \emph{linear normalization map of $R$}.
\end{defn}

We denote by $\cN^\refi_\Gamma$ the collection of all normalized copies of topological rectangles in $\sQ^\refi_\Gamma$. These rectangles are `normalï¿½ in the following sense.

\begin{lem}
Let $\Gamma \in \Jordan(p,C_0)$. Then there exists a constant $c=c(p,C_0)>0$ having the property that, for each $R\in \sQ_\Gamma^\refi$, the $C_\Box$-normalized copy $R^\Box$ of $R$ satisfies
\begin{equation}
\label{eq:normalized-diam}
\frac{1}{c} \le \diam(R^\Box) \le c
\end{equation}
and, for each horizontal line $Z \subset \R^2$ intersecting $R^\Box$,
\begin{equation}
\label{normalized1}
\frac{1}{c} \le \ell(Z\cap R^\Box) \le c.
\end{equation}
\end{lem}

\begin{proof}
Let $\tilde Q\in \tilde \sQ_\Gamma$ be the rectangle satisfying $R \subset h_\Gamma(\tilde Q)$. Since $\tilde Q$ is a Whitney square, we have that $\dist(\tilde Q, \Gamma) \sim \dist(\tilde R, \Gamma)$.

Let $Z$ be a horizontal line intersecting $R^\Box$ and $z \in \Gamma$ for which $R^\Box \cap Z = \phi_R^2(\phi_R^1(R \cap \Gamma^\perp(z)))$. Then, by the choice of $C_\Box$, we have that
\[
\ell(Z\cap R^\Box)\sim \frac{\ell(R\cap \Gamma^\perp(z))}{\dist(R,\Gamma)^{p-1} \dist(\tilde R, \Gamma)^{2-p}} \sim 1.
\]
This proves \eqref{normalized1}.

To estimate the diameter of $R^\Box$, let $\pr_j \colon \R^2 \to \R^2$ be the coordinate projection $(x_1,x_2) \mapsto x_j$. By the previous remark, we have that
\[
\diam \pr_2(R^\Box) \sim 1.
\]
Towards the second estimate, we note that $\diam \pr_1(R^\Box) = \diam (\phi_R^2(\phi_R^1(R\cap \horz(Q))$, where $Q\in \sQ_\Gamma$ is the parent of $R$. Thus, by Lemmas \ref{upper boundary diam} and \ref{lemma:Qi1}, we have that
\begin{align*}
\diam \pr_1(R^\Box)) &\sim \frac{\diam \phi_R^1(R\cap \horz(Q))}{\dist(R,\Gamma)} \\
&\sim \frac{\diam R\cap \horz(Q)}{\dist(R,\Gamma)} 
\sim \frac{\diam(S_\Gamma(R))}{\dist(R,\Gamma)} \sim  1.
\end{align*}
The claim follows by combining these projection estimates.
\end{proof}

Lemma \ref{lemma:good partition} and Corollary \ref{cor:finite_shadow_neighbors} yield the following corollary for normalized rectangles.

\begin{cor}
\label{cor:last_battle}
Let $\Gamma \in \Jordan(p,C_0)$. Then there exist constants $c_0=c_0(p,C_0)$ and $\epsilon_0=\epsilon_0(p,C_0)$ such that, for every $R\in \sQ^\refi_\Gamma$ and $R'\in\sQ^\refi_\Gamma$ satisfying $R'\cap R\neq \emptyset$, we have that
\begin{enumerate}
\item $\dist(x_{\phi_R(R')},\,\partial({\phi_R(R')}) )\ge c_0$, \label{item:blahblah1}
\item for any $0<\epsilon<\epsilon_0$, the $\epsilon$-neighborhood 
\[
N_{\epsilon}(R^\Box):=\left\{ z\in {\mathbb C} \colon \dist(z,\,R^\Box)\le \epsilon \right\}
\]
of $R^\Box$ intersects at most $C=C(p,\,C_{GM})$ sets $\phi_R(R')$, and
\begin{equation}\label{not a lot}
\dist\left(N_{\ez}(R^\Box),\,x_{\phi_R(R')}\right)\ge \frac {9} {10} c_0. 
\end{equation}
\end{enumerate}
\end{cor}

\begin{proof}
By Remark \ref{rmk:good partition},
we have that intersecting rectangles have comparable diameter. Therefore $\phi_R(R')$ has size comparable to $R^\Box$. Thus $\phi_R(R')$ satisfies \eqref{normalized1} with another constant depending only on $p$ and $C_0$. Thus there exists $c_0=c_0(p,C_0)$ for which \eqref{item:blahblah1} holds.

Moreover, by choosing $\epsilon_0= \epsilon_0(p,\,C_{GM})$ small enough we have, by Corollary~\ref{cor:finite_shadow_neighbors}, that there are at most $C=C(p,\,C_0)$ sets  $\phi_R(R')$, $R'\in \sQ^\refi_\Gamma$ intersecting $N_{\ez}(R^\Box)$ with \eqref{not a lot}. This gives us the choice of $\epsilon_0$. 
\end{proof}

\subsection{Proof of Proposition \ref{prop:inside lip}}

For each $R\in \sQ^\refi_\Gamma$, let $\horz(R^\Box)$ and $\horz^\flat(R^\Box)$ be the upper and lower horizontal edge of $R^\Box$ respectively, with the convention that horizontal lines from the real axis meet the upper edge first. More formally, let 
\[
\horz(R^\Box) = \phi_R(R\cap \horz(Q))
\]
and 
\[
\horz^\flat(R^\Box) = \phi_R(R\cap \horz^\flat(Q)),
\]
where $Q\in \sQ_\Gamma$ is the parent of $R$. Let also $\horz(R) = R\cap \horz(Q)$ and $\horz^\flat(R) = R\cap \horz^\flat(Q))$.  

Let now $R\in \sQ_\Gamma^\refi$, $Q\in \sQ_\Gamma$ be the parent of $R$, and $\tilde Q\in \tilde \sQ_\Gamma$ be the parent of $R$ in $\tilde \Omega_\Gamma$, that is, $R\subset Q = h_\Gamma(\tilde Q)$. Let also $\tilde R\in \tilde \sQ_\Gamma^\refi$ be the rectangle in $\tilde Q$ corresponding to $R$, that is, $\iota_\Gamma^\refi(\tilde R) = R$.

Recall that there exists $L=L(p,C_0)$ and a bilipschitz map $\mu_{\tilde Q}:\tilde Q\to [0,\diam(\tilde Q)]^2$ preserving corners. Thus, by post-composing this map with a similarity, we may fix a bilipschitz map $\phi_{\tilde R} \colon \tilde Q\to [0,1]^2$, with constant $L \diam(\tilde Q)$, which maps corners of $\tilde Q$ to the corners of $[0,1]^2$.
Define $I_R=\phi_{\tilde R}(\horz(\tilde R))$ and $I^\flat_R=\phi_{\tilde R}(\horz^\flat(\tilde R))$ for $\tilde R\subset \tilde Q$. Note that  
\begin{equation}\label{bounded length}
\ell(\phi_{\tilde R}(\horz(\tilde R)))\sim \ell(\phi_{\tilde R}(\horz^\flat(\tilde R)))\sim \dist(\tilde R,\,\Gamma)^{\frac {p-r}{r-1}}\dist(R,\,\Gamma)^{\frac {r-p}{r-1}}\ls 1,
\end{equation}
where constants depend only on $p$, $C_0$, and $r$.

In these terms, the arcs $\horz(R^\Box)$ and $\horz^\flat(R^\Box)$ are graphs of continuous functions $u_R = \phi_R\circ h_\Gamma\circ \phi^{-1}_{\tilde R}|_{I_R} \colon I_R \to \R$ and $u^\flat_R = \phi_R\circ h_\Gamma\circ \phi^{-1}_{\tilde R}|_{I^\flat_R} \colon I^\flat_R \to \R$, respectively.

We claim that the family $\{u^\flat_R\}_R$ is uniformly equicontinuous. Here we say that a collection $\mathcal{F}$ of functions (or maps) is \emph{uniformly equicontinuous} if for each $\varepsilon>0$ there exists $\delta>0$ having the property that, for each $f\in \mathcal{F}$ and each pair of points $x$ and $y$ in the domain of $f$ of distance at most $\delta$, we have $|f(x)-f(y)|<\varepsilon$.

Since $h_\Gamma \colon \mathrm{cl}(\tilde A_\Gamma) \to \mathrm{cl}(A_\Gamma)$ is uniformly continuous and maps the horizontal edges of (topological) rectangles in $\tilde \sQ_\Gamma^\refi$ to those of (topological) rectangles in  $\sQ_\Gamma^\refi,$ the lower horizontal edges $\horz^\flat(R)$ are images of equicontinuous arcs $\horz^\flat(\tilde R) \to \horz^\flat(R)$, where $\tilde R$ and $R$ satisfy $\iota_\Gamma^\refi(\tilde R) = R$. Since the maps $\phi_R^1$ are bilipschitz with a uniform constant, this property holds also for the horizontal edges $\phi_R^1(\horz^\flat(R))$. Recall now that 
\[
u^\flat_R=\phi_R\circ h_\Gamma\circ \phi^{-1}_{\tilde R}|_{I^\flat_R}.
\]
Since both $\phi^{-1}_{\tilde R}$ and $\phi_R^2$ are linear, 
the graph of $u^\flat_R$ is a dilated image of $\phi_R^1(\horz^\flat(R))$ with the dilation factor comparable to 
\[
\dist(\tilde R,\,\Gamma)^{\frac {r-p}{r-1}}\dist(R,\,\Gamma)^{\frac {p-r}{r-1}},
\]
which is uniformly bounded away from zero  by \eqref{bounded length}. This proves the uniform equicontinuity of the family $\{u^\flat_R\}_R$. An analogous argument shows that $\{u_R\}_R$ is uniformly equicontinuous. This completes the proof of the claim.

For each $R\in \sQ_\Gamma^\refi$, we denote by ${\mathscr E}^\flat(R)$ the collection of end points of arcs in the collection
\[
\{ \sigma \in \Sigma^\flat(Q') \colon \sigma \subset \horz^\flat(R)\},
\]
 where $Q'\in \sQ_\Gamma$ is the unique element satisfying $Q=(Q')^{+}$ or $Q=(Q')^{-}$ for $Q\supset R$. 
By Corollary~\ref{cor:finite_shadow_neighbors}, the number of points in ${\mathscr E}(R)$ has an upper bound depending only on $p$ and $C_0$.

Let $\epsilon_0=\epsilon_0(p,C_0)$ as in Corollary \ref{cor:last_battle} and let $0< \epsilon <\epsilon_0$. Since the family $\{u_R\}_R$ is equicontinuous and Lipschitz functions are dense in the space of continuous functions, by the uniform  finiteness of the number of points in ${\mathscr E}^\flat(R)$, there exists a constant $L_\epsilon\ge 1$ and, for each $R\in \sQ_\Gamma^\refi$, $L_\epsilon$-Lipschitz functions $v^\flat_R\colon I^\flat_R \to \R$ satisfying the following properties: 
\begin{enumerate}
\item $\phi_R(a)=u^\flat_R(\phi_{\tilde R}^{-1}\circ h^{-1}_\Gamma(a))=v^\flat_R(\phi_{\tilde R}^{-1}\circ h^{-1}_\Gamma(a))$, \ \text{for any $a\in {\mathscr E}^\flat(R)$},
\item $||u^\flat_R - v^\flat_R||_\infty < \epsilon$; 
\end{enumerate}
recall that $u_R=\phi_R\circ h\circ \phi^{-1}_{\tilde R}|_{I^\flat_R}$. 
Observe  that, functions $u_R$ and $u_{R'}^\flat$ agree on $I_R\cap I_{R'}$ if $R\cap R'\ne \emptyset$.
Then we can apply a similar argument to all $R', R' \cap R\neq \emptyset$ under the same map $ \phi_R\circ h_\Gamma\circ \phi^{-1}_{\tilde R}$ to obtain $v^\flat_{R'\cap R}$.

Now, for each $R\in \sQ_\Gamma^\refi$, the graph of the function $v^\flat_R$ and the graphs of  $v^\flat_{R'\cap R}, R' \cap R\neq \emptyset$ together with vertical edges of $R^\Box$ enclose a topological rectangle $V^\Box_R$. Since $\diam R^\Box \sim 1$, we conclude that $V^\Box_R$ is uniformly bilipschitz to the unit disk $\D$. Let $V_R = \phi_R^{-1}(V^\Box_R)$. Note that, by $\phi_R^1$, $V_R$ is bilipschitz equivalent to $\Rect(R)$, quantitatively.

To obtain the partition $\sW_\Gamma^\refi$, we obtain the collection $\{V_R\}_R$ of bilipschitz rectangles as follows. Let $R\in \sQ_\Gamma^\refi$. If $R$ meets the boundary of $A_\Gamma$, we set $W_R = V_R$. 

Suppose now that $R$ does not meet $\partial A_\Gamma.$ Then its lower horizontal edge $\horz^\flat(R)$ meets rectangles $R'\in \sQ_\Gamma^\refi$ for which we have already defined the rectangle $W_{R'}$ using $V_{R'}$. Let $\mathcal U_R$ be this (finite) family of rectangles $R'\in \sQ_\Gamma^\refi$. 

Recall that the functions $u_R$ and $u_{R'}^\flat$ agree on $I_R\cap I_{R'}$ if $R\cap R'\ne \emptyset$. Now, for $R'\in \mathcal U_R$, the union $V_R \cup V_{R'}$ contains a neighborhood of $\mathrm{int}(R\cap R')$, where the interior is taken with respect to $\horz^\flat(R)$. We may now take $W_R = V_R$ since in the construction of $v_R$ we fixed all the vertices of $R'$ intersecting $\horz(R)$ whenever $R\cap R'\neq \emptyset$. Note that the rectangle $V_R$ is still bilipschitz equivalent to $\Rect(R)$, quantitatively. 

The collection $\sW_\Gamma^\refi = \{ W_R \colon R\in \sQ_\Gamma^\refi\}$ is now a partition of the closure of a domain $\hat A_\Gamma$, which is also a collar of $\Gamma$ in $\Omega_\Gamma$. 
 
The conditions of $\sW_\Gamma^\refi$ apart from \eqref{item:need-to-mention} are clearly satisfied. The last condition \eqref{item:need-to-mention} is satisfied, if in the beginning of the construction we divide $\epsilon$ by a constant depending on $p$ and $C_0$. Indeed, since $\dist_H(V_R^\Box, R^\Box)<\varepsilon$, the estimate follows, up to a multiplicative constant, from the facts that $\phi_R^1$ is Lipschitz with an absolute constant and the inverse of $\phi_R^2$ scales with a factor smaller than $\dist(\tilde R,\Gamma)$ in the second direction; notice that $\dist(R,\Gamma)\gs \dist(\tilde R, \Gamma)$.

This concludes the proof of Proposition \ref{prop:inside lip}. \qed


\section{Proof of Theorem \ref{weak main thm}}
\label{sec:proof-end-game}

The proof of Theorem \ref{weak main thm} is in the spirit of the proofs of theorems \cite[Theorem A]{T1980} and \cite[Theorem 3.2]{TV1982}. In our case, heuristically, the bijection $\eta^\refi_\Gamma \colon \tilde \sQ^\refi_\Gamma \to \sW^\refi_\Gamma$ gives the large scale, or combinatorial, blueprint for the reflection to be constructed. This large scale information is combined with the local bilipschitz information on topological rectangles $\tilde R$
 and $W_{\tilde R} = \eta^\refi_\Gamma(\tilde R)$ for $R\in \sQ_\Gamma^\refi$.

\newcommand{\sN}{\mathsf{N}}

Before heading to the final proof, let us see the heuristic geometric idea behind it. 
Recall that, by Corollary~\ref{cor:wz Qi}, each $\tilde R\in \tilde \sQ_{\Gamma}^\refi$ is bilipschitz equivalent to 
\[
\Rect_\Gamma(\tilde R) =[0,\,\dist(\tilde R,\,\Gamma)^{\frac {p-1}{r-1}}\dist(R,\,\Gamma)^{\frac {r-p}{r-1}}]\times [0,\,\dist(\tilde R,\,\Gamma)].
\]
Moreover, by Proposition \ref{prop:inside lip}, 
the corresponding
 $W_{\tilde R}$ is bilipschitz equivalent to 
\[
\Rect_\Gamma(R)= [0,\,\dist(R,\,\Gamma)]  \times [0,\, \dist(R,\,\Gamma)^{p-1}  \dist(\tilde R,\,\Gamma)^{2-p}],
\]
respectively, where $R =\iota_{\Gamma}^{\refi}(\tilde R) \in \mathscr {\tilde Q_{\Gamma}^{\refi}}$.

Now the linear map $A_{\tilde R}\colon  \Rect_\Gamma(\tilde R) \to \Rect_\Gamma(\iota_\Gamma^\refi(\tilde R))$, which maps the horizontal edges of $\Rect_\Gamma(\tilde R)$ to the horizontal edges of $\Rect_\Gamma(\tilde R)$, satisfies
\[
||A_{\tilde R}||^r\ls J_{f_{\tilde \Omega}}, 
\]
since $\dist(R,\Gamma) \gs \dist(\tilde R, \Gamma)$.

To obtain the map $\tilde \Omega_\Gamma \to \Omega_\Gamma$, we first construct, 
using this information, a $p$-reflection $\tilde A_\Gamma \to \hat A_\Gamma$ and then extend it over the remaining domains.

\subsection{A version of Tukia's lemma}

\newcommand{\norm}[1]{\lVert #1 \lVert}

We use the following simple version of Tukia's lemma \cite[Theorem A]{T1980}. For the statement, we define, for $a\ge 1$ and $1<p<2$, the modified sup-norm $\norm{\cdot}_{a,p}\colon \R^2 \to [0,\infty)$ by setting
\[
\norm{(x_1,x_2)}_{a,p} = \max\{ a |x_1|, a^{p-1}|x_2|\}
\]
for $(x_1,x_2)\in \R^2$.

\begin{lem}\label{Tukia}
Let $S$  be the unit square $[0,\,1]\times [0,\,1]$, and $R$ be the rectangle $[0,\,a]\times [0,\,a^{p-1}]$ for some $a\ge 1$ and $1<p<2$. Given a constant $L\ge 1$, consider an orientation preserving $L$-bilipschitz homeomorphism $f\colon (\partial S, \norm{\cdot}_{a,p}) \to (\partial R, |\cdot|)$, which maps the edges of $S$ to those of $R$. Then there exists a homeomorphism $F\colon S\to R$ extending $f$ so that 
\begin{align}\label{F}
|DF|^p\le C(L,\,p)J_F
\end{align}
almost everywhere. 
\end{lem}

\begin{proof}
We may assume that  the components of $f$ on each of the horizontal edges of $S$ are increasing. 
For every $(x,\,y)\in S$, let $I_{x}$ be the line segment joining $f(x,\,0)$ and $f(x,\,1)$ and let 
$J_{y}$ be the line segment  joining $f(0,\,y)$ and $f(1,\,y)$. We define $F(x,\,y)$ to be the intersection point $I_{x}\cap J_{y}$. 

Since $f$ is a homeomorphism, the point $F(x,\,y)$ is well defined. Since $f$ maps edges of $S$ to those of $R$, we further have that $F|_{\partial S}=f$. 
We claim that $F$ is the desired map.

We first show that $F$ is a homeomorphism. Since $f$ is continuous, we have that $F$ is continuous. Since $F(\partial S) = \partial R$, we have, by the standard degree theory, that $F$ is surjective. To show injectivity, suppose $F(x_1,\,y_1)=F(x_2,\,y_2)$ with $x_1\ne x_2$. Then the segment $I_{x_1}$ intersects the segment $I_{x_2}$, which contradicts the monotonicity of $f$ on horizontal edges of $S$. The injectivity of $F$ now follows by symmetry. Thus $F$ is a homeomorphism by compactness of $S$.

To obtain \eqref{F}, we first estimate  $|F(x+\ez,\,y)-F(x,\,y)|$ for a small $\ez$.
Write $f=(f^1,\,f^2)$ and $F=(F^1,\,F^2)$ and let $(x,y),(x+\epsilon,y)\in S$.

Since $F(x+\ez,\,y)$ and $F(x,\,y)$ are on the line joining $f(0,\,y)$ and $f(1,\,y)$, we have, by considering similar triangles, that
\begin{equation}\label{tangent}
\frac {|F^2(x+\ez,\,y)-F^2(x,\,y)|}{|F^1(x+\ez,\,y)-F^1(x,\,y)|}=\frac {|f^2(1,\,y)-f^2(0,\,y)|}{a}\le a^{p-2} (L-1/L). 
\end{equation}
On the other hand, since the segments  $I_{x+\ez}$ and $I_{x}$ together with the horizontal edges of $S$ yield a trapezoid, we have that
\begin{align*}
&|F^1(x+\ez,\,y)-F^1(x,\,y)| \\
&\quad \le \max\{|f(x+\ez,\,0)-f(x,\,0)|,\, |f(x+\ez,\,1)-f(x,\,1)|\} \\
&\quad \le L \max\{ a|x+\epsilon - x|, a|x+\epsilon -x|\} = La\epsilon.
\end{align*}
and 
\begin{align*}
&|F^1(x+\ez,\,y)-F^1(x,\,y)| \\
&\quad \ge \min\{|f(x+\ez,\,0)-f(x,\,0)|,\, |f(x+\ez,\,1)-f(x,\,1)|\} \\
&\quad \ge \frac{1}{L} \min\{ a|x+\epsilon -x|, a|x+\epsilon -x|\} = \frac{a}{L} \ez.
\end{align*}
To conclude, we have that
\begin{equation}\label{horizontal difference}
\frac{a}{L}\ez\le |F^1(x+\ez,\,y)-F^1(x,\,y)|\le La\ez. 
\end{equation}
Combining this estimate with \eqref{tangent}, we conclude that
\begin{equation}\label{vertical  difference}
|F^2(x+\ez,\,y)-F^2(x,\,y)|\le C(L)a^{p-1}\ez \le C(L)a\ez,
\end{equation}
since $a\ge 1$ and $1<p<2$.
Similar estimates hold for $|F(x,\,y+\ez)-F(x,\,y)|$. Thus $F$ is a Lipschitz map with Lipschitz constant $C(L)a\epsilon$ and
\[
|DF(x,\,y)|^p\le C_1(L,\,p) a^p. 
\]
It remains to prove an estimate for the Jacobian of $F$. 

Let $(x,y)$, $(x+\epsilon_1,y)$, $(x,y+\epsilon_2)$, $(x+\epsilon_1,y+\epsilon_2)\in S$ and let $Q_{xy} \subset R$ be the quadrilateral with corners $F(x,\,y)$, $F(x+\ez_1,\,y)$, $F(x,\,y+\ez_2)$ and $F(x+\ez_1,\,y+\ez_2)$.
 
Using again the associated trapezoids, we observe that the rectangle $Q_{xy}$ contains a Euclidean rectangle with side lengths 
\[
l_x = \min\{ |f(x+\epsilon_1,0)-f(x,0)|, |f(x+\epsilon_1,1)-f(x,1)|\}\ge \frac{a}{L}\epsilon_1
\]
and 
\[
l_y = \min\{ |f(0,y+\epsilon_2)-f(0,y)|, |f(1,y+\epsilon_1)-f(1,y)|\} \ge \frac{a^{p-1}}{L}\epsilon_2.
\] 
Thus the area of $Q_{xy}$ is bounded from below by the product $l_x l_y$. Thus the area of $Q_{xy}$ is bounded from below by $Ca^p \epsilon_1 \epsilon_2$, where $C>0$ is a constant depending only on $p$ and $L$. Hence
\[
|J_F(x,\,y)|^p\ge c_1(L,p) a^p. 
\]
The claim follows.  
\end{proof}

\begin{rem}\label{rem tukia}
The proof of Lemma~\ref{Tukia} above also gives a lower bound for $|DF|^p$, i.e.\ one also has
$$\frac 1{C(L,\,p)}J_F\le |DF|^p$$
since  $f$ is not only $L$-Lipschitz but  $L$-bilipschitz. 
This can be easily deduced from \eqref{horizontal difference} and \eqref{vertical  difference}, together with an upper bound for the Jacobian
$$|J_F(x,\,y)|^p\ge C'(L,p) a^p$$
via a calculation similar to what we made in the proof. 
Therefore one can in fact improve \eqref{F} to 
$$\frac 1{C(L,\,p)}J_F\le |DF|^p\le C(L,\,p) J_F.$$
\end{rem}

Recall that ${\mathscr E}^\flat(R)$ is the collection of end points in $\horz^\flat(R)$ of 
\[
\sigma \in  \Sigma( Q'),\  \sigma\subset \horz^\flat(R)
\]
 for each  $R\in \sQ_\Gamma^\refi$, where $Q'\in \sQ_\Gamma$ is the unique element so that $Q=(Q')^{+}$ or $Q=(Q')^{-}$ and $R\subset Q$. 
Also define ${\mathscr E}(R)$ to be  the collection of end points in $\horz(R)$ of 
\[
\sigma \in  \Sigma^{\pm}(Q),\  \sigma\subset \horz(R)
\]
for each  $R\in \sQ_\Gamma^\refi$, where $R\subset Q$. We define  $\tilde {\mathscr E}^\flat(\tilde R)$ and $\tilde {\mathscr E}(\tilde R)$ in a similar manner. Notice that there is a natural correspondence between the elements in ${\mathscr E}(R)$ and ${\mathscr E}^\flat(R)$, and those in $\tilde{\mathscr E}(\tilde R)$ and $\tilde{\mathscr E}^\flat(\tilde R)$,  respectively.

To apply Lemma \ref{Tukia}, observe that for each $\tilde R\in \tilde \sQ_\Gamma^\refi$ there exists a natural map between the boundaries of the Euclidean rectangles $\Rect(\tilde R)$ and $\Rect(\iota_\Gamma^\refi(\tilde R))$. 
Let $\rho_{\tilde R} \colon \tilde R \to \Rect_\Gamma(\tilde R)$ and $\rho_{W_{\tilde R}} \colon W_{\tilde R} \to \Rect_\Gamma(R)$ be bilipschitz maps preserving horizontal and vertical edges as in Lemma \ref{lemma:biLip} and in Proposition \ref{prop:inside lip}.
We establish the following lemma.

\begin{lem}\label{bilip boundary map}
Let $\tilde R\in \tilde \sQ^\refi_\Gamma$. Then there exists a map
\[
f^\diamond_R\colon \partial\Rect_\Gamma(\tilde R) \to \partial\Rect_\Gamma(\iota^\refi_{\Gamma}(\tilde R))
\]
which is $L$-bilipschitz equivalent to the linear map 
\begin{equation}
\label{eq:rescaling}
\begin{split}
(x,\,y)\mapsto &\left(\dist(\tilde R,\Gamma)^{\frac {1-p}{r-1}}\dist(\iota_\Gamma^\refi(\tilde R),\Gamma)^{\frac {p-1}{r-1}} x\right.,\,\\
& \qquad \quad \left.  \dist(\tilde R ,\Gamma)^{1-p}\dist(\iota^\refi_{\Gamma}(\tilde R),\Gamma)^{p-1}  y\right). 
\end{split}
\end{equation}
Moreover,  $f^\diamond_R$ sends corners to corners and the composition 
$$f|_{\tilde R} = \rho_{W_{\tilde R}}^{-1} \circ F_{\tilde R}^\diamond \circ \rho_{\tilde R}\colon \partial \tilde R\to \partial W_{\tilde R}$$
maps elements in $\tilde{\mathscr E}(\tilde R)$ and $\tilde{\mathscr E}^\flat(\tilde R)$ to the corresponding ones in ${\mathscr E}(R)$ and ${\mathscr E}^\flat(R)$. 
\end{lem}
\begin{proof}
Recall that $\tilde R$ and $W_{\tilde R}$ are uniformly bilipschitz to Euclidean rectangles. Then by the construction of the combinatorial partition, especially \eqref{subdivision tilde sigma}, together with  Proposition \ref{prop:inside lip}, we conclude the claim via the bilipschitz property of  $\rho_{\tilde R}$ and  $\rho_{W_{\tilde R}}$.
\end{proof}

We record a consequence of Lemma \ref{Tukia} in these terms.

\begin{cor}\label{tukia}
Let $\tilde R\in \tilde \sQ^\refi_\Gamma$, and let
\[
f^\diamond_R\colon \partial\Rect_\Gamma(\tilde R) \to \partial\Rect_\Gamma(\iota^\refi_{\Gamma}(\tilde R))
\]
be the map in Lemma~\ref{bilip boundary map}.
Then there exists a homeomorphism $F^\diamond_R\colon  \Rect_\Gamma(\tilde R) \to \Rect_\Gamma(\iota_\Gamma^\refi(\tilde R))$ extending $f^\diamond_R$ and satisfying
\[
|DF^\diamond_R|^r\le C(L,\,p)J_{F^\diamond_R}
\]
almost everywhere in $\Rect_\Gamma(\tilde R)$. 
\end{cor}

\begin{proof}
Since $\dist(\iota_\Gamma^\refi(\tilde R),\Gamma) \gs \dist(\tilde R, \Gamma)$, we may take 
\[
a\sim \dist(\iota^\refi_{\Gamma}(R),\Gamma)^{\frac {1-p}{r-1}}\dist(R,\Gamma)^{\frac {p-1}{r-1}}
\]
and rescale the map \eqref{eq:rescaling}, if necessary, by a constant depending only on $p$ and $C_0$.  
Then take $p = r$ in Lemma \ref{Tukia}.
\end{proof}

\begin{rem}\label{rem final}
As mentioned in Remark~\ref{rem tukia}, we can also improve the conclusion of this corollary to 
$$\frac1 {C(L,\,p)}J_{F^\diamond_R}\le |DF^\diamond_R|^r\le C(L,\,p)J_{F^\diamond_R}. $$
\end{rem}

\subsection{Proof of Theorem \ref{weak main thm}}

We define an $r$-reflection $f \colon \tilde \Omega_\Gamma \to \Omega_\Gamma$ as follows.

Let $\tilde R\in \sQ_\Gamma^\refi$, $R=\iota_\Gamma^\refi(\tilde R)$, and $W_{\tilde R} = \eta_\Gamma^\refi(\tilde R)$. Let also $\rho_{\tilde R} \colon \tilde R \to \Rect_\Gamma(\tilde R)$ and $\rho_{W_{\tilde R}} \colon W_{\tilde R} \to \Rect_\Gamma(R)$ be bilipschitz maps preserving horizontal and vertical edges as in Lemma \ref{lemma:biLip} and in Proposition \ref{prop:inside lip}.

Let  $f|_{\partial \tilde R} \colon \partial \tilde R\to \partial W_{\tilde R}$ be the bilipschitz map given by Lemma~\ref{bilip boundary map}. Notice that the vertical edges of $W_{\tilde R}$ are hyperbolic segments and that $W_{\tilde R}$ has well-defined upper and lower horizontal edges. Let now 
\[
f_{\tilde R}^\diamond = \rho_{W_{\tilde R}} \circ f|_{\partial \tilde R} \circ (\rho_{\tilde R}|_{\partial \tilde R})^{-1} \colon \partial \Rect_\Gamma(\tilde R) \to \partial \Rect_\Gamma(R).
\]
By Proposition \ref{prop:inside lip}, $f_{\tilde R}^\diamond$ satisfies the assumptions of Corollary \ref{tukia} with $L=L(p,C_0,r)$. Let $F_{\tilde R}^\diamond \colon \Rect_\Gamma(\tilde R) \to \Rect_\Gamma(R)$ be the extension of 
$f_{\tilde R}^\diamond$ as in Corollary \ref{tukia} and define
\[
F|_{\tilde R} = \rho_{W_{\tilde R}}^{-1} \circ F_{\tilde R}^\diamond \circ \rho_{\tilde R} \colon \tilde R \to W_{\tilde R}.
\]
Since $\rho_{\tilde R}$ and $\rho_{W_{\tilde R}}$ are bilipschitz maps with uniform constants, we obtain that 
\[
|DF|^r \le C(p,C_0,r) J_F
\]
almost everywhere in $\tilde R$ and \emph{a fortiori} almost everywhere in $\tilde A_\Gamma$.

It remains to extend $F$ as a reflection $\tilde \Omega_\Gamma \setminus \tilde A_\Gamma \to \Omega_\Gamma \setminus \hat A_\Gamma$. Since the boundary component of $\hat A_\Gamma$ contained in $\Omega_\Gamma$ is a Lipschitz Jordan curve, the extension follows from a theorem of Tukia \cite[Theorem A]{T1980}. The proof is complete. \qed

\section{Proofs of Theorem~\ref{main},  Corollary~\ref{dual extension}, and Corollary~\ref{p quasidisk}}

\begin{proof}[Proof of Theorem~\ref{main}]
First, the implication from (2) to (1) is the content of 
Theorem~\ref{thm:main1}. 
Next, (1) is equivalent to  (3) since the inverse of a $p$-morphism is 
a $q$-morphism with $q=p/(p-1)$ by  e.g.\ \cite[Theorem 3]{V2008}  and vice versa.

 Finally, (3) 
implies (2) by the characterization of Sobolev extension domains from 
\cite{S2010}.
Indeed, first, the inner composition of any Lipschitz function in 
$L^{1,\,q}(\tilde\Omega)$ with a 
$q$-morphism $f$ belongs to $L^{1,\,q}(\Omega)$ with desired norm control. 
Next, since $q>2,$ our $q$-morphism $f$ is locally uniformly Lipschitz
by \cite{G1971}. By joining given $x,y\in\Omega$ via a line segment $I$, 
picking the first and last points on $I\cap \Gamma$ if this intersection is 
non-empty and using
the fact that $f$ is the identity on $\Gamma$, we conclude that $f$ is 
uniformly Lipschitz. Thus the inner 
composition gives us a desired extension operator for Lipschitz functions. Recall that $C^{\infty}(\overline G)$ is dense in $W^{1,\,p}(G)$ for $1<p<\infty$ if
$G$ is a planar Jordan domain, see \cite{lewis1987}. 
This guarantees that the above procedure gives a Sobolev extension 
with the desired norm control. Thus the inner composition generates a 
Sobolev extension 
operator, which allows us to conclude (2) via \cite{S2010}. 
\end{proof}

\begin{proof}[Proof of Corollary~\ref{dual extension}]
The case $p=2$ is known by \cite{VG1975}. When $p\neq 2$, it is a direct consequence of \cite{S2010, KRZ2015} that  $\Gamma$ is the boundary of a John domain. 
Since the inner composition of any Lipschitz function in 
$L^{1,\,p}(\Omega)$ with a 
$p$-morphism $f:\tilde \Omega\to \Omega$ belongs to $L^{1,\,p}(\tilde \Omega)$ with desired norm control and
$f$ is identity on $\Omega,$
by \cite{JS2000} this composition operator gives us a desired extension 
operator for Lipschitz functions. Then by the fact that $C^{\infty}(\overline{\Omega})$ is dense in $W^{1,\,p}(\Omega)$ for $1<p<\infty$ if
$\Omega$ is a planar Jordan domain  \cite{lewis1987}, we conclude the corollary from Theorem~\ref{main}. 
\end{proof}

\begin{proof}[Proof of Corollary~\ref{p quasidisk}]
To begin, by Theorem~\ref{main}, we obtain 
a $p$-reflection $\phi_1$ from $\tilde \Omega$ to 
$\Omega.$ Next,  \cite[Theorem 1.4]{B2002} gives us a locally uniformly 
Lipschitz continuous quasiconformal mapping $\phi_2\colon \tilde \Omega\to \mathbb D.$ By the Carath\'eodory-Osgood theorem $\phi_2$ extends homeomorphically
up to the boundary and we may also extend $\phi_2$ to the point at infinity. We refer also to this extension by $\phi_2.$ Since $q=p/(p-1)>2,$ the locally uniform Lipschitz continuity together with $K$-quasiconformality
of $\phi_2$ for some $K$ yields that
$$|D\phi_2(x)|^q\le |D\phi_2(x)|^{q-2}|D\phi_2(x)|^2\le M^{q-2}KJ_{\phi_2}(x)$$
for almost every $x\in \tilde \Omega,$ where $M$ is the uniform local Lipschitz
constant of $\phi_2.$ Thus $\phi_2$ is a $q$-morphism and hence $\phi_2^{-1}$  is
a $p$-morphism; see e.g.\ \cite[Theorem 3]{V2008}. Pick a bilipschitz reflection $\phi_3$ with respect to 
$\partial \mathbb D$ in $\widehat \C$. 
Then $\Phi:=\phi_1\circ \phi_2^{-1}\circ \phi_3$ maps $\widehat \C\setminus
\overline {\mathbb D}$ homeomorphically (up to boundary) onto $\tilde \Omega.$
Next, $\Phi$ is a $p$-morphism as a composition of
$p$-morphisms, see e.g. \cite[Theorem 1]{VU1998} and the references therein.
We additionally define $\Phi=\phi_2^{-1}$ in $\mathbb D.$ Then $\Phi$ is a homeomorphism of $\hat {\mathbb C}$ and it is easy to check that $\Phi$ is a $p$-morphism. 
\end{proof}

\end{document}